\documentclass{amsart}
\usepackage{mathtools, hyperref,
mathrsfs, xfrac, amssymb, enumerate, xspace, xcolor, relsize}
\usepackage[latin1]{inputenc}
\usepackage{tikz-cd}
\usepackage{todonotes}

\numberwithin{equation}{section}

\numberwithin{subsection}{section}

\allowdisplaybreaks[1]

\newenvironment{enumeratea}
{\begin{enumerate}[\upshape (a)]}
{\end{enumerate}}

\newenvironment{enumeratei}
{\begin{enumerate}[\upshape (i)]}
{\end{enumerate}}

\newtheorem{theorem}{Theorem}[section]
\newtheorem{proposition}[theorem]{Proposition}
\newtheorem{proposition-definition}[theorem]
{Proposition-Definition}
\newtheorem{corollary}[theorem]{Corollary}
\newtheorem{lemma}[theorem]{Lemma}

\theoremstyle{definition}
\newtheorem{definition}[theorem]{Definition}
\newtheorem{notation}[theorem]{Notation}
\newtheorem{example}[theorem]{Example}

\newtheorem{remark}[theorem]{Remark}

\theoremstyle{remark}

\newcommand\nome{testing}
\newcommand\call[1]{\label{#1}\renewcommand\nome{#1}}
\newcommand\itemref[1]{\item\label{\nome;#1}}
\newcommand\refall[2]{\ref{#1}(\ref{#1;#2})}
\newcommand\refpart[2]{(\ref{#1;#2})}

\newcommand{\proofofpart}[2]{\smallskip\noindent\emph{Proof of {\rm\refpart{#1}{#2}}}.}

\newcommand{\step}[1]{\smallskip\emph{Step #1.}}

\newcommand\cA{\mathcal{A}} \newcommand\cB{\mathcal{B}}
\newcommand\cC{\mathcal{C}} \newcommand\cD{\mathcal{D}}
 \newcommand\cF{\mathcal{F}}
\newcommand\cG{\mathcal{G}}

\newcommand\cM{\mathcal{M}} 
\newcommand\cO{\mathcal{O}} 
 \newcommand\cR{\mathcal{R}}
\newcommand\cS{\mathcal{S}} 
 
 \newcommand\cX{\mathcal{X}}
\newcommand\cY{\mathcal{Y}} 

\renewcommand\AA{\mathbb{A}} 
\newcommand\CC{\mathbb{C}} 
 
\newcommand\GG{\mathbb{G}}

 \newcommand\NN{\mathbb{N}}
 
\newcommand\QQ{\mathbb{Q}} \newcommand\RR{\mathbb{R}}
\renewcommand\SS{\mathbb{S}}

 \newcommand\ZZ{\mathbb{Z}}

\newcommand\rmm{\mathrm{m}} 
 \newcommand\rmp{\mathrm{p}}

 \newcommand\bfx{\mathbf{x}}

\renewcommand{\cong}{\simeq}

\newcommand\arr{\ifinner\to\else\longrightarrow\fi}

\newcommand\arrto{\ifinner\mapsto\else\longmapsto\fi}

\newcommand{\xarr}{\xrightarrow}

\renewcommand\H{\operatorname{H}}

\newcommand\op{^{\mathrm{op}}}

\newcommand{\eqdef}{\mathrel{\smash{\overset{\mathrm{\scriptscriptstyle def}} =}}}

\renewcommand\th{^\text{th}}

\def\displaytimes_#1{\mathrel{\mathop{\times}\limits_{#1}}}

\def\displayotimes_#1{\mathrel{\mathop{\bigotimes}\limits_{#1}}}

\renewcommand\hom{\operatorname{Hom}}

\newcommand\pic{\operatorname{Pic}}

\newcommand\spec{\operatorname{Spec}}

\newcommand\id{\mathrm{id}}

\newcommand\pr{\operatorname{pr}}

\newcommand\indlim{\varinjlim}

\renewcommand\projlim{\varprojlim}

\newcommand{\cat}[1]{(\mathrm{#1})}

\newcommand\double{\mathbin{\rightrightarrows}}

\newcommand{\underhom}
{\mathop{\underline{\mathrm{Hom}}}\nolimits}

\newlength{\ignora}
\newcommand{\hsmash}[1]{\settowidth{\ignora}{#1}#1\hspace{-\ignora}}

\renewcommand{\setminus}{\smallsetminus}

\newcommand{\rest}[1]{|_{#1}}

\renewcommand\projlim{\varprojlim}

\newcommand{\mmu}{\boldsymbol{\mu}}

\newcommand{\gm}[1][\relax]{\GG_{\rmm #1}}

\newcommand{\colim}{\operatorname{colim}}

\newcommand\radice[2]{\hspace{-1.5pt}\sqrt[\uproot{2}#1]{#2}}

\newcommand{\abs}[1]{\left|#1\right|}

\DeclareFontFamily{U}{mathx}{\hyphenchar\font45}
\DeclareFontShape{U}{mathx}{m}{n}{
      <5> <6> <7> <8> <9> <10>
      <10.95> <12> <14.4> <17.28> <20.74> <24.88>
      mathx10
      }{}
\DeclareSymbolFont{mathx}{U}{mathx}{m}{n}
\DeclareFontSubstitution{U}{mathx}{m}{n}
\DeclareMathAccent{\widecheck}{0}{mathx}{"71}
\DeclareMathAccent{\wideparen}{0}{mathx}{"75}

\newcommand{\infroot}[1]{\radice{\infty}{#1}}
\newcommand{\infquot}[1]{#1/\infty}
\newcommand{\gr}{^{\mathrm{gp}}}
\newcommand{\grq}{\gr_{\QQ}}

\newcommand{\catmon}{(\mathrm{ComMon})}

\newcommand{\et}{_\text{\rm\'et}}

\renewcommand{\div}{\operatorname{Div}}

\newcommand{\Div}{\div}

\renewcommand{\mod}{\operatorname{Mod}}

\newcommand{\wt}{^{\mathrm{wt}}}
\newcommand{\p}{\rmp}
\newcommand{\sz}[1]{\spec\ZZ[#1]}

\newcommand{\fs}{fine saturated\xspace}
\newcommand{\df}{Deligne--\hspace{0pt}Faltings\xspace}
\newcommand{\irs}{infinite root stack\xspace}
\newcommand{\irss}{infinite root stacks\xspace}
\newcommand{\sm}{symmetric monoidal\xspace}
\newcommand{\fp}{finitely presented\xspace}

\newcommand{\red}{_{\mathrm{red}}}

\newcommand{\aff}{\cat{Aff}}
\newcommand{\cataff}[1]{(\mathrm{Aff}/#1)}
\newcommand{\catring}{\cat{Ring}}
\newcommand{\catab}{\cat{Ab}}
\newcommand{\qcoh}{\cat{QCoh}}

\newcommand{\FP}{\operatorname{FP}}
\newcommand{\catqcoh}{\operatorname{QCoh}}
\newcommand{\catmod}{\operatorname{Mod}}
\newcommand{\ob}{\operatorname{Ob}}

\newcommand{\qc}{quasi-coherent\xspace}
\newcommand{\catpar}{\operatorname{Par}}

\newcommand{\fslogsch}{\cat{FSLogSch}}
\newcommand{\rootstack}{\cat{RootStack}}

\newcommand{\fppf}{_{\mathrm{fppf}}}
\newcommand{\sfppf}{_{\mathrm{s.fppf}}}

\newcommand{\sh}{\operatorname{Sh}}

\newcommand{\sch}{_{\mathrm{sch}}}
\newcommand{\rep}{_{\mathrm{rep}}}

\newcommand{\ket}{_{\text{\rm K\'{e}t}}}
\newcommand{\kfl}{_{\mathrm{Kfl}}}

\DeclareMathOperator{\fsa}{fs}

\newcommand{\Res}{\operatorname{Res}}
\newcommand{\Ind}{\operatorname{Ind}}

\renewcommand{\mathcal}{\mathscr}

\begin{document}

\title[Infinite root stacks]{Infinite root stacks and quasi-coherent sheaves\\on logarithmic schemes}

\author{Mattia Talpo}

\author{Angelo Vistoli}

\address[Talpo]{Department of Mathematics\\
Simon Fraser University\\
8888 University Drive\\
Burnaby BC\\
V5A 1S6 Canada, and Pacific Institute for the Mathematical Sciences\\ 4176-2207 Main Mall \\ Vancouver BC\\ V6T 1Z4 Canada}

\email{mtalpo@sfu.ca}

\address[Vistoli]{Scuola Normale Superiore\\Piazza dei Cavalieri 7\\
56126 Pisa\\ Italy}

\email{angelo.vistoli@sns.it}


\thanks{Both authors were supported in part by the PRIN project ``Geometria delle varietà algebriche e dei loro spazi di moduli'' from MIUR, and by research funds from the Scuola Normale Superiore}


\begin{abstract}
We define and study infinite root stacks of fine and saturated logarithmic schemes, a limit version of the root stacks introduced by Niels Borne and the second author in \cite{borne-vistoli1}. We show in particular that the infinite root stack determines the logarithmic structure, and recovers the Kummer-flat topos of the logarithmic scheme. We also extend the correspondence between parabolic sheaves and quasi-coherent sheaves on root stacks to this new setting.
\end{abstract}

\maketitle

\tableofcontents

\section{Introduction}

\subsubsection*{Logarithmic geometry}
Logarithmic structures were introduced in the late '80s, in the work of Fontaine, Illusie, Deligne and Faltings, and studied systematically starting with the work of K. Kato \cite{kato}. The motivating philosophy was that, sometimes, a degenerate object behaves like a smooth one, when equipped with the correct logarithmic structure. 

For a quick introduction to the theory one can profitably consult \cite{olssonetal} and references therein. The flavor of the idea is the following. A logarithmic stucture on a scheme $X$ is a sheaf of commutative monoids $M_{X}$ on the small étale site $X\et$ of $X$, with a homomorphism of sheaves $\alpha\colon M_{X} \arr \cO_{X}$ (here the structure sheaf is considered as a sheaf of monoids via multiplication), such that the induced homomorphism $\alpha^{-1}\cO_{X}^{\times} \arr \cO^{\times}_X$ is an isomorphism. We think of $\alpha$ as an exponential map; a section $s$ of $M_{X}$ is thought of as a logarithm of $\alpha(s)$, defined ``up to period''.

Suppose one has a proper morphism $f\colon X \arr S$, where $S$ is a Dedekind scheme, with a closed point $s_{0} \in S$. In a geometric context $S$ would be a smooth curve, while in an arithmetic context it could be the spectrum of a discrete valuation ring, typically in mixed characteristic. Assume that $f$ is smooth outside of $s_{0}$, while the fiber $X_{0}$ over $s_{0}$ is a divisor with normal crossings in $X$. Then one can put logarithmic structures $M_{X}$ and $M_{S}$ on $X$ and $S$, so that the morphism $X \arr S$ extends to a morphism of logarithmic schemes $(X, M_{X}) \arr (X, M_{S})$, which is smooth, in the sense of logarithmic geometry. 

Now, logarithmic schemes have well-behaved cohomology theories: de Rham, étale, crystalline, Hodge, and so on (see for example \cite{kato-nakayama}). This makes logarithmic geometry a powerful tool for studying degenerations of cohomology theories, and cohomologies of varieties with bad reduction over valued fields.

Logarithmic geometry is intimately connected with toric geometry and toroidal embeddings (smooth logarithmic varieties are locally modeled on toric varieties). Also, it has been successfully applied to moduli theory: adding a logarithmic structure often isolates the ``main component'' of a moduli space, i.e. the closure of the locus of smooth objects (or its normalization); a few examples are \cite{katof, olsson3, olsson1, olsson2} (see also \cite[Sections 4 and 10]{olssonetal}).

Recently, logarithmic geometry was also employed to define Gromov-Witten invariants for singular targets \cite{gross-siebert-GW, abramovichgw, chen} and to formulate a version of mirror symmetry, in the form of the ``Gross--Siebert program'' \cite{gross-siebert, gross-siebert-II, gross-siebert-III}.

\subsubsection*{Geometric incarnations of logarithmic schemes}
Suppose that $X$ is a fine and saturated logarithmic scheme, in the sense of Kato. 
There have been at least two attempts to produce a topological space, or a Grothendieck topology, that captures the ``logarithmic geometry'' of $X$.

The first is for schemes of finite type over $\CC$ (or more generally log-analytic spaces), the so called Kato--Nakayama space $X^{\log}$, introduced in \cite{kato-nakayama} (see also \cite{ogus}). In simple cases, this is obtained as a ``real oriented blowup'' of $X$, along the locus where the log structure is concentrated.

In a different direction, one would like to have analogues of the étale and fppf site for a logarithmic scheme. There are natural notions of étale and flat map between logarithmic schemes; however, they turn out to be too inclusive, and yield sites that are much too large. {For example, the log geometric version of blowups are \'etale in the logarithmic sense, and although the topologies obtained by allowing such maps to be coverings have been used for some purposes (see for example \cite{niziol-k-theory}), it is often better to restrict the allowed coverings to maps that do not change the geometry of the space in such a drastic way}.

The correct notions are those of Kummer-étale and Kummer-flat maps, introduced by Kato. They allow to associate with $X$ two sites $X\ket$ and $X\kfl$, called respectively the Kummer-étale and the Kummer-flat site; the first works well in characteristic~$0$ and for studying $l$-adic cohomologies, while the second is more suited for positive characteristic. Coherent sheaves on  Kato's Kummer flat site $X\kfl$ have been used to define the $K$-theory of $X$ in \cite{hagihara,niziol-k-theory}.

These two constructions are related: for example, if $X$ is of finite type over $\CC$ and $F$ is a constructible sheaf in the Kummer-étale topology, the cohomology of $F$ equals the cohomology of its pullback to $X^{\log}$ (\cite[Theorem~0.2(1)]{kato-nakayama}).

\subsubsection*{The infinite root stack}
In this paper we construct an algebraic version of the Kato--Nakayama space, the \irs $\infroot{X}$ of $X$, which is a proalgebraic stack over $X$. For this we build on the construction of the root stacks $\radice{B}X$ in \cite{borne-vistoli1} (denoted by $X_{B/\overline{M}_{X}}$ there), which in turn is based on several particular cases constructed by Olsson \cite{matsuki-olsson,olsson-log-twisted}. Denote by $\alpha_{X}\colon M_{X} \arr \cO_{X}$ the logarithmic structure on $X$, {a fine saturated logarithmic scheme}, and set $\overline{M}_{X} \eqdef M_{X}/\cO^{*}_{X}$; then $\overline{M}_{X}$ is a sheaf of monoids on the small étale site $X\et$, whose geometric fibers are sharp \fs monoids (these are the monoids that appear in toric geometry, consisting of the integral points in a strictly convex rational polyhedral cone in some $\RR^{n}$). A Kummer extension $B$ of $\overline{M}_{X}$ is, roughly speaking, a sheaf of monoids containing $\overline{M}_{X}$, such that every section of $B$ has locally a positive multiple in $\overline{M}_{X}$ (see \cite[Definition~4.2]{borne-vistoli1}); the typical example is the sheaf $\frac{1}{d}\overline{M}_{X}$ of fractions of sections of $\overline{M}_{X}$ with some fixed denominator $d$.
Loosely speaking, $\radice{B}X$ parametrizes extensions $\beta\colon N \arr \cO_X$ of $\alpha_{X}\colon M_{X} \arr \cO_X$ with $\overline{N}= B$. When the logarithmic structure is generated by a single effective Cartier divisor $D \subseteq X$, so that $\overline{M}$ is the constant sheaf $\NN_{D}$ on $D$, and we take $B$ to be $\frac{1}{d}\NN_{D}$, then $\radice{B}{X}$ is the root stack $\radice{d}{(X,D)}$ introduced in \cite{dan-tom-angelo2008,cadman}.

The \irs $\infroot{X}$ can be thought of as the limit of the root stacks $\radice{B}X$ as the sheaf $B$ becomes larger. It is not an algebraic stack, as its diagonal is not of finite type. If $p$ is a geometric point of $X$ and $r$ is the rank of the group $\overline{M}_{X,p}\gr$, then the reduced fiber of $\infroot{X}$ over $p$ is the classifying stack $\cB\mmu_{\infty}^{r}$, where $\mmu_{\infty} \eqdef \projlim \mmu_{n}$, where $\mmu_{n}$ is the group scheme of $n\th$ roots of $1$. Over $\CC$ of course $\mmu_{\infty}$ is isomorphic to the profinite completion $\widehat{\ZZ}$; here we can see the similarity with the Kato--Nakayama space, in which the fiber over a point $p$ as above is $\SS^{r}$, which is the classifying space $B{\ZZ}^{r}$ of the group ${\ZZ}^{r}$.

The \irs of a \fs logarithmic scheme is very large, it is not even an algebraic stack, and {it} may look discouragingly complicated. However, it has a very explicit local description (Proposition~\ref{prop:local-model}); and the fact that it is locally a quotient by the action of a diagonalizable group scheme, albeit not of finite type, makes it very amenable to study. For example, one can apply to it results of Toën--Riemann--Roch type \cite{toen-rr,vezzosi-vistoli02} to get very explicit formulas for its K-theory (we plan to go back on this point in a future paper).

\smallskip

Our main results are as follows.

\subsubsection*{Reconstruction results}
First of all, we show how the logarithmic structure can be reconstructed from the \irs (Corollary~\ref{cor:isomorphisms}). In fact, we construct a faithful functor from the category of \fs logarithmic schemes into the $2$-category of proalgebraic stacks. This is not fully faithful; however, it induces a bijection on isomorphisms. In particular if $X$ and $Y$ are \fs logarithmic schemes, and if $\infroot{X}$ is equivalent to $\infroot{Y}$ as fibered categories, then $X$ and $Y$ are isomorphic as logarithmic schemes. We describe the essential image of this functor, and characterize the morphisms of proalgebraic stacks that come from morphisms of \fs logarithmic schemes (Theorem~\ref{thm:equivalence2}).

In principle one could develop the theory of logarithmic structures entirely from the point of view of \irss (not that we advocate doing this); a more interesting point is that in this way we are enlarging the category of \fs logarithmic schemes, and the morphisms of \irss that do not come from morphisms of logarithmic schemes (Example \ref{ex:not.full}) could have interesting applications.

Here is an example. Suppose that $X \arr S$ is a morphism to a Dedekind scheme of the type mentioned above {(either a smooth curve, or the spectrum of a discrete valuation ring)}. Then the special fiber $X_{0}$ over the point $s_{0}$ inherits a logarithmic structure $M_{X_{0}}$; however, $(X_{0}, M_{X_{0}})$ is not logarithmically smooth in the absolute sense, {but} it is only log smooth over the point $s_{0}$ equipped with the logarithmic structure induced by $M_{S}$. {Moreover, the logarithmic central fiber $(X_{0}, M_{X_{0}})$} has a nontrivial logarithmic structure even at the points where $X_{0}$ is smooth {in the classical sense}. By passing to the induced morphism of \irss $\infroot X \arr \infroot S$, the embedding of the point $s_{0} \subseteq S$ lifts to a morphism of stacks $s_{0} \arr \infroot S$, which does not come from a morphism of logarithmic structures. Then one can take the pullback $s_{0}\times_{\infroot S} \infroot X$; this is a pro-algebraic stack that can be described very explicitly, but is not an \irs. It has some advantages over $(X_{0}, M_{X_{0}})$; for example, the projection $s_{0}\times_{\infroot S} \infroot X \arr X_{0}$ is an isomorphism where $X_{0}$ is smooth. It would be interesting to investigate  to what extent $s_{0}\times_{\infroot S} \infroot X \arr X_{0}$ can be used as a substitute for $(X_{0}, M_{X_{0}})$ in problems such as degenerations of Gromov--Witten invariants.

We plan to go back to this in a later paper. This ``central fiber'' of the infinite root stack, and much of the basic formalism developed in the present article, are crucial in the recent work \cite{logmckay} about a logarithmic version of the derived McKay correspondence, of S. Scherotzke, N. Sibilla and the first author.

\subsubsection*{The connection with the Kummer-étale and Kummer-flat topologies}
The \irs is closely related to the Kummer-étale and Kummer-flat topologies.

For example, we show that a locally finitely presented morphism  of \fs logarithmic schemes is Kummer-flat (respectively Kummer-étale) if and only if the corresponding morphism of \irss is representable, flat and finitely presented (respectively representable and étale) (Theorems \ref{thm:flat<->Kummer-flat} and \ref{thm:etale<->Kummer-etale}).

We also explain how to recover the Kummer-flat topos of $X$ from $\infroot{X}$. We define the small fppf site $\infroot{X}\fppf$ of $\infroot{X}$, whose objects are representable, finitely presented flat maps $\cA \arr \infroot{X}$; substituting flat with étale we obtain the definition of the small étale site $\infroot{X}\et$. Sending a Kummer-flat (respectively Kummer-étale) map $Y \arr X$ into the induced morphism $\infroot{Y} \arr \infroot{X}$ defines a morphism of sites $X\kfl \arr\infroot{X}\fppf$ (respectively $X\ket \arr\infroot{X}\et$). We are not able to show that these are equivalences; however we show that they induce equivalences of topoi (Theorems \ref{equiv.topoi} and \ref{thm:equivalence-ket}).


\subsubsection*{Quasi-coherent sheaves}
Another major theme of this paper is \qc sheaves on \fs logarithmic schemes. We argue that these should be defined as \qc sheaves on the \irs. If $\overline{M}_{X} \subseteq B$ is a fine Kummer extension, then one finds in \cite{borne-vistoli1} a definition of \qc parabolic sheaf that generalizes those given in many particular cases in \cite{metha-seshadri,maruyama-yokogawa,biswas,iyer-simpson,borne}; the main result is \cite[Theorem 6.1]{borne-vistoli1}: the category of parabolic sheaves with coefficients in $B$ is equivalent as an abelian
category to the category of \qc sheaves on the root stack $\radice{B}X$. The definition of a parabolic sheaf extends immediately to the present case, giving a definition of a parabolic sheaf with arbitrary rational weights on a \fs logarithmic scheme. The proof of \cite[Theorem 6.1]{borne-vistoli1} extends to the present setup, yielding an equivalence of the category of parabolic sheaves on $X$ and that of \qc sheaves on $\infroot{X}$ (Theorem~\ref{BV.rational}).

This point of view is exploited in \cite{parabolic.moduli} to construct moduli spaces of parabolic sheaves with arbitrary rational weights. Furthermore, the correspondence was recently extended, in the complex analytic case, to parabolic sheaves with real weights and sheaves of modules on the Kato--Nakayama space \cite{real.parabolic}.

We are aware of another possible definition of a quasi-coherent sheaf on a (derived) logarithmic scheme, introduced by Steffen Sagave, Timo Schürg and Gabriele Vezzosi in \cite{sagave-schurg-vezzosi}. We are not sure whether this is connected with ours.

We conclude by characterizing \fp sheaves on $\infroot{X}$ in purely parabolic terms (Theorem~\ref{thm:char-fp-parabolic}). By putting this together with Corollary~\ref{thm:equivalence2}, we obtain a parabolic interpretation of \fp sheaves on the Kummer-flat site of $X$.

There is an issue with the present construction: even when $X$ is the spectrum of a field, the structure sheaf of $\infroot{X}$ is in general not coherent (Example~\ref{ex:non-coherent}); see \cite{niziol-k-theory} for a discussion of this issue for the Kummer-flat site. This does not happen when the logarithmic structure is simplicial, in the sense that the geometric stalks of $\overline{M}_{X}$ are isomorphic to the monoid of integral elements of a simplicial rational polyhedral cone. This means that, even when $X$ is noetherian, we can't expect \fp parabolic sheaves on $X$ to form an abelian category, in general. We don't know whether this is inevitable, or it can be fixed with a different construction.

\subsubsection*{The connection with the Kato--Nakayama space}
The relation between the \irs and the Kato--Nakayama space of a fine saturated logarithmic scheme $X$ locally of finite type over $\CC$ is more than an analogy. This has been clarified in later work.

In fact, if $X$ is a fine saturated logarithmic scheme, locally of finite type over $\CC$, there exists a map from the the Kato--Nakayama space $X^{\log}$ of $X$ to the topological stack $(\hspace{-1pt}\infroot X)_{\rm top}$ associated with $\infroot X$, that induces an equivalence of profinite homotopy types \cite{knvsroot}. This map, constructed in \cite{knvsroot} by gluing together locally defined maps, has in fact a natural interpretation, due to the fact that $X^{\log}$ can be defined as a ``transcendental infinite root stack'', in which instead of adding roots of all orders, one adds logarithms \cite{TVnew}.

\subsection*{Description of content}

Section~\ref{sec:preliminaries} contains a list of conventions and preliminary definitions and results that will be used in the rest of the paper; in \ref{subsec:projlim} we review the notion of projective limit of fibered categories that will be used, and \ref{sec:log-schemes} contains a review of the nonstandard point of view on logarithmic structures introduced in \cite{borne-vistoli1}, which is the one most suitable for the purposes of this paper. In particular the notion of parabolic bundles is stated in this language; and the construction of the logarithmic structure starting from an \irs, which plays a fundamental role, is most naturally carried out with this formalism.

The definition of the \irs $\infroot{X}$ of a \fs logarithmic scheme $X$ appears in Section~\ref{sec:irss-log-schemes}; we also show that $\infroot{X}$ is a limit of finite root stacks $\radice{n}X$. The next Section~\ref{sec:local-models} contains a local description of \irss (Proposition~\ref{prop:local-model}), which is fundamental for the rest of the paper. In \ref{sec:abstract-irs} we take a somewhat different point of view, and define an abstract \irs over a scheme, by assuming the local description of Proposition~\ref{prop:local-model} as an axiom. This gives the correct framework for describing how to recover the logarithmic structure of $X$ from $\infroot{X}$.

Section~\ref{realsec:qc} contains a discussion of \qc sheaves on an infinite root stack. We begin by considering in \ref{sec:qc} the case of \qc sheaves on a more general fibered category. In one definition given for example in \cite{rosenberg-kontsevich}, they are cartesian functors into the category of modules. On the other hand, \qc sheaves on a ringed site are defined as sheaves of $\cO$-modules that locally have a presentation. Both points of view are useful to us, so we show that for an \irs, or, more generally, for a fibered category with an fpqc cover by a scheme the categories of \qc sheaves in the first sense is equivalent to the category of \qc sheaves on three different associated sites with the fpqc topology (Proposition~\ref{prop:equivalence-qc}). Subsection~\ref{sec:qc-on-irs} specializes the discussion to \irss and is dedicated to the proof of some technical results on \qc sheaves on them.

The heart of the paper is Section~\ref{sec:reconstruction}, in which we show how to produce a logarithmic scheme from an \irs; this cost us fairly intense suffering. It requires some very technical work on \fs monoids, and an analysis of the Picard group of \irss over algebraically closed fields. This work pays off in \ref{sec:morphisms-irss}, where we show (Theorem~\ref{thm:equivalence2}) that the category of \fs logarithmic schemes is equivalent to the category of abstract \irss; the morphisms of \irss are defined as base-preserving functors satisfying a technical condition.

Section~\ref{realsec:fp-sheaves-fppf} is about the comparison between the Kummer-flat site of a logarithmic scheme and the fppf site of its \irs. In ~\ref{sec:fp-sheaves-fppf} we introduce the small fppf site of an \irs, and we show that the categories on  \fp sheaves on the root stack and on this fppf site coincide. The next subsection contains our results (Theorems \ref{thm:flat<->Kummer-flat} and \ref{equiv.topoi}, and Corollary~\ref{cor:equiv-fp}) on the connection between the \irs and the Kummer-flat site of a \fp logarithmic scheme. 

Finally, Section~\ref{sec:parabolic} contains the parabolic description of \qc and \fp sheaves on an \irs.

\subsection*{Acknowledgments}
This paper builds on approximately half of the PhD thesis of the first author, which was developed under the supervision of the second author at the Scuola Normale Superiore, in Pisa. During the latest phase of writing, the first author was supported by the Max Planck Institute for Mathematics of Bonn.

We would like to thank Johan De Jong for a very useful conversation. Furthermore, the second author would like to express his debt toward his collaborator Niels Borne; this paper builds on ideas in their previous joint paper, and owes much to discussions that they have had over the years.

Finally, we are grateful to the anonymous referee for a thorough reading, a long list of useful comments, and, overall, an outstanding job. In particular, he or she suggested several simplifications and corrections.


\section{Preliminaries}\label{sec:preliminaries}
 
In this preliminary section we fix some notations, recall the notion of projective and inductive limits of fibered categories and the formalism of \df structures in logarithmic geometry.
 
\subsection{Notations and conventions}
 
All monoids will be commutative. Often they will be fine, saturated and sharp, in which case they are also torsion-free (the geometric stalks of the sheaf $\overline{M}_{X}$ of a fine saturated logarithmic scheme $X$ are of this type). If $P$ is a monoid, $X_{P}$ will always denote the spectrum $\sz P$ of the monoid algebra $\ZZ[P]$, which will also usually be equipped with the natural logarithmic structure. For $p \in P$, the corresponding element of $\ZZ[P]$ will be denoted by $x^{p}$, and the same notation will be used for $k[P]$ where $k$ is a field.

All symmetric monoidal functors will be strong, meaning that the morphisms $FX\otimes FY\to F(X\otimes Y)$ will be isomorphisms.

All fibered categories will be fibered over the category $\aff$ of affine schemes, which can also be seen as the dual of the category of commutative rings. If $\cX$ is a fibered category over $\aff$, we will always denote by $\p_{\cX}\colon \cX \arr \aff$ the structure functor.

We will identify as usual the fibered category $\cataff X \arr \aff$ of maps into a scheme $X$ with $X$ itself; thus, a fibered category over $X$ will be a fibered category over $\aff$ with a cartesian functor to $\cataff X$, or, equivalently, a fibered category over $\cataff X$.

We will usually refer to categories fibered in groupoids over $\aff$ as ``stacks''. The fibered categories we will deal with are indeed stacks in the fpqc topology, but this fact will not play a major role in our treatment. The point is that ``stack'' is short, and saying, for example, ``open substack'' is more convenient that ``open fibered subcategory''.

A morphism of categories fibered in groupoids on $\aff$ will be called ``representable'' if it is represented by schemes. Algebraic spaces will not play any role in this paper.

If $X = \spec A$ is an affine scheme over a ring $R$, and $G$ is an affine group scheme over $R$ acting on $X$, we denote by $X/G$ the spectrum of the ring of invariants $A^{G}$.
 
We will use $*$ to denote the Godement product (or ``horizontal composite'') of natural transformations, see for example \cite[II.5]{maclane}.

If $X$ is a scheme, we denote by $X\et$ the small étale site of $X$, whose objects are étale maps $U \arr X$, where $U$ is an affine scheme. Of course one could extend this to all schemes étale over $X$, but the resulting topos would be equivalent.

We will use the following notation: if $X$ is a logarithmic scheme, we denote the logarithmic structure by $\alpha_{X}\colon M_{X} \arr \cO_{X}$, and the corresponding \df structure by $(A_{X}, L_{X})$ (see \ref{sec:log-schemes}). Also, since sometimes we will have to distinguish between fibered products in the category of \fs logarithmic schemes, and fibered products of the underlying schemes, which do not coincide, we will adopt a standard notation and denote by $\underline X$ the underlying scheme to a logarithmic scheme $X$.

If $\cC$ is a site, we denote by $\sh\cC$ the topos of sheaves of sets on $\cC$.

The symbol $\spadesuit$ will denote the end of a proof or the absence of one.

\subsection{Projective and inductive limits of fibered categories}\label{subsec:projlim}

Let $\cC$ be a category. Suppose that $I$ is a filtered partially ordered set, considered as a category. A projective system $(\cM_{i}, F_{ij})$of categories fibered over $\cC$ consists of a strict $2$-functor from $I\op$ to the $2$-category of categories fibered over $\cC$. Concretely, for each $i \in I$ we have a fibered category $\cM_{i} \arr \cC$, with a cartesian functor $F_{ij}\colon \cM_{j} \arr \cM_{i}$, such that $F_{ij}F_{jk} = F_{ik}$ for every triple $(i, j, k)$ such that $i \leq j \leq k$ and $F_{ii}=\id_{\cM_{i}}$ for all $i$ (where equality should be interpreted as strict equality).

\begin{remark}
Of course to give a definition of projective limit  one should take a lax $2$-functor, or, equivalently, assume that $I$ is a filtered $2$-category, i.e., a $2$-category that is equivalent to a filtered partially ordered set. However, for the purposes of this paper the present context is sufficient.
\end{remark}

We define the projective limit $\projlim_{i}(\cM_{i}, F_{ij})$ as in \cite[Definition~3.5]{borne-vistoli2}.

\begin{definition}
An object $(T, \{\xi_{i}\}, \{\phi_{ij}\})$ of $\projlim_{i}\cM_{i} = \projlim_{i}(\cM_{i}, F_{ij})$ consists of the following data.

\begin{enumeratea}

\item An object $T$ of $\cC$, and an object $\xi_{i}$ of $\cC(T)$ for all $i \in I$.

\item For each pair $(i, j)$ with $i \leq j$, an isomorphism $\phi_{ij}\colon F_{ij}\xi_{j} \arr \xi_{i}$ in $\cM_{i}(T)$.

\end{enumeratea}
These are required to satisfy the following condition: if $i \leq j \leq k$, then 
   \[
   \phi_{ik} = \phi_{ij}\circ(F_{ij}\phi_{jk})\colon F_{ik}\xi_{k} \arr \xi_{i}\,.
   \]

An arrow $(\phi, \{f_{i}\})\colon (T', \{\xi'_{i}\}, \{\phi'_{ij}\}) \arr (T, \{\xi_{i}\}, \{\phi_{ij}\})$ consists of an arrow  $\phi\colon T' \arr T$ in $\cC$ and an arrow $f_{i}\colon \xi'_{i} \arr \xi_{i}$ in $\cM_{i}$ for each $i$, satisfying the following conditions.

\begin{enumeratei}

\item The image of $f_{i}$ in $\cC$ is $\phi$ for all $i$.

\item For each pair $(i, j)$ with $i \leq j$, the diagram
   \[
   \begin{tikzcd}[column sep = large]
   F_{ij}\xi'_{j} \ar{r}{F_{ij}f_{j}}\ar{d}{\phi'_{ij}} & F_{ij}\xi_{j}\ar{d}{\phi_{ij}}\\
   \xi'_{i} \ar{r}{f_{i}} & \xi_{i}
   \end{tikzcd}
   \]
commutes.
\end{enumeratei}

\end{definition}

It is easily seen that $\projlim_{i}\cM_{i}$ is a fibered category over $\cC$. The functor $\projlim_{i}\cM_{i} \arr \cC$ is the obvious one, sending $(T, \{\xi_{i}\}, \{\phi_{ij}\})$ to $T$. An arrow $(\phi, \{f_{i}\})$ is cartesian if and only if each $f_{i}$ is cartesian. In particular, if each $\cM_{i}$ is fibered in groupoids, then $\projlim_{i}\cM_{i}$ is also fibered in groupoids.

\begin{remark}
As pointed out by the referee, one can also see a projective system of fibered categories over $\cC$ indexed by $I$ as a fibered category over $I\times \cC$. The limit can then be seen as the pushforward of this fibered category to the category $\cC$.
\end{remark}

We are also going to use filtered inductive limits, in a slightly more general situation. Suppose that $I$ is a filtered partially ordered set as above. A lax inductive system $(\cC_{i}, F_{ij})$ of categories is a lax $2$-functor from $I$ to the category of categories. We define the colimit $\indlim_{I}(\cC_{i}, F_{ij})$ in the obvious way: the class of objects is the disjoint union $\bigsqcup_{I}\ob\cC_{i}$, while if $\xi_{i} \in \ob\cC_{i}$ and $\xi_{j}\in \ob\cC_{j}$ we define $\hom(\xi_{i}, \xi_{j})$ as the colimit $\indlim_{k}\hom_{\cC_{k}}(F_{ki}\xi_{i}, F_{kj}\xi_{j})$ over the set $\{k\in I \mid k \geq i, \ k \geq j\}$.


\subsection{Logarithmic schemes, Deligne-Faltings structures and charts}\label{sec:log-schemes}

In this section we outline the theory of logarithmic schemes, using the nonstandard approach of \cite{borne-vistoli1}.

Assume that $X$ is a scheme, and denote by $\Div_{X\et}$ the fibered category over $X{\et}$ consisting of pairs $(L,s)$ where $L$ is an invertible sheaf and $s$ is a global section. 

\begin{definition}
A \emph{Deligne-Faltings structure} on $X$ is a symmetric monoidal functor $L\colon A\to \Div_{X\et}$ with trivial kernel, where $A$ is a sheaf of monoids on the small \'{e}tale site $X{\et}$.

A \emph{logarithmic scheme} is a scheme $X$ equipped with a \df structure.
\end{definition}

We refer the reader to Sections 2.4 and 2.5 of \cite{borne-vistoli1} for background about symmetric monoidal categories and functors, and for a full discussion of the above definition. 

The phrase ``with trivial kernel''  means that for $U\to X$ \'{e}tale, the only section of $A(U)$ with image isomorphic to the object $(\cO_{U},1)$ is the zero section.

We will denote a logarithmic scheme by $(X,A,L)$ or just $X$, when the \df structure is understood.

Recall that the standard definition of a logarithmic scheme (\cite{kato}) is that of a scheme $X$ with a sheaf of monoids $M$ on $X{\et}$, with a morphism $\alpha\colon M\to \cO_{X}$ (where $\cO_{X}$ is a sheaf of monoids with the multiplication) such that the restriction of $\alpha$ to $\alpha^{-1}(\cO_{X}^{\times})$ induces an isomorphism $\alpha|_{\alpha^{-1}(\cO_{X}^{\times})}\colon \alpha^{-1}(\cO_{X}^{\times})\to \cO_{X}^{\times}$. A logarithmic scheme in this sense is \emph{quasi-integral} if the natural resulting action of $\cO_{X}^{\times}$ on $M$ is free.

The link between our definition and the standard notion of a quasi-integral logarithmic scheme is the following: given a morphism of sheaves of monoids $\alpha\colon M\to \cO_{X}$, one takes the stacky quotient by $\cO_{X}^{\times}$ to obtain  a symmetric monoidal functor $L\colon \overline{M}\to [\cO_{X}/\cO_{X}^{\times}]\cong \Div_{X\et}$, and sets $A=\overline{M}$. In other words, a section of $A$ is sent by $L$ to the dual $L_{a}$ of the invertible sheaf $N_a$, associated to the $\GG_{\rmm}$-torsor given by the fiber $M_{a}$ of $M\to \overline{M}=A$ over $a$ (meaning that $N_a$ is the sheaf of sections of the line bundle over $X$ associated to the $\GG_{\rmm}$-torsor $M_a$), and the restriction of $\alpha$ to $M_{a}\to \cO_{X}$ gives the section of $L_{a}$.

In the other direction, starting with a \df structure $L\colon A\to \Div_{X\et}$ we can take the fibered product $A\times_{\Div_{X\et}}\cO_{X}\to \cO_{X}$, and verify that $M=A\times_{\Div_{X\et}}\cO_{X}$ is equivalent to a sheaf.

\begin{remark}
The fact that we are taking a dual in the identification $[\cO_{X}/\cO_{X}^{\times}]\cong \Div_{X\et}$ might look unnatural. In fact, we could have worked equally well with the stack $\Div'_{X\et}$ of line bundles with a map to $\cO_X$, instead of a global section, and the natural isomorphism $[\cO_{X}/\cO_{X}^{\times}]\cong \Div'_{X\et}$ does not involve taking any duals.

This choice would for example have the following effect on the treatment of parabolic sheaves (see \cite{borne-vistoli1}, or Section \ref{sec:parabolic}): adding an element $p$ to the index of the sheaf would correspond to twisting by the line bundle $L_p^\vee$. If the log structure is induced by a single effective Cartier divisor $D$, this would make adding $1$ to the index correspond to twisting by $\cO_X(-D)$, which also looks unpleasant.

We prefer to perform this change of sign right away in the definition of a \df structure, and have translation by $1$ correspond to twisting by the line bundle $\cO_X(D)$.
\end{remark}

We will denote the logarithmic structure of $X$ by $\alpha_{X}\colon M_{X} \arr \cO_{X}$, and the corresponding \df structure by $(A_{X}, L_{X})$. When necessary we will denote by $\underline{X}$ the bare scheme underlying the logarithmic scheme $X$.

\begin{definition}
A morphism $(A,L)\to (B,N)$ of \df structures on a scheme $X$ is a pair $(f,\alpha)$ where $f\colon A\to B$ is a morphism of sheaves of monoids, and $\alpha\colon N\circ f \cong M$ is a monoidal isomorphism of symmetric monoidal functors $A\to \Div_{X\et}$.
\end{definition}

\begin{remark}
If $f\colon Y\to X$ is a morphism of schemes and $(A,L)$ is a \df structure on $X$, one defines a \emph{pullback} \df structure $f^{*}(A,L)=(f^{*}A,f^{*}L)$ (see \cite[Proposition 3.9]{borne-vistoli1}). The sheaf of monoids $f^{*}A$ is, like the notation suggests, the usual pullback as a sheaf of sets to $Y{\et}$, and if $g\colon Z\to Y$ is another morphism of schemes, there is a canonical isomorphism $g^{*}f^{*}(A,L)\cong (fg)^{*}(A,L)$.
\end{remark}

\begin{definition}
A morphism of logarithmic schemes $(X,A,L)\to (Y,B,N)$ is a pair $(f,f^{\flat})$, where $f\colon X\to Y$ is a morphism of schemes, and $f^{\flat}\colon f^{*}(B,N)\to (A,L)$ is a morphism of \df structures on $X$.
\end{definition}

Morphisms can be composed in the evident way, and logarithmic schemes form a category.

\begin{definition}
A morphism of logarithmic schemes $(f,f^{\flat})\colon (X,A,L)\to (Y,B,N)$ is \emph{strict} if $f^{\flat}$ is an isomorphism.
\end{definition}

Strict morphisms are morphisms that do not change the logarithmic structure.

We will be interested only in \df structures that arise from local models. Recall that a homomorphism of monoids $\phi\colon P\to Q$ is a \emph{cokernel} if the induced homomorphism $P/\phi^{-1}(0) \to Q$ is an isomorphism. A morphism $A\to B$ of sheaves of monoids on $X{\et}$ is a cokernel if the homomorphism induced on every stalk is a cokernel.

\begin{definition}\label{def:chart}

A \emph{chart} for a sheaf of monoids $A$ on $X{\et}$ is a homomorphism of monoids $P\to A(X)$ such that the induced map of sheaves $P_{X}\to A$ is a cokernel.

A sheaf of monoids $A$ on $X{\et}$ is \emph{coherent} if $A$ has charts with finitely generated monoids locally for the \'{e}tale topology of $X$.

A logarithmic scheme $(X,A,L)$ is \emph{coherent} if the sheaf $A$ is coherent.
\end{definition}

\begin{remark}
Equivalently, a \emph{chart} for a \df structure $(A,L)$ on $X$ can be seen as a symmetric monoidal functor $P\to \Div(X)$ for a monoid $P$, that induces the functor $L\colon A\to \Div_{X\et}$ (basically by ``sheafifying and trivializing the kernel'', see \cite[Proposition 3.3]{borne-vistoli1} for the precise construction).
\end{remark}

This differs from the standard notion of chart for a logarithmic scheme, which is a morphism of monoids $P\to \cO_{X}(X)$ that induces the logarithmic structure $\alpha\colon M\to \cO_{X}$ (see \cite[Section 1]{kato} for details). We will distinguish the two notions by calling the standard charts \emph{Kato charts}. Every Kato chart $P\to \cO_{X}(X)$ induces a chart by composing with $\cO_{X}(X)\to \Div(X)$.

Moreover one can show that having finitely generated charts \'{e}tale locally is equivalent to having finitely generated Kato charts \'{e}tale locally \cite[Proposition $3.28$]{borne-vistoli1}.

\begin{remark}
Note that for a monoid $P$, the scheme $\spec  \ZZ[P]$ has a natural \df structure induced by the composition $P\to \ZZ[P]\to \Div(\spec \ZZ[P])$. 

Giving a Kato chart on $X$ with monoid $P$ is the same as giving a strict morphism of logarithmic schemes $X\to \spec  \ZZ[P]$, and giving a chart is the same as giving a strict morphism $X\to [\spec  \ZZ[P]/\widehat{P}]$, where $\widehat{P}$ is the Cartier dual of $P\gr$, the action is the natural one, and the \df structure of the quotient stack is defined by descent from the one of $\spec  \ZZ[P]$.

From now on $\spec \ZZ[P]$ and $ [\spec \ZZ[P]/\widehat{P}]$ will be equipped with these \df structures without further mention.
\end{remark}

One can show that on a coherent logarithmic scheme, charts can be constructed by taking as monoid $P$ the stalk of the sheaf $A$ over a geometric point of $X$. More precisely, any geometric point ${x}$ of $X$ has an \'{e}tale neighborhood where we have a chart with monoid $A_{x}$ (\cite[Proposition 3.15]{borne-vistoli1}).

Using charts one can also describe the logarithmic part of morphisms between coherent logarithmic schemes by using homomorphisms of monoids.

\begin{definition}
A \emph{chart} for a morphism of sheaves of monoids $A\to B$ on $X{\et}$ is given by two charts $P\to A(X)$ and $Q\to B(X)$ for $A$ and $B$, together with a homomorphism of monoids $P\to Q$ making the diagram
   \[
   \begin{tikzcd}
   P\rar\dar & Q\dar\\
   A(X)\rar & B(X)
   \end{tikzcd}
   \]
commutative.

A \emph{chart} for a morphism of logarithmic schemes  $(f,f^{\flat})\colon (X,A,L)\to (Y,B,N)$ is a chart for the morphism of sheaves of monoids $f^{*}B\to A$ given by $f^{\flat}$.
\end{definition}

In other words, a chart for a morphism of logarithmic schemes $(f,f^{\flat})\colon (X,A,L)\to (Y,B,N)$ can be seen as two symmetric monoidal functors $P\to \Div(Y)$ and $Q\to \Div(X)$ that are charts for $(B,N)$ and $(A,L)$ respectively, and a morphism of monoids $P\to Q$ inducing $f^{\flat}\colon f^{*}(B,N)\to (A,L)$.

\begin{definition}
A \emph{Kato chart} for a morphism $(f,f^{\flat})\colon (X,A,L)\to (Y,B,N)$ of logarithmic schemes is a chart such that the functors $P\to \Div(Y)$ and $Q\to \Div(X)$ lift to $P\to \cO_{Y}(Y)$ and $Q\to \cO_{X}(X)$.
\end{definition}

Equivalently a Kato chart can be seen as a commutative diagram of logarithmic schemes
   \[
   \begin{tikzcd}
   (X,A,L)\rar\dar & \spec \ZZ[Q] \dar\\
   (Y,B,N)\rar & \spec \ZZ[P]
   \end{tikzcd}
   \]
with strict horizontal arrows, and analogously a chart can be seen as such a commutative diagram, with the quotient stacks $[\spec \ZZ[P]/\widehat{P}]$ and $[\spec \ZZ[Q] /\widehat{Q}]$ in place of  $\spec \ZZ[P]$ and $\spec \ZZ[Q]$ respectively.

One can show \cite[Proposition $3.17$]{borne-vistoli1} that it is always possible to find local charts for morphisms $(X,A,L)\to (Y,B,N)$ between coherent logarithmic schemes, and moreover one can take the monoids of the charts to be stalks of the sheaves $A$ and $B$.

We recall the definition of fine and saturated logarithmic schemes. These are the kind of logarithmic schemes we will be interested in.

Recall that a monoid $P$ is integral if the natural homomorphism $P\to P\gr$ is injective, or equivalently if $p+r=q+r$ in $P$ implies $p=q$. An integral monoid $P$ is saturated if whenever for $p \in P\gr$ we have a positive integer $n$ such that $np \in P$, then $p \in P$.

\begin{definition}
A logarithmic scheme $(X,A,L)$ is \emph{fine} if it is coherent and the stalks of $A$ are fine monoids (i.e. integral and finitely generated).

A logarithmic scheme $(X,A,L)$ is \emph{fine and saturated} if it is fine, and the stalks of $A$ are (fine and) saturated monoids.
\end{definition}

\begin{remark}
Equivalently one can check these conditions on charts, i.e. a logarithmic scheme $(X,A,L)$ is fine (resp. \fs) if and only if it is coherent and it admits local charts by fine (resp. fine and saturated) monoids.
\end{remark}

From now on all our logarithmic schemes will be assumed to be \fs.

\begin{example}[Log points]\label{ex:log-points}
Let $k$ be a field and $P$ a sharp monoid. Then we have a \df structure on $\spec k$ given by the functor $P\to \Div({\spec k})$ sending $0$ to $(k,1)$ and everything else to $(k,0)$. This is the \emph{logarithmic point} on $\spec k$ with monoid $P$. If $k$ is algebraically closed, one can show that every fine saturated \df structure on $\spec k$ is of this form.

In particular if $P=\NN$ we obtain what is usually called the \emph{standard logarithmic point} (over $k$). 
\end{example}


\section{Infinite root stacks of logarithmic schemes}\label{sec:irss-log-schemes}

Let $X = (X, A, L)$ be a \fs logarithmic scheme, $j\colon A \arr B$ an injective homomorphism of sheaves of monoids on $X\et$. With this we associate a category $\radice{B}X = \radice{B}{(X, A, L)}$ fibered in groupoids over $\cataff X$, as in \cite[Definition~4.16]{borne-vistoli1}, with the difference that here we do not assume that $B$ has local models given by fine monoids.

Given a morphism $t\colon T \arr X$, where $T$ is an affine scheme, an object $(M, \alpha)$ of $\radice{B}{X}(T)$ is a pair $(M, \alpha)$, where $M\colon t^{*}B \arr \div_{X\et}$ is a \sm functor, and $\alpha\colon t^{*}L \arr M \circ t^{*}j$ is an isomorphism of \sm functors from $t^{*}L$ to the composite $M \circ t^{*}j\colon t^{*}A \arr \div_{T\et}$.

An arrow $h$ from $(M, \alpha)$ to $(M', \alpha')$ is an isomorphism $h\colon M \arr M'$ of \sm functors $B \arr \div_{T\et}$, such that the diagram 
   \[
   \begin{tikzcd}
   {}&L\ar{rd}{\alpha'}\ar{ld}[swap]{\alpha}\\
   M \circ t^{*}j\ar{rr}{h \circ j} && M'\circ t^{*}j
   \end{tikzcd}
   \]
commutes.

Given a morphism $T' \arr T$ of affine $X$-schemes there is an obvious pullback functor $\radice{B}{X}(T) \arr \radice{B}{X}(T')$, which gives a structure of pseudo-functor from $\cataff X$ to groupoids. The associated fibered category is $\radice{B}{X}$.

\begin{remark}
Since logarithmic structures satisfy descent conditions for the Zariski topology, it is easy to check that for any morphism of schemes $T \arr X$ (where $T$ is not necessarily affine), there is an equivalence between the category of morphisms $T \arr \radice{B}X$ and pairs $(M , \alpha)$ defined as above.
\end{remark}

If the morphism $j\colon A\to B$ is Kummer and $B$ is coherent, then the root stack $\radice{B}{X}$ is algebraic \cite[Proposition 4.19]{borne-vistoli1}. This follows from Corollary \ref{cor:local-model-chart} below. Recall that $j$ is \emph{Kummer} if for every geometric point $x$ of $X$, the map $j_x\colon A_x\to B_x$ is a Kummer homomorphism of monoids, i.e. it is injective, and every element of $B_x$ has a multiple in the image. Moreover if the sheaf $B$ is saturated as well, one easily checks that $\radice{B} X\to X$ is the coarse moduli space.

There is an obvious variant of this construction, when $P\to A(X)$ is a chart for the logarithmic structure of $X$ and $P\to Q$ is a Kummer homomorphism, that we denote by $\radice{Q}{X}$ (see also \cite[Definition 4.12]{borne-vistoli1}). If $P\to Q$ is a chart for $B\to A$, we have a canonical equivalence $\radice{Q}{X}\cong \radice{B}{X}$ of fibered categories over $\cataff X$ \cite[Proposition 4.18]{borne-vistoli1}.

If $A \arr B$ and $B \arr B'$ are injective homomorphisms of sheaves of monoids on $X\et$, there is a base-preserving functor $\radice {B'} X \arr \radice B X$, defined as follows. If $(M',\alpha)$ is an object of $\radice {B'}X(T)$, call $M$ the composite $B \arr B' \xarr{M'} \div_{X\et}$. Then the restriction of $M$ to $A$ equals the restriction of $M'$, so $(M, \alpha)$ is an object of $\radice{B}X(T)$. This, together with the obvious map on arrows, defines a base-preserving functor $\radice{B'}X \arr \radice{B}X$.

\begin{proposition}
Let $A \arr B$, $B \arr B'$ and $B' \arr B''$ be homomorphisms of sheaves on monoids on $\cataff X$. The composite of the induced base-preserving functors $\radice{B''}X \arr \radice{B'}X$ and $\radice{B'}X \arr \radice{B}X$ equals the base-preserving functor $\radice{B''}X \arr \radice{B}X$ induced by the composite $B \arr B' \arr B''$. \qed
\end{proposition}

Here ``equals'' really implies equality, not just the existence of an isomorphism. The proof is a straightforward check.

If $n$ is a positive integer, we set $\frac{1}{n}A \eqdef A$, while the morphism $A \arr \frac{1}{n}A$ is multiplication by $n$. We have a maximal Kummer extension of $A$ given by $A\to A_{\QQ}$, where $A_{\QQ}$ is the sheaf of monoids $\varinjlim_{n}\frac{1}{n}A$. Note that the formation of $\frac{1}{n}A$ and $A_{\QQ}$ is compatible with pullback, i.e. if $f\colon Y\to X$ is a morphism of schemes, then we have natural isomorphisms $f^{*}\frac{1}{n}A\cong \frac{1}{n}f^{*}A$ and $f^{*}A_{\QQ}\cong (f^{*}A)_{\QQ}$.

We define the infinite root stack as the root stack corresponding to this maximal Kummer extension.

\begin{definition}
The \emph{\irs} of $X$ is $\infroot{X} \eqdef \radice {A_{\QQ}}{X}$. 
\end{definition}

We also define $\radice n {X} \eqdef \radice{\frac{1}{n}A}{X}$. The section $a\in A(T)$ will be denoted by $a/n$ when it is considered as a section of $\frac{1}{n}A$.

As a matter of notation, when we need to specify the logarithmic structure of the logarithmic scheme $X$ we will use the notations $\infroot{(X, A, L)}$ and $\radice n {(X, A, L)}$, or simply $\infroot{(A, L)}$ and $\radice n {(A, L)}$ when the scheme $X$ is fixed throughout the discussion.

Formation of root stacks and infinite root stacks commutes with base change.

\begin{proposition}\label{prop:base-change}
Let $(A, L)$ be a \df structure on a scheme $X$, and let $f\colon Y \arr X$ be a morphism of schemes. Then we have canonical equivalences of fibered categories
   \[
   \infroot{f^{*}(A, L)} \simeq Y\times_{X} \infroot{(A, L)}
   \]
and
   \[
   \radice n{f^{*}(A, L)} \simeq Y\times_{X} \radice n{(A, L)}\, 
    \]
over $Y$. \qed
\end{proposition}

The proof is immediate.

If $m$ and $n$ are positive integers and $m \mid n$, we can factor $A \arr \frac{1}{n}A$ as $A \arr \frac{1}{m}A \arr \frac{1}{n}A$, where the homomorphism $\frac{1}{m}A \arr \frac{1}{n}A$ is multiplication by $n/m$, sending $a/m$ into $\frac{(n/m)a}{n}$. This gives a base-preserving functor $\radice n {X} \arr \radice m {X}$. 

Let $\ZZ^{+}$ be the set of positive integer ordered by divisibility. The construction above gives a strict $2$-functor from $\ZZ^{+}$ to categories fibered in groupoids over $\cataff X$.  

The factorizations $A \arr \frac{1}{n}A \arr A_{\QQ}$ induce base-preserving functors $\infroot X \arr \radice n X$, and a morphism $\infroot X \arr \projlim_n\radice n X$.

\begin{proposition}\label{prop:projective-limit}
The morphism $\infroot X \arr \projlim_n\radice n X$ described above is an equivalence of fibered categories.
\end{proposition}

\begin{proof}
Given a morphism $t\colon T\to X$, where $T$ is an affine scheme, note preliminarily that there is a natural isomorphism $\varinjlim_n t^*\frac{1}{n}A\cong t^*A_\QQ$ of sheaves on $T\et$. Consequently, we have
\begin{align*}
\hom_X(T, \varprojlim_n \radice{n}{X}) & \cong \varprojlim_n \hom_X(T,\radice{n}{X})\\
& \cong \varprojlim_n \hom(t^*\tfrac{1}{n}A,\Div_{T\et})\\
& \cong \hom(\varinjlim_n t^*\tfrac{1}{n}A,\Div_{T\et})\\
& \cong \hom(t^*A_\QQ,\Div_{T\et})\\
& \cong \hom_X(T,\infroot{X})
\end{align*}
where in the middle lines we are considering the categories of symmetric monoidal functors. This proves the statement.
\end{proof}

\begin{remark}
Suppose that $\{B_{i}\}$ is a directed system of subsheaves of monoids of $A_{\QQ}$. We say that $\{B_{i}\}$ is \emph{cofinal} if $A_{\QQ}$ is the sheaf-theoretic union of the $B_{i}$; for example if $A\to B$ is a Kummer extension, the system $\{\frac{1}{n}B\}_{n \in \NN}$ is cofinal.

The same proof as in Proposition~\ref{prop:projective-limit} shows that $\infroot X = \projlim_i\radice{B_{i}}X$.
\end{remark}


\subsection{Local models for infinite root stacks}\label{sec:local-models}

Let $P$ be a \fs monoid, $n$ a positive integer. Set $X_{P}\eqdef \sz P$, where $\ZZ[P]$ is the monoid algebra of $P$. Then, as mentioned in \ref{sec:log-schemes}, $X_{P}$ has a canonical \fs \df structure, that we denote by $(A_{P}, L_{P})$. By definition, any \fs logarithmic structure on a scheme $X$ is étale-locally the pullback of $(A_{P}, L_{P})$ along a morphism $X \arr X_{P}$, for some $P$.

We aim at describing the infinite root stack $\infroot{X_{P}}$. For each positive integer $n$ set
   \[
   X_{p}^{[n]} \eqdef \spec\ZZ[\tfrac{1}{n}P];
   \]
we will consider $X_{p}^{[n]}$ as a scheme over $X_{P}$, via the ring homomorphism $\ZZ[P] \arr \ZZ[\frac{1}{n}P]$ induced by the embedding $P \subseteq \frac{1}{n}P$. We also set
   \[
   X_{P}^{[\infty]} \eqdef \sz{P_{\QQ}}\,.
   \]
Notice that $X_{P}^{[\infty]} = \projlim_{n} X_{P}^{[n]}$, where the projective limit is taken on the set of positive integers ordered by divisibility. Call $r$ the rank of the free abelian group $P\gr$, and
   \[
   \mmu_{n}(P) \eqdef \underhom_{\ZZ}\bigl(\tfrac{1}{n}P\gr/P\gr, \gm\bigr)
   \]
 the Cartier dual over $\spec \ZZ$ of the finite group $\tfrac{1}{n}P\gr/P\gr \simeq (\ZZ/n\ZZ)^{r}$. Clearly $\mmu_{n}(P)$ is isomorphic to $\mmu_{n}^{r}$ as a group scheme over $\ZZ$.

We also call
   \[
   \mmu_{\infty}(P) \eqdef \underhom_{\ZZ}\bigl(P_{\QQ}\gr/P\gr, \gm\bigr)
   \]
the Cartier dual of $P_{\QQ}\gr/P\gr \simeq (\QQ/\ZZ)^{r}$; we have a natural isomorphism $\mmu_{\infty}(P) \cong \projlim_{n}\mmu_{n}(P)$.

The $(\tfrac{1}{n}P\gr/P\gr)$-grading of $\ZZ[\tfrac{1}{n}P]$ induced by the tautological $\tfrac{1}{n}P$-grading induces an action of $\mmu_{n}(P)$ on $X_{P}^{[n]}$. Similarly, $\mmu_{\infty}(P)$ acts on $X_{P}^{[\infty]}$; the natural morphism $X_{P}^{[\infty]} \arr X_{P}^{[n]}$ is $\mmu_{\infty}(P)$-equivariant, when we let $\mmu_{\infty}(P)$ act on $X_{P}^{[n]}$ via the projection $\mmu_{\infty}(P) \arr \mmu_{n}(P)$.

Notice that the kernel of the projection $\mmu_{\infty}(P) \arr \mmu_{n}(P)$ equals $\mmu_{\infty}(\tfrac{1}{n}P)$.

\begin{lemma}\label{lem:right-quotients}
The natural homomorphism $X_{P}^{[\infty]} \arr X_{P}^{[n]}$ induces an isomorphism $X_{P}^{[\infty]}/\mmu_{\infty}(\tfrac{1}{n}P) \simeq X_{P}^{[n]}$.

In particular, $X_{P}^{[\infty]}/\mmu_{\infty}(P) = X_{P}$.
\end{lemma}

\begin{proof}
We have $\mmu_{\infty}(\tfrac{1}{n}P) = \spec\bigl(\ZZ[P\gr_{\QQ}/\tfrac{1}{n}P\gr]\bigr)$; the statement above is equivalent to the statement that the part of degree~$0$ with respect to the $P\gr_{\QQ}/\tfrac{1}{n}P\gr$-grading in $\ZZ[P_{\QQ}]$ is $\ZZ[\tfrac{1}{n}P]$. This follows from the equality $P_{\QQ} \cap \tfrac{1}{n}P\gr = \tfrac{1}{n}P$, which holds because $P$ is integral and saturated.
\end{proof}

\begin{definition}\label{def:local-irs}
The \emph{\irs of a \fs monoid $P$} is the fpqc quotient $\cR_{P}\eqdef [X_{P}^{[\infty]}/\mmu_{\infty}(P)]$.
\end{definition}

In other words, $\cR_{P}$ is the fibered category over $\aff$ defined as follows. An object $(T, E, f)$ of $\cR_{P}$, consist of an affine scheme $T$, a $\mmu_{\infty}(P)$-torsor $E \arr T$ and an equivariant map $f\colon E \arr X_{P}^{[\infty]}$. An arrow $(\phi, \Phi)\colon (T', E', f') \arr (T, E, f)$ consists of a morphisms of schemes $\phi\colon T' \arr T$ and $\Phi\colon E' \arr E$, such that $\Phi$ is $\mmu_{\infty}(P)$-equivariant, and the diagram
   \[
   \begin{tikzcd}[row sep = small]
   E' \rar{\Phi} \ar[bend left = 30]{rr}{f'} \dar & E \rar{f} \dar &X_{P}^{[\infty]}\\
   T' \rar{\phi} & T
   \end{tikzcd}
   \]
commutes.

Given an object $(T, E, f)$ of $\cR_{P}$, because of Lemma \ref{lem:right-quotients} the $\mmu_{\infty}(P)$-equivariant map $E\to X_{P}^{[\infty]}$ induces a map $T \arr X_{P}$; this gives a base-preserving functor $\cR_{P} \arr \cataff {X_{P}}$. So $\cR_{P}$ is a fibered category over $X_{P}$.

\begin{remark}
{Note that if $\{G_{i} \to S\}_{i \in I}$ is a projective system of affine group schemes of finite type and $G := \varprojlim G_{i}$, then the isomorphism classes of $G$-torsors on $S$ form a set. Here is a sketch of proof: if $T \arr S$ is a $G$-torsor, we obtain a projective system of $G_{i}$-torsors $T \times^{G}G_{i} \to S$, and $T = \varprojlim_{i}(T \times^{G}G_{i})$. Hence $T$ is a projective limit of schemes of finite type over $S$ indexed by $I$, and this limits its size.}
\end{remark}

We are aiming at the following result.

\begin{proposition}\label{prop:local-model}
We have an equivalence $\infroot{X_{P}} \simeq \cR_{P}$ of fibered categories over $X_{P}$.
\end{proposition}

\begin{remark}
This clarifies the nature of \irss considerably. On the one hand, we see that they are not algebraic stacks in the usual sense. On the other hand, étale-locally on $X$ they are projective limits of tame algebraic stacks, in the sense of \cite{dan-olsson-vistoli1}. This makes them very concrete and possible to study.
\end{remark}

Let us start from the following lemma, which is an analogue of Proposition 4.18 of \cite{borne-vistoli1}.

\begin{lemma}
Let $(X,A,L)$ be a logarithmic scheme with a chart $P\to A(X)$. Then the fibered category on $X\et$ of liftings of $P\to \Div(X)$ to a symmetric monoidal functor $P_\QQ\to \Div(-)$ is equivalent to the infinite root stack $\infroot{(X,A,L)}$.
\end{lemma}

\begin{proof}
The category of liftings in the statement can be seen as the inverse limit of the categories of liftings of $P\to \Div(X)$ to a symmetric monoidal functor $\frac{1}{n}P\to \Div(-)$. By \cite[Proposition 4.18]{borne-vistoli1}, this latter category of liftings is equivalent to the $n$-th root stack $\radice{n}{(X,A,L)}$. These equivalences are compatible with the structure maps of the two inverse systems, and hence the statement follows.
\end{proof}

{
We thank the referee for suggesting the following short proof of Proposition \ref{prop:local-model}.
}

\begin{proof}[\hbox{Proof of Proposition~\ref{prop:local-model}}]
By the previous lemma, we can identify $\infroot{X_P}$ with the fibered category on $(X_P)\et$ of liftings of $P\to \Div(X_P)$ to a symmetric monoidal functor $P_\QQ\to \Div(-)$.

Note that there is a morphism $X_P^{[\infty]}\to \infroot{X_P}$, given by the natural symmetric monoidal functor $P_\QQ\to \cO_{X_P^{[\infty]}}\to \Div({X_P^{[\infty]}})$ lifting $P\to \Div({X_P})$. Then, if $T\to \infroot{X_P}$ is a morphism corresponding to a lifting $P_\QQ\to \Div(T)$ of $P\to \Div(T)$, then the lifts to a map $T\to X_P^{[\infty]}$ correspond to choices of a trivialization of $P_\QQ\to \Div(T)\to (\cB{\gm})(T)$, lifting the natural trivialization of $P\to \Div({X_P})\to (\cB{\gm})({X_P})$. These liftings exist locally in $T$, and any two of them differ by a homomorphism $P\gr_\QQ/P\gr\to \gm$, so they form a torsor under $\mmu_\infty(P)$. Hence we have an equivalence $\infroot{X_P}\cong [X_P^{[\infty]}/\mmu_{\infty}(P)]$.
\end{proof}

From this description and the fact that root stacks are compatible with strict base-change (Proposition~\ref{prop:base-change}) we get the following corollary, which gives a local description of root stacks of a general \fs logarithmic scheme.

\begin{corollary}\label{cor:local-model-chart}
Let $X$ be a \fs logarithmic scheme with a Kato chart $P\to \cO_{X}(X)$, corresponding to a strict morphism $X\to X_{P}$. Then we have isomorphisms
   \begin{align*}
   \radice n {X} &\simeq [ ( X\times_{X_{P}}X_{P}^{[n]})/\mmu_{n}(P)  ]\\
   \intertext{and}
   \infroot{X} &\simeq [(X\times_{X_{P}}X_{P}^{[\infty]})/\mmu_{\infty}(P)]
   \end{align*}
of stacks over $X$.\qed
\end{corollary}


\subsection{Abstract infinite root stacks}\label{sec:abstract-irs}

Now we introduce an ``abstract'' definition for an infinite root stack over a scheme $X$. We will ultimately prove that these are the same objects as infinite root stacks of \df structures over $X$, but this alternative definition will give a convenient framework for describing the procedure that recovers the logarithmic structure from the infinite root stack.

\begin{definition}\label{def:airs}
If $X$ is a scheme, an \emph{infinite root stack} on $X$ is a stack $\cR$ on $\cataff{X}$ with the étale topology, such that there exists an étale covering $\{U_{i}\arr X\}$, and for each $i$ a morphism $f_{i}\colon U_{i} \arr \sz {P_{i}}$, where $P_{i}$ is a \fs monoid, and an equivalence $U_{i} \times_{X} \cR \simeq U_{i} \times_{\sz {P_{i}}} \cR_{P_{i}}$ of stacks over $U_{i}$.

If $f\colon \cR \arr \cX$ is a base-preserving functor of stacks in the étale topology on $\aff$, we say that $\cR$ is an \irs over $\cX$ if for any affine scheme $T$ and any object of $\cX(T)$, the fibered product $T\times_{\cX}\cR$ is an \irs over $T$.
\end{definition}

Note that in the definition above $U_{i} \times_{\sz {P_{i}}} \cR_{P_{i}}$ is the infinite root stack of the logarithmic structure on $U_{i}$ given by pullback from $\sz {P_{i}}$ (by \ref{prop:base-change}). An explicit quotient presentation for this stack was just described in Corollary~\ref{cor:local-model-chart}.

\begin{remark}
It is obvious from the definition that if $\cR$ is an \irs over a scheme $X$ and $Y \arr X$ is a morphism of schemes, the pullback $Y \times_{X} \cR$ is an \irs over $Y$. Since the property of being an \irs is clearly local in the étale topology, we see that $\cR$ is an \irs over a scheme $X$ if an only if it is an \irs over $\cataff X$. In other words, the  second part of the definition generalizes the first.
\end{remark}

\begin{proposition}
Let $(A, L)$ be a \fs \df structure on a scheme $X$. Then $\infroot{(A, L)}$ is an \irs over $X$. More generally, the morphism $\infroot{(A, L)} \arr \radice n {(A,L)}$ makes $\infroot{(A,L)}$ into an \irs over $\radice n {(A,L)}$.
\end{proposition}

\begin{proof}
It follows easily from étale descent of logarithmic structures that $\infroot{(A, L)}$ and $\radice n {(A,L)}$ are stacks in the étale topology.

The fact that $\infroot{(A, L)} \arr X$ is an \irs follows immediately from Proposition~\ref{prop:local-model}.

Let us check that $\infroot{(A,L)}$ is an \irs over $\radice n {(A,L)}$. Let $T$ be an affine scheme, $T \arr \radice n {(A,L)}$ a morphism, corresponding to a morphism $f\colon T \arr X$ and a \df structure $(\tfrac{1}{n}f^{*}A, M)$ on $T$. It is clear from the definition that the fibered product $T\times_{\radice n {(A, L)}} \infroot{(A, L)}$ is equivalent to $\infroot{(\tfrac{1}{n}f^{*}A, M)}$, reducing the second statement to the first.
\end{proof}

We will need a description of the geometric fibers of an \irs, or, equivalently, a description of an \irs over an algebraically closed field.

Let $P$ a sharp \fs monoid, and $k$ a field. 
We will denote by $\infroot{P/k}$ the infinite root stack of the logarithmic point on $k$ with monoid $P$ (Example \ref{ex:log-points}), i.e. the \df structure $(A, L)$ on $\spec k$, where $A$ is the constant sheaf of monoids on $(\spec k)\et$ corresponding to $P$, and $L\colon A \arr \div_{(\spec k)\et}$ corresponds to the homomorphism $\Lambda\colon P \arr k$ that sends $0$ to $1$ and everything else to $0$.

In other words $\infroot{P/k}$ is the fibered product $\spec k \times_{\sz P} \cR_{P}$, where $\spec k \arr \sz{P}$ corresponds to the ring homomorphism $\ZZ[P] \arr k$ determined by $\Lambda$. Or, again, by Corollary~\ref{cor:local-model-chart} we have
   \[
   \infroot{P/k} = \bigl[\spec(k[P_{\QQ}]/(P^{+}))/\mmu_{\infty}(P)\bigr]
   \]
where $P^{+} \eqdef P\setminus \{0\}$, and the action of $\mmu_{\infty}(P)$ on $\spec(k[P_{\QQ}]/(P^{+}))$ is determined by the natural $P\grq/P\gr$-grading on $k[P_{\QQ}]/(P^{+})$.

The reduced substack $(\infroot{P/k})\red$ is the classifying stack
   \[
   \cB_{k} \mmu_{\infty}(P) = \bigl[\spec k/ \mmu_{\infty}(P)\bigr]\,.
   \]

If $k$ is algebraically closed, every logarithmic structure on $\spec k$ is of the form above. Since every \irs comes étale-locally from a logarithmic structure, and the étale topology on $\spec k$ is trivial, we obtain the following.

\begin{proposition}\label{prop:geometric-fibers}
If $k$ is an algebraically closed field, an \irs over $\spec k$ is of the form $\infroot{P/k}$ for some sharp \fs monoid $P$.
\end{proposition}

Let $\pi\colon \cX \arr X$ be a category fibered in groupoids on a scheme $X$. If $U \subseteq X$ is an open substack, the inverse image $\pi^{-1}(U) \eqdef U \times_{X} \cX$ is an open substack of $\cX$. We say that $\pi$ is \emph{a homeomorphism} if this gives a bijective correspondence between open subschemes of $X$ and open substacks of $\cX$. We say that $\pi$ is \emph{a universal homeomorphism} if for any morphism $Y \arr X$, the projection $Y\times_{X}\cX \arr Y$ is a homeomorphism.

\begin{definition}\label{def:fpqc-atlas}
Let $\cX$ be a category fibered in groupoids over $\aff$. An \emph{fpqc atlas $U \arr \cX$} is a map from a scheme $U$ that is representable by fpqc maps (in the sense of \cite[Definition~2.34]{vistoli-descent}).
\end{definition}

Here are some important properties of \irss, which we collect in the following Proposition.

\begin{proposition}\call{prop:properties-irs} Let $\pi\colon \cR \arr X$ be an \irs.
\begin{enumeratea}

\itemref{1} The morphism $\pi$ is a universal homeomorphism {for the Zariski topology}.

\itemref{2} If $\Omega$ is an algebraically closed field, $\pi$ induces a bijection between isomorphism classes in $\cR(\Omega)$ and $X(\Omega)$.

\itemref{3} If $T$ is a scheme, any morphism $\cR \arr T$ factors through a unique morphism $X \arr T$.

\itemref{4} The diagonal $\cR \arr \cR\times_{X}\cR$ is affine.

\itemref{6} The fibered product of two schemes over $\cR$ is a scheme.

\itemref{5} The stack $\cR$ has an fpqc atlas.

\end{enumeratea}
\end{proposition}

Properties \refpart{prop:properties-irs}{2} and \refpart{prop:properties-irs}{3} are usually expressed by saying that $X$ is the moduli space of $\cR$.

\begin{remark}
One can show that property \refpart{prop:properties-irs}{3} also holds when $T$ is an algebraic space. We omit the proof, as this fact won't be needed in this paper.
\end{remark}

\begin{proof}[Proof of Proposition~\ref{prop:properties-irs}]\hfil

\proofofpart{prop:properties-irs}{2}
We can base change and assume that $X = \spec \Omega$, with $\Omega$ an algebraically closed field. Then by Proposition~\ref{prop:geometric-fibers} $\cR$ is of the form $\cR = \infroot{P/\Omega}$, so that $\cR(\Omega) \simeq (\cR\red)(\Omega) \simeq (\cB_{\Omega}\mmu_{\infty}(P))(\Omega)$. Since $\mmu_{\infty}(P)$ is profinite and $\Omega$ is algebraically closed, every $\mmu_{\infty}(P)$-torsor is trivial, so $\cR(\Omega)$ has only one isomorphism class, which is exactly what we need.

{\proofofpart{prop:properties-irs}{1} First of all, being an \irs is a stable property under base change, so it is enough to prove that $\pi\colon \cR \arr X$ is a homeomorphism.}

{The morphism $\pi\colon \cR \arr X$ is an fpqc cover, hence it satisfies descent for open substacks. But the two projections $\cR\times_{X}\cR \arr \cR$ are bijective on geometric points, by \refpart{prop:properties-irs}{2}, and they have continuous sections, hence they are homeomorphisms. So $\cR \arr X$ is a homeomorphism, as claimed.}

\proofofpart{prop:properties-irs}{3} Let $f\colon \cR \arr T$ be a morphism. Assume that $T$ and $X$ are affine, and that $\cR$ is a pullback $X\times_{X_{P}}\cR_{P}$ for a certain map $X \arr X_{P}$, where $P$ is a \fs monoid. Then the result follows immediately from the facts that $\cR = \bigl[(X\times_{X_{P}}X_{P}^{[\infty]})/\mmu_{\infty}(P)\bigr]$ and $(X\times_{X_{P}}X_{P}^{[\infty]})/\mmu_{\infty}(P) = X$. 

In the general case, take a covering $\{T_{a}\}$ of $T$ by affine open subschemes. By \refpart{prop:properties-irs}{1} the open substack $f^{-1}(T_{a})\subseteq \cR$ is the inverse image $\pi^{-1}({V_{a}})$ of an open subscheme $V_{a} \subseteq X$. Let $\{U_{i} \arr X\}$ be an étale covering of $X$ such that for each $i$ the image of $U_{i}$ in $X$ is contained in some $V_{a}$, and the pullback $U_{i}\times_{X}\cR$ comes from a map $U_{i} \arr X_{P_{i}}$. Then it follows from the previous case that for each $i$ the composite $\cR_{U_{i}} \arr \cR \arr T$ factors uniquely through $U_{i}$. By étale descent of morphisms to $T$, this proves the uniqueness of the factorization $\cR \arr X \arr T$. On the other hand this uniqueness applied to the fibered products $U_{i}\times_{X} U_{j}$ shows that the composites $U_{i}\times_{X}U_{j} \xarr{\pr_{1}} U_{i} \arr T$ and $U_{i}\times_{X}U_{j} \xarr{\pr_{2}} U_{j} \arr T$ are equal, so the maps $U_{i} \arr T$ glue to a map $X \arr T$, which gives the required factorization.

\proofofpart{prop:properties-irs}{4} This is a local problem in the étale topology of $X$, so we may assume that $X$ is affine, and that $\cR$ is of the form $X \times_{X_{P}} \cR_{P}$ for some \fs monoid $P$ and some map $X \arr X_{P}$. Set once again $Z \eqdef X\times_{X_{P}}X_{P}^{[\infty]}$, so that $\cR = [Z/\mmu_{\infty}(P)]$. Then the diagonal $\cR \arr \cR \times_{X} \cR$ corresponds to the morphism
   \[
   [Z/\mmu_{\infty}(P)] \arr [Z\times_{X}Z/\mmu_{\infty}(P)\times\mmu_{\infty}(P)]
   \]
induced by the diagonals $Z \arr Z\times_{X}Z$ and $\mmu_{\infty}(P) \arr \mmu_{\infty}(P)\times\mmu_{\infty}(P)$. Hence the result follows from Lemma~\ref{lem:affine-map} below.

\proofofpart{prop:properties-irs}{6} This follows formally from \refpart{prop:properties-irs}{4}.

\proofofpart{prop:properties-irs}{5} 
Let $\{U_{i} \arr X\}$ be an étale covering such that each pullback $U_{i}\times_{X}\cR$ is obtained as $U_{i} \times_{X_{P_{i}}}\cR_{P_{i}}$ for a certain map $U_{i} \arr X_{P_{i}}$ for some \fs monoid $P_{i}$. Set $Z_{i} \eqdef U_{i}\times _{X_{P_{i}}}X_{P_{i}}^{[\infty]}$, so that $U_{i}\times_{X}\cR = [Z_{i}/\mmu_{\infty}(P_{i})]$. Set $X' \eqdef \bigsqcup_{i}U_{i}$ and $U \eqdef \bigsqcup_{i}Z_{i}$. The natural morphism $U \arr X'\times_{X}\cR$ is affine and faithfully flat, while the projection $X'\times_{X}\cX \arr \cR$ is étale and surjective. Hence, the composite $U \arr X'\times_{X}\cR \arr \cR$ is fpqc.
\end{proof}

\begin{lemma}\label{lem:affine-map}
Let $G$ be a diagonalizable group scheme over a ring $R$, and $H \subset G$ a diagonalizable subgroup scheme. 
Suppose that $X$ and $Y$ are 
schemes over $R$, that $G$ acts on $X$, $H$ acts on $Y$, and $f\colon Y \arr X$ is an $H$-equivariant affine morphism of $R$-schemes. Then the induced morphism $[Y/H] \arr [X/G]$ is representable and affine.

Furthermore, if the quotient $G/H$ and the morphism $Y \arr X$ are \fp, then $[Y/H] \arr [X/G]$ is also \fp.
\end{lemma} 

\begin{proof}
Since fpqc descent for affine maps is effective, we have the following fact. Suppose that $\cF$ and $\cG$ are stacks for the fpqc topology of $\cataff R$, where $R$ is a ring, and $\Phi\colon \cG \arr \cF$ is a base-preserving functor. Assume that there exists {a representable fpqc cover} $X \arr \cF$, where $X$ is a
scheme, such that $\cG \times_{\cF}X$ is represented by {schemes that are affine over $X$}. Then $\Phi$ is is representable and affine.

For the proof of the Lemma, notice that we have a cartesian diagram
   \[
   \begin{tikzcd}
   {}[Y/H] \rar & {[X/H]} \rar \dar & {[X/G]}\dar\\
   {}& \cB_{R}H\rar & \cB_{R}G
   \end{tikzcd}
   \]
hence it is enough to show that $[Y/H] \arr [X/H]$ and $\cB_{R}H \arr \cB_{R}G$ are affine.

For the first, we have a cartesian diagram
   \[
   \begin{tikzcd}
   {}Y \rar\dar &{}X\dar\\
   {[Y/H]}\rar &{[X/H]}\,.
   \end{tikzcd}
   \]
For the second, the point is that the stack quotient $[G/H] = G/H$ is represented by an affine diagonalizable group scheme, and the diagram
   \[
   \begin{tikzcd}
   {}[G/H] \rar \dar & \spec R\dar\\
   \cB_{R}H\rar & \cB_{R}G\,
   \end{tikzcd}
   \]
is cartesian.
\end{proof}


\section{Quasi-coherent and \fp sheaves\\on \irss}%
\label{realsec:qc}

This section is dedicated to a discussion of quasi-coherent sheaves on an infinite root stack. We start off by considering quasi-coherent sheaves on an arbitrary fibered category, and then step by step we specify the situation to an infinite root stack over a scheme, as defined in \ref{sec:abstract-irs}.

\subsection{Quasi-coherent sheaves on fibered categories}\label{sec:qc}

Let $\rmp_{\cX}\colon \cX \arr \aff$ be a category fibered in groupoids. Define the presheaf $\cO_{\cX}\colon \cX\op \arr \catring$ by sending an object $\xi \in \cX$ into $\cO(\rmp_{\cX}\xi)$. A presheaf of $\cO_{\cX}$-modules is defined in the obvious way: it is a functor $F\colon \cX\op \arr \catab$ with a structure of $\cO(\rmp_{\cX}\xi)$ module on each $F(\xi)$, such that if $\phi\colon \eta \arr \xi$ is an arrow, the corresponding homomorphism $\phi^{*}\colon F(\xi) \arr F(\eta)$ is linear with respect to the ring homomorphism $\phi^{*}\colon \cO(\rmp_{\cX}\eta) \arr \cO(\rmp_{\cX}\xi)$.

There is an obvious notion of homomorphism of presheaves of $\cO_{\cX}$-modules; the resulting category of $\cO_{\cX}$ modules is an abelian category, denoted by $\catmod \cX$. 

The category $\catmod \cX$ has an obvious symmetric monoidal structure given by tensor product: if $F$ and $G$ are presheaves of $\cO_{\cX}$-modules, we define $F\otimes_{\cO_{\cX}}G$ via the rule $(F\otimes_{\cO_{\cX}}G)(T) = F(T)\otimes_{\cO(T)}G(T)$.

\begin{remark}
In this paper we are not concerning ourselves with set-theoretic difficulties; however, it is hard to ignore the fact that the category $\catmod \cX$ is not locally small, and we don't see a general method for reducing its size. Fortunately, this problem does not arise with \qc sheaves on stacks with an fpqc atlas, which are the only ones that really interest us.
\end{remark}

The following is the definition of \qc sheaf used in \cite{rosenberg-kontsevich}.

\begin{definition}
A presheaf of $\cO_{\cX}$-modules $F$ is \emph{\qc} if for every arrow $\eta \arr \xi$ in $\cX$, the induced homomorphism $\cO(\rmp_{\cX}\eta) \otimes_{\cO(\rmp_{\cX}\xi)}F(\xi) \arr F(\eta)$ is an isomorphism.
\end{definition}

\begin{remark}
The category of \qc (pre)sheaves on $\cX$ is a full additive subcategory $\catqcoh \cX$ of $\catmod \cX$. It is also a monoidal subcategory, because it is easy to see that tensor products of \qc sheaves are \qc. However, kernels of homomorphism of \qc presheaves in $\catmod \cX$ are not necessarily \qc, so it is not clear to us (and probably not true) that $\catqcoh \cX$ is abelian in general.
\end{remark}

Alternatively, a presheaf $F$ is \qc if for every affine scheme $U$ and every object of $\cX(U)$, the pullback of $F$ to $\cataff U$ via the corresponding base-preserving functor $\cataff U \arr \cX$ is isomorphic to the presheaf defined by an $\cO(U)$-module.

Here is an alternative definition. Call $\qcoh$ the fibered category of \qc sheaves over $\aff$. A clean way is to think of $\aff$ as the dual of the category of commutative rings, and define $\qcoh$ as the dual of the category of modules, in the following sense. An object of $\qcoh$ is a pair $(A, M)$, where $A$ is a commutative ring and $M$ is an $A$-module. An arrow from $(B, N)$ to $(A, M)$ is a pair $(\phi, \Phi)$, where $\phi\colon A \arr B$ is a ring homomorphism, and $\Phi\colon M \arr N$ is a homomorphism of groups that is $\phi$-linear.

The category $\qcoh$ is fibered over $\aff$; an arrow $(B, N) \arr (A, M)$ is cartesian if and only if the induced homomorphism $B\otimes_{A}M \arr N$ is an isomorphism.

Then an $\cO_{\cX}$-module corresponds to a base-preserving functor $\cX \arr \qcoh$ of categories fibered over $\aff$. Such a functor is \qc if and only if it is cartesian.

\begin{remark}\label{rmk:topology}
Any topology on $\aff$ induces a topology on $\cX$: a set of arrows $\{\xi_{i} \arr \xi\}$ is a covering if their images $\{\rmp_{\cX}\xi_{i} \arr \rmp_{\cX}\xi\}$ are a covering. In particular, we have fpqc, fppf, étale and Zariski topologies on $\cX$.
\end{remark}

\begin{remark}\label{rmk:qc->sheaf}
Descent theory for modules easily implies that a \qc sheaf on a fibered category is a sheaf in the fpqc topology. This justifies the terminology ``\qc sheaf'' as opposed to the more cumbersome ``\qc presheaf''.
\end{remark}

Suppose that $\cX$ and $\cY$ are categories fibered in groupoids over $\aff$, and let $f\colon \cY \arr \cX$ be a base-preserving functor. There is an obvious pullback map $f^{*}\colon \catmod \cX \arr \catmod \cY$ sending each presheaf $F$ of $\cO_{\cX}$ into the presheaf $\eta \arrto F(f\eta)$; this sends $\catqcoh X$ into $\catqcoh Y$. Furthermore, we have $f^{*}\cO_{\cX} = \cO_{\cY}$ tautologically.


Suppose that $\cX$ has an fpqc atlas $U \arr \cX$ (Definition \ref{def:fpqc-atlas}). Then the fibered product $R \eqdef U \times_{\cX} U$ is represented by a scheme, and we obtain an fpqc groupoid $R \double U$. We consider the category $\catqcoh{(R \double U)}$ of \qc sheaves $F$ on $U$ (in the classical sense) with descent data $\pr_{2}^{*}F \simeq \pr_{1}^{*}F$.

The following is a standard application of fpqc descent.

\begin{proposition}\label{prop:qc<->qc-on-groupoid}
Let $\cX$ and $R \double U$ be as above. Then we have an equivalence of additive categories of $\catqcoh \cX$ with $\catqcoh{(R \double U)}$.

In particular, if $X$ is a scheme the category of \qc sheaves on $\cataff{X}$ is equivalent to the category of \qc sheaves on $X$.\qed
\end{proposition}

Since $\catqcoh{(R \double U)}$ is clearly abelian, we have the following.

\begin{corollary}
If $\cX$ has an fpqc atlas, the category $\catqcoh \cX$ is abelian.\qed
\end{corollary}

\begin{remark}
The category $\catqcoh{(R \double U)}$ is locally small, so in each of the following proofs we can take the stacks $\cX$ with an fpqc atlas and substitute a locally small category for $\catqcoh\cX$.
\end{remark}

We are going to use the following particular case of Proposition~\ref{prop:qc<->qc-on-groupoid}. If $G$ is an affine group scheme over $\ZZ$ acting an on a ring $R$, we define an equivariant $R$-module as an $R$-module $M$ with an action of $G$, such that for every ring $A$, every $r \in R \otimes_{\ZZ} A$, $m\in M \otimes_{\ZZ} A$, and $g \in G(A)$, we have $(gr)(gm) = g(rm)$.

\begin{proposition}\label{prop:qc-group-actions}
Suppose that $X = \spec A$ is a affine scheme, $G$ an affine group scheme over $\ZZ$ acting on $X$. Then there is an equivalence of tensor categories between $\catqcoh{[X/G]}$ and the category of $G$-equivariant $A$-modules.\qed
\end{proposition}

Let $\cX$ be a category fibered in groupoids over $\aff$ and $F$ be a sheaf of $\cO_{\cX}$ modules. We define $F(\cX) = \H^{0}(\cX, F)$ to be $\hom_{\cO_{\cX}}(\cO_{\cX}, F) = \lim_{\cX}F$. This is a (large) module over the commutative ring $\cO(\cX) \eqdef \cO_{\cX}(\cX)$. It is easy to see that $\cO(\cX)$ can also be seen as the ring of base-preserving functors $\cX \arr \AA^{1}_{\ZZ}$.

Let $f\colon \cY \arr \cX$ be a base-preserving functor of categories fibered in groupoids over $\aff$. The pullback functor $f^{*}\colon \catmod \cX \arr \catmod \cY$ induces a homomorphism
   \[
   F(\cX) = \hom_{\cO_{\cX}}(\cO_{\cX}, F)\xarr{f^{*}} \hom_{\cO_{\cY}}(f^{*}\cO_{\cX}, f^{*}F)
   = (f^{*}F)(\cY)\,.
   \]

Given a sheaf of $\cO_{\cY}$-modules $G$, we can also define a sheaf of $\cO_{\cX}$-modules $f_{*}G$ as follows. For each affine scheme $T$ and each object $\xi$ of $\cX(T)$ we consider the fibered product
   \[
   \cY_{T} \eqdef \cataff T \times_{\cX} \cY\,,
   \]
where the morphism $\cataff T \arr \cX$ corresponds to $\xi$, with the projection $\pi_{T}\colon \cY_{T} \arr \cY$. We define $f_{*}G(T) \eqdef (\pi_{T}^{*} G)(\cY_{T})$. The natural homomorphism $\cO(T) \arr \cO(\cY_{T})$ makes $f_{*}G(T)$ into an $\cO(T)$-module, and the obvious arrows make $f_{*}G$ into a sheaf of $\cO_{\cX}$-modules.

In what follows we will shorten $(\pi_{T}^{*} G)(\cY_{T})$ into $G(\cY_{T})$; this should not give rise to confusion.

Let $\phi\colon G' \arr G$ be a morphism of sheaves of $\cO_{\cY}$-modules; this induces a homomorphism $f_{*}\phi\colon f_{*}G' \arr f_{*}G$. If we ignore set-theoretic problems, we have defined an additive functor $f_{*}\colon \catmod \cY \arr \catmod \cX$, which is easily checked to be a left adjoint to $f^{*}\colon \catmod \cX \arr \catmod\cY$. 

Wishing to be a little more careful, assume that $\cY$ has an fpqc atlas $U \arr \cY$, that the diagonal $\cX \arr \cX \times \cX$ is representable, and that $G$ is \qc. Then $U_{T} \eqdef T \times_{\cY} U$ is a scheme, and the projection $U_{T} \arr \cY_{T}$ is an fpqc atlas. Hence the pullback $f_{*}G(T) \eqdef G(\cY_{T}) \arr G(U_{T})$ is injective. On the other hand $G(U_{T})$ is in canonical bijection with the set of global sections of the restriction of $G$ to the Zariski site of $U_{T}$. This allows us to replace $G(U_{T})$ with a set, defining a functor $f_{*}\colon \catqcoh{\cY} \arr \catmod \cX$, with the property that for every sheaf of $\cO_{\cX}$-module $F$ there is a canonical isomorphism $\hom_{\cO_{\cY}}(f^{*}F, G) \simeq \hom_{\cO_{\cX}}(F, f_{*}G)$ that is functorial in $F$ and in $G$.

This proves the following.

\begin{proposition}
Let $f\colon \cY \arr \cX$ be a base-preserving functor between categories fibered in groupoids over $\aff$. Assume that $\cY$ has an fpqc atlas and the diagonal $\cX \arr \cX \times \cX$ is representable. Then there exists an additive functor
   \[
   f_{*}\colon \catqcoh \cY \arr \catmod \cX
   \]
with the property that for every \qc sheaf $G$ in $\cY$ and every presheaf of $\cO_{X}$ modules $F$, there is a canonical isomorphism
   \[
   \hom_{\cO_{\cY}}(f^{*}F, G) \simeq \hom_{\cO_{\cX}}(F, f_{*}G)
   \]
that is functorial in $F$ and in $G$.\qed
\end{proposition}

\begin{remark}
When $X$ is a scheme, the category $\catmod X \eqdef \catmod\cataff X$ is much larger than the category of $\cO_{X}$-modules on the small Zariski site of $X$, even {though} the corresponding subcategories of \qc sheaves are equivalent.

Also, if $G$ is a \qc coherent sheaf on $X$, the pushforward $f_{*}G$ is not necessarily \qc, even if $f\colon Y \arr X$ is, for example, a proper map of noetherian schemes. In fact, $f_{*}G$ is \qc if and only if the restriction of $f_{*}G$ to the small Zariski site of $X$ is \qc, and the formation of $f_{*}G$ in the Zariski topology commutes with base change.
\end{remark}

There is an alternate definition of the category of \qc sheaves. Recall that if $\cC$ is a site with a sheaf $\cO_{\cC}$ of commutative rings, a sheaf of $\cO_{\cC}$-modules is \emph{\qc} if it has local presentations, that is, for any object $\xi$ there is a covering $\{\xi_{i} \to \xi\}$, such that the restriction $F\rest{(\cC/\xi_{i})}$ to the comma site $(\cC/\xi_{i})$ is the cokernel of a homomorphism of free sheaves $\cO_{(\cC/\xi_{i})}^{\oplus B} \arr \cO_{(\cC/\xi_{i})}^{\oplus A}$, where the sets $A$ and $B$ are not necessarily finite. To avoid confusion, we will refer to such sheaves as sheaves with local presentations.

We will use the following lemma.

\begin{lemma}\label{lem:subcategory-equivalence-qc}
Let $\cC$ be a site with finite fibered products and a sheaf of rings $\cO_{\cC}$. Let $\cD \subseteq \cC$ be a full subcategory closed under fibered products, considered as a ringed site with the induced topology, and the restriction $\cO_{\cD}$ of $\cO_{\cC}$. Suppose that each object $\xi$ of $\cC$ has a covering $\{\eta_{i} \arr \xi\}$ in which each $\eta_{i}$ is in $\cD$. Then the embedding $\cD \subseteq \cC$ induces an equivalence of topoi of sheaves of sets on $\cC$ and $\cD$, and an equivalence of the categories of sheaves with local presentations on $\cC$ and on $\cD$.
\end{lemma}

\begin{proof}
From \cite[Tag~039Z]{stacks-project}) we see that the embedding $\cD \subseteq \cC$ induces an equivalence of topoi. The second result follows easily.
\end{proof}

We will consider $\cX$ as a site with the fpqc topology (Remark \ref{rmk:topology}). This site can be enlarged in two ways. 

One is the category $\cX\sch$ whose objects are morphisms $U \arr \cX$, where $U$ is a scheme. The arrows from $U  \arr \cX$ to $V \arr \cX$ are isomorphism classes of $2$-commutative diagrams
   \[
   \begin{tikzcd}[column sep=small, row sep = small]
   {}U \ar{rr}\ar{rd}&&V\ar{dl}\\
   {}&\cX
   \end{tikzcd}
   \]

There is an obvious fpqc topology on $\cX\sch$ (a set of arrows $\{(U_{i} \to \cX) \to (U \to \cX)\}$ is an fpqc covering if the corresponding set of morphisms $\{U_{i} \arr U\}$ is an fpqc covering).

The $2$-categorical version of Yoneda's lemma gives a fully faithful functor $\cX \arr \cX\sch$. 

The other is the category $\cX\rep$, each of whose objects consist of a category fibered in groupoids $\cA$ and a representable base-preserving functor $\cA \arr \cX$. An arrow from $\cA \arr \cX$ to $\cB \arr \cX$ consists of an isomorphism class of $2$-commutative diagrams
   \[
   \begin{tikzcd}[column sep=small, row sep = small]
   {}\cA \ar{rr}\ar{rd}&&\cB\ar{dl}\\
   {}&\cX
   \end{tikzcd}
   \]

Let us define the fpqc topology on the category$\cX\rep$. Given a set $\{\cA_{i} \arr \cX\}_{i\in I}$ of representable base-preserving functors, we denote by $\bigsqcup_{i \in I}\cA \arr \cX$ the object of $\cX$ whose fibers over an affine scheme $T$ are the disjoint union of the fibers of the $\cA_{i} \arr \cX$. More precisely, an object $(\{T_{i}, \xi_{i}\})$ of $\bigsqcup_{i \in I}\cA_{i}$ over an affine scheme $T$ is a set of open and closed subschemes $T_{i} \subseteq T$ such that $T= \bigsqcup_{i} T_{i}$, and $\xi_{i} \in \cA_{i}(T_{i})$. The arrows are defined in the obvious way.

Then an fpqc covering $\{(\cA_{i} \arr \cX) \arr (\cA \arr \cX)\}_{i \in I}$ is a set of maps such that the induced morphism $\bigsqcup_{i \in I}\cA_{i} \arr \cA$ is fpqc (it is automatically representable).

The category $\cX\rep$ comes equipped with a natural structure sheaf $\cO_{\cX\rep}$, which restricts to the structure sheaf $\cO_{\cX}$ on $\cX$. This associates with every representable map $\cA \arr \cX$ the ring $\cO(\cA)$.

The category $\cX\sch$ is a full subcategory of $\cX\rep$.

\begin{proposition}\call{prop:equivalence-qc}
Let $\cX$ be a category fibered in groupoids on $\aff$ with an fpqc atlas $U \arr \cX$; set $R \eqdef U \times_{X} U$. Then we have equivalences of additive tensor categories among the following.

\begin{enumeratea}

\itemref{5} The category $\catqcoh{(R \double U)}$ of \qc sheaves with descent data on the groupoid $R \double U$.

\itemref{1} The category $\catqcoh{\cX}$ of \qc sheaves on $\cX$.

\itemref{2} The category of sheaves with local presentations on the site $\cX$ with the fpqc topology.

\itemref{3} The category of sheaves with local presentations on the site $\cX\sch$ with the fpqc topology.

\itemref{4} The category of sheaves with local presentations on the site $\cX\rep$ with the fpqc topology.

\end{enumeratea}
\end{proposition}

We thank the referee for suggesting a simplification in the last step of the proof.

\begin{proof}
We have already discussed the equivalence between \refpart{prop:equivalence-qc}{5} and \refpart{prop:equivalence-qc}{1}.

Consider the fully faithful functor $\cX \arr \cX\sch$; this preserves fibered products. Since every scheme has an fpqc cover by affine schemes, from Lemma~\ref{lem:subcategory-equivalence-qc} we obtain an equivalence between \refpart{prop:equivalence-qc}{2} and \refpart{prop:equivalence-qc}{3}. An analogous argument with the fully faithful functor $\cX\sch \arr \cX\rep$ give the equivalence of \refpart{prop:equivalence-qc}{3} and \refpart{prop:equivalence-qc}{4}.

Let us produce an equivalence between \refpart{prop:equivalence-qc}{1} and \refpart{prop:equivalence-qc}{2}; more precisely, let us show that a presheaf of $\cO_{\cX}$-modules $F\colon \cX \op \arr \catab$ is \qc if and only if it is an fpqc sheaf with local presentations. The fact that $F$ is an fpqc sheaf follows easily from descent theory (see Remark~\ref{rmk:qc->sheaf}).

Let us show that it has local presentations.

Suppose that $\xi$ is an object of $\cX$, and set $T \eqdef \p_{\cX}\xi$. Choose a presentation $\cO(T)^{\oplus B} \arr \cO(T)^{\oplus A} \arr F(\xi) \arr 0$ of $F(\xi)$ as an $\cO(T)$-module. This extends uniquely to a sequence
   \[
   \cO_{(\cX/\xi)}^{\oplus B} \arr \cO_{(\cX/\xi)}^{\oplus A} \arr F\rest{(\cX/\xi)} \arr 0\,.
   \]
Since $F$ is \qc and tensor product is right exact, it follows that the sequence of sheaves above is exact, which gives a local presentation for $F$.

Let us remark that this gives an embedding of $\catqcoh \cX$ into the category of sheaves of $\cO_{\cX}$-modules in the fpqc topology. We claim that this embedding is right exact, that is, it preserves cokernels. This is clear when $\cX$ is a scheme, and is easily reduced to this case using the atlas $U \arr \cX$.

Conversely, let $F\colon \cX\op \arr \catab$ be an fpqc sheaf of $\cO_{\cX}$-modules with local presentations. Suppose that $\eta \arr \xi$ is an arrow in $\cX$; denote by $V \arr T$ its image in $\aff$. We need to show that the corresponding homomorphism
   \begin{equation}\label{eq:base-change}
   \cO(V)\otimes_{\cO(T)}F(\xi) \arr F(\eta) 
   \end{equation}
is an isomorphism.  This will be done in two steps.

\step 1 It is straightforward to check that (\ref{eq:base-change}) holds when $F = \cO_{\cX}^{\oplus A}$ is a free sheaf (not necessarily of finite type): by definition every fpqc covering of an affine scheme has a finite subcovering, and so we have $F(\xi) = \cO(\p_{\cX}\xi)^{\oplus A}$ for every object $\xi$ of $\cX$. Since the embedding of $\catqcoh \cX$ into the category of sheaves of $\cO_{\cX}$-modules in the fpqc topology preserves cokernels, our conclusion is also true for a cokernel of a homomorphism of sheaves $\cO_{\cX}^{\oplus B} \arr \cO_{\cX}^{\oplus A}$. We can conclude that if the restriction of $F$ to the comma category $(\cX/\xi)$ has a finite presentation, (\ref{eq:base-change}) is an isomorphism.

\step 2  Now notice that the presheaf $\eta\mapsto \cO(V)\otimes_{\cO(T)}F(\xi)$ on the site $(\cX/\xi)$ is quasi-coherent (this is an easy check). Therefore, by Remark \ref{rmk:qc->sheaf}, it is a sheaf. As $\eta\mapsto F(\eta)$ is also a sheaf on the same site, to verify that the map $\cO(V)\otimes_{\cO(T)}F(\xi) \arr F(\eta)$ is an isomorphism is a local problem on $\xi$, and hence we can replace $\xi$ with a covering where $F$ has a presentation. This reduces the general case to the previous step.
\end{proof}

\begin{definition}
Let $\cX$ be a category fibered in groupoids on $\aff$. A sheaf $F$ of $\cO_{\cX}$-modules is called \emph{\fp} if it is \qc, and for each object $\xi$ of $\cX(T)$, the corresponding $\cO(T)$-module $F(\xi)$ is \fp.
\end{definition}

We will denote by $\FP \cX$ the full subcategory of $\catqcoh \cX$ consisting of \fp sheaves. We have the following obvious variant of Proposition~\ref{prop:equivalence-qc}.

\begin{proposition}\call{prop:equivalence-fp}
Let $\cX$ be a category fibered in groupoids on $\aff$ with an fpqc atlas $U \arr \cX$. Then we have equivalences of additive categories among the following.

\begin{enumeratea}

\itemref{1} The category $\FP{\cX}$ of \fp sheaves on $\cX$.

\itemref{2} The category of \fp sheaves on the site $\cX$ with the fpqc topology.

\itemref{3} The category of \fp sheaves on the site $\cX\sch$ with the fpqc topology.

\itemref{4} The category of \fp sheaves on the site $\cX\rep$ with the fpqc topology.

\end{enumeratea}
\end{proposition}


\subsection{Quasi-coherent sheaves on an infinite root stack}%
\label{sec:qc-on-irs}

In this short subsection we specialize the general theory to infinite root stacks, and collect some results that will be useful later.


We are going to need the following.

\begin{proposition}\call{prop:coherent-pushforward}
Let $\cX$ be a stack on $\aff$ with an fpqc atlas, and $\pi\colon \cR \arr \cX$ an \irs. Let $F$ be a \qc sheaf on $\cX$ and $G$ a \qc sheaf on $\cR$.
\begin{enumeratea}

\itemref{1} The pushforward $\pi_{*}G$ is a \qc sheaf on $\cX$.

\itemref{5} The pushforward $\pi_{*}\colon \catqcoh \cR \arr \catqcoh \cX$ is exact.

\itemref{2} If $G$ is \fp, then $\pi_{*}G$ is \fp.

\itemref{3} The canonical homomorphism $\cO_{\cX} \arr \pi_{*}\cO_{\cR}$ is an isomorphism.

\itemref{4} (Projection formula for infinite root stacks) The natural morphism $F \otimes_{\cO_{\cX}}\pi_{*}G \arr \pi_{*}(\pi^{*}F \otimes_{\cO_{\cR}}G)$ is an isomorphism.

\itemref{6} The unit homomorphism $F \arr \pi_{*}\pi^{*}F$ is an isomorphism.

\end{enumeratea}
\end{proposition}

\begin{proof}
All the statements are local in the fpqc topology, so we may assume that $\cX =X = \spec R$, where $R$ is a ring, and $\cR = X \times_{X_{P}} \cR_{P}$ for some \fs monoid $P$ and some map $X \arr X_{P}$. Set $A \eqdef R\otimes_{\ZZ[P]}\ZZ[P_{\QQ}]$, $G \eqdef \mmu_{\infty}(P)$, so that $\cR = [\spec A/G]$ by Corollary~\ref{cor:local-model-chart}. By Lemma~\ref{lem:right-quotients} we have $R = A^{G}$. The \qc sheaf $F$ corresponds to an $R$-module $M$, while by Proposition~\ref{prop:qc-group-actions} the \qc sheaf $G$ corresponds to a $G$-equivariant $A$-module $N$. The pushforward $\pi_{*}$ corresponds to the $R$-module $N^{G}$. Thus, the proposition can be translated into algebraic terms (we omit the translation of \refpart{prop:coherent-pushforward}{3}, which becomes tautological).

\begin{enumeratea}

\item[\refpart{prop:coherent-pushforward}{1}] If $R \arr S$ is a homomorphism of rings, the natural homomorphism $S\otimes_{R}N^{G} \arr (S\otimes_{R}N)^{G}$ is an isomorphism.

\item[\refpart{prop:coherent-pushforward}{5}] The functor $N \arrto N^{G}$ is exact.

\item[\refpart{prop:coherent-pushforward}{2}] If $N$ is \fp on $A$, then $N^{G}$ is \fp over $R$.

\item[\refpart{prop:coherent-pushforward}{4}] The natural homomorphism $M\otimes_{R}(N^{G}) \arr (M\otimes_{R}N)^{G}$ is an isomorphism.

\item[\refpart{prop:coherent-pushforward}{6}] The natural homomorphism $M \arr (M \otimes_{R}A)^{G}$ is an isomorphism.

\end{enumeratea}

Since $G$ is diagonalizable, every abelian group with an action of $G$ decomposes as a sum of eigenspaces; this is easily seen to imply \refpart{prop:coherent-pushforward}{1}, \refpart{prop:coherent-pushforward}{5}, \refpart{prop:coherent-pushforward}{4} and \refpart{prop:coherent-pushforward}{6}. We have only left to prove \refpart{prop:coherent-pushforward}{2}.

The character group $\widehat{G}$ equals $P\gr_{\QQ}/P\gr$, and we have a decomposition $A = \oplus_{\lambda \in P\gr_{\QQ}/P\gr}A_{\lambda}$. For each $\lambda \in P\gr_{\QQ}/P\gr$ let $A(\lambda)$ be the $A$-module $A$ with the $G$-action defined by the grading $A(\lambda)_{\mu} = A_{\lambda+\mu}$.  

Consider the decomposition $N = \oplus_{\lambda \in P\gr_{\QQ}/P\gr}N_{\lambda}$; the $A$ module $N$ is generated as an $A$-module by a finite number of homogeneous elements of degrees, say, $\lambda_{1}$, \dots~$\lambda_{r}$, so $N$ is a quotient of $\bigoplus_{i = 1}^{r}A(-\lambda_{i})$. The kernel of the surjection $\bigoplus_{i = 1}^{r}A(-\lambda_{i}) \arr N$ is also finitely generated, so we have a presentation
   \[
   \bigoplus_{j=1}^{s}A(-\mu_{j}) \arr \bigoplus_{i = 1}^{r}A(-\lambda_{i}) \arr N \arr 0\,.
   \]
By taking invariants we obtain a presentation
   \[
   \bigoplus_{j=1}^{s}A_{-\mu_{j}} \arr \bigoplus_{i = 1}^{r}A_{-\lambda_{i}}
   \arr N^{G} \arr 0\,.
   \]
Hence it is enough to show that each $A_{\lambda}$ is \fp as an $R$-module. Since $A_{\lambda} = R\otimes_{\ZZ[P]}\ZZ[P_{\QQ}]_{\lambda}$ it is enough to show that $\ZZ[P_{\QQ}]_{\lambda}$ is \fp as a $\ZZ[P]$-module. Since $\ZZ[P]$ is noetherian, it is enough to prove that $\ZZ[P_{\QQ}]_{\lambda}$ is finitely generated. {Moreover, we can identify $\ZZ[P_{\QQ}]_{\lambda}$ with the free abelian group generated by} $P_{\lambda} \eqdef \{p \in P_{\QQ} \mid [p] = \lambda \in P\gr_{\QQ}/P\gr\}$. If $v_{1}$, \dots~$v_{m}$ is a set of generators of $P$ giving a basis of $P\gr$, and
   \[
   S \eqdef
   \{a_{1}v_{1} + \dots + a_{m}v_{m} \mid a_{i}\in \QQ,\ 0 \leq a_{i} < 1\text{ for all }i\}\,,
   \]
it is clear that $\ZZ[P_{\QQ}]_{\lambda}$ is generated by $\{t^a\mid a\in S \cap P_{\lambda}\}$, as a $\ZZ[P]$-module: given a monomial $t^a$ with $a\in P_\lambda$, write $a=a_{1}v_{1} + \dots + a_{m}v_{m}$ with $a_i\geq 0$ for all $i$, and note that for every index we can subtract off the vector $v_i$ until the corresponding coefficient is non-negative and strictly less than $1$. We end up with an element of the form $t^p\cdot t^{a'}$, with $p\in P$ and $a'\in S\cap P_\lambda$, and this implies the claim made above. Now $S$ is bounded in $P\gr_{\QQ}$ and $P_\lambda$ is discrete, so $S \cap P_{\lambda}$ is finite. This completes the proof of Proposition~\ref{prop:coherent-pushforward}.
\end{proof}

%
%
%
%
%
%
%
%

Unfortunately, \fp sheaves on an \irs $\infroot{(X, A, L)}$ are not as well-behaved as one would wish. Specifically, we do not believe that they form an abelian category; the reason for this is that the ringed site $\infroot{(X, A, L)}$ is not coherent, in general, even over noetherian schemes. This issue is discussed in \cite{niziol-k-theory} in the context of \qc sheaves in the Kummer-flat site.

Recall that a sheaf of commutative rings $\cO$ on a site $\cC$ is \emph{coherent} when for any object $T$ of $\cC$, every finitely generated sheaf of ideals on the restriction of $\cO$ to $(\cC/T)$ is \fp. This condition ensures that all \fp sheaves of $\cO$-modules have the corresponding property, and that they form an abelian category.

\begin{example}\label{ex:non-coherent}
Let $P$ be the submonoid of $\QQ^{3}$ generated by $e_{1}\eqdef (1,0,0)$, $e_{2} \eqdef (0,1,0)$, $e_{3} \eqdef (0,0,1)$ and $e_{4}\eqdef e_{2} + e_{3} - e_{1}$; then $P$ is sharp, fine and saturated. The associated rational cone $P_{\QQ}\subseteq \QQ^{3}$ is given by the inequalities
\[
P_{\QQ}=\{(a_{1},a_{2},a_{3}) \in \QQ^{3}\mid a_{1}\geq 0, \; a_{2}\geq 0,\; a_{1}+a_{3}\geq 0,\; a_{2}+a_{3}\geq 0\}.
\]

Let $k$ be a field, and set $X = \spec k[P]$. Let $\cR$ be the \irs of the usual logarithmic structure on $X$; by Proposition~\ref{prop:local-model} we have that $\cR = [\spec k[{P_{\QQ}}]/\mmu_{\infty}(P)]$. We will show that $\cO_\cR$ is not coherent.

Set $R \eqdef k[P_{\QQ}]$, and let $x_{i}=x^{e_{i}} \in R$ be the element corresponding to $e_{i}$. The sheaves of ideals of $\cO_{\cR}$ correspond to ideals in $R$ that are homogeneous with respect to the natural $P_{\QQ}\gr/P\gr$-grading. Let $I = (x_{1}, x_{3}) \subseteq R$; we will show that $I$ is not \fp, by showing that the kernel
\[
K=\{(f_{1},f_{3}) \in R^{2}\mid x_{1}f_{1}+x_{3}f_{3}=0 \}
\]
of the presentation of $I$ is not finitely generated.

To check this, we will show that its image $J\subseteq R$ along the first projection $R^{2}\to R$ is not finitely generated. Since $J$ is a homogeneous ideal, it corresponds to an ideal $A\subseteq P_{\QQ}$, the set of degrees of non-zero elements in $J$.

We claim we can describe $A$ as the set of $a\in P_\QQ$ such that $a - e_{3}+e_{1} \in P_{\QQ}$, or equivalently as
\[
A=\{ (a_{1},a_{2},a_{3}) \in \QQ^{3}\mid a_{1}\geq 0, \; a_{2}\geq 0, \; a_{1}+a_{3}\geq 0,\; a_{2}+a_{3}\geq 1  \}.
\]
In fact, if $a \in A$ then there exist $f_{1}, f_{3} \in R$ such that $x_{1}f_{1}+x_{3}f_{3}=0$, with $f_{1}$ of degree $a$. Note that necessarily $f_{3}\neq 0$, and call $b$ the degree of $f_{3}$. Then we conclude that $a+e_{1}=b+e_{3}$, and consequently $a - e_{3}+e_{1}$ is in $P_{\QQ}$.
Conversely if $a - e_{3}+e_{1} \in P_{\QQ}$ and $a \in P_{\QQ}$, we have that $x_{1}x^{a}-x_{3}x^{a-e_{3}+e_{1}}=0$ (where as usual $x^{p}$ denotes the element of $k[P_{\QQ}]$ corresponding to $p \in P_{\QQ}$), so $a \in A$.

Now consider
\begin{align*}
A_{0}   & =   \{ a=(a_{1},a_{2},a_{3}) \in A\mid a_{1}=0,\; a_{2}+a_{3}=1 \}\\
 & = \{ (0,a_{2},a_{3})\in \QQ^{3}\mid a_{2}\geq 0, \; a_{3}\geq 0, a_{2}+a_{3}=1\}.
\end{align*}
It is easy to check that $a+b \in A_{0}$ implies $a=0$ for $a \in P_{\QQ}$ and $b \in A$, and this says that any set of generators of $A$ as an ideal of $P_{\QQ}$ must contain all elements of
$A_{0}$, and thus must be infinite. In conclusion the ideal $J$ is not finitely generated.

\end{example}

This kind of unpleasantness does not happen when the logarithmic structure is simplicial, in the following sense.

\begin{definition}
A monoid $P$ is \emph{simplicial} if it is fine, saturated and sharp, and $P_{\QQ}$ is generated by linearly independent vectors in $P\gr_{\QQ}$.

A \df structure $(A,L)$ on a scheme is simplicial if all the geometric stalks of $A$ are simplicial.
\end{definition}

If $(A,L)$ is simplicial, then étale-locally on the scheme $X$ there exists a Kato chart $P \arr \cO(X)$, where $P$ is a simplicial monoid.

\begin{proposition}
If $(A, L)$ is a simplicial \df structure on a locally noetherian scheme $X$, the structure sheaf on the \irs $\infroot{(A,L)}$ is coherent.
\end{proposition}

\begin{proof}
The problem is local in the étale topology, so we may assume that $X = \spec R$ is affine, and that the \df is generated by a homomorphism of monoids $P \arr R$, where $P$ is simplicial. By Proposition~\ref{prop:local-model} we have
   \[
   \infroot{(A,L)} = [\spec (R \otimes_{\ZZ[P]} \ZZ[P_{\QQ}])/ \mmu_{\infty}(P)]\,;
   \]
it is enough to show that the ring $R \otimes_{\ZZ[P]} \ZZ[P_{\QQ}]$ is coherent. But $P_{\QQ}$ is isomorphic to the monoid $\NN^{r}_{\QQ}$ of $r$-tuples of non-negative rational numbers; so $\ZZ[P_{\QQ}]$ is an inductive limit $\indlim_{n}\ZZ[\tfrac{1}{n}\NN^{r}]$ of finitely generated $\ZZ$-algebras with flat transition maps. Hence $R \otimes_{\ZZ[P]} \ZZ[P_{\QQ}]$ is a limit of of finitely generated $R$-algebras with flat transition maps, and it is an easy exercise to show that an inductive limit of coherent rings with flat transition maps is coherent.
\end{proof}

See also the discussion in \cite[\S 3]{niziol-k-theory}.


\section{Logarithmic structures from \irss}\label{sec:reconstruction}

Let us show how one can produce a \df structure from an \irs $\pi\colon \cR \arr X$. Here we describe our results and state them without proving them right away. We collect most of the proofs in \ref{subsec:proofs}, after a discussion of the fundamental notion of \emph{infinite quotients} in a fine saturated monoid.

Denote by $\div_{\cR\et}$ the fibered category over the small étale site $X\et$, whose objects are pairs $(U \arr X, \Lambda)$, where $U \arr X$ is an étale map, and $\Lambda$ is an object of $\div(\cR_{U})$. Note that we can also describe $\div_{\cR\et}$ as the pushforward {$\pi_*\Div_{\cR}$}, and there is an obvious pullback functor $\div_{X\et} \arr \div_{\cR\et}$.

By definition, and by descent theory, given a \df structure $(A, L)$ on $X$, and set $\cR \eqdef \infroot{(A,L)}$, there exist a \sm functor $\widetilde{L}\colon A_{\QQ} \arr \div_{\cR\et}$, and an isomorphism of \sm functors between the restriction of $\widetilde{L}$ to $A$, and the composite of $L$ with the pullback $\div_{X\et} \arr \div_{\cR\et}$.

\begin{definition}\label{def:infinite.root}
Let $\pi\colon \cR\arr X$ be an \irs. Consider the \sm fibered category $\cA_{\cR} \arr X\et$ defined as follows. For each étale map $U \arr X$, the objects of $\cA_{\cR}(U)$ are of the form $(\Lambda_{n},\alpha_{m,n})$, where:

\begin{enumeratea}


\item For each positive integer $n$, $\Lambda_n$ is an object of $\div \cR_{U}$.

\item The object $\Lambda_{1}$ of $\div \cR_{U}$ is isomorphic to a pullback $\pi^{*}\Lambda$, where $\Lambda$ is an object of $\div U$.

\item For each $m \mid n$, $\alpha_{m,n}\colon \Lambda_n^{\otimes (n/m)} \simeq \Lambda_{m}$ is an isomorphism in $\div \cR_{U}$.

\item Suppose that $p$ is a point of $X$; denote by $\cR_{p}$ the fiber of $\cR$ over $p$. If $n$ is sufficiently divisible and $\Lambda_{n} = (L_{n}, s_{n})$, then the restriction of $s_{n}$ to $\cR_{p}$ is nonzero.

\end{enumeratea}

We require the isomorphisms $\alpha_{m,n}$ to be subject to the following compatibility conditions.

\begin{enumeratei}

\item $\alpha_{n,n} = \id_{\Lambda_n}$ for any $n$.

\item if $m \mid n$ and $n \mid p$, then
   \[
   \alpha_{m,p} = \alpha_{m,n}^{\otimes(p/n)} \circ \alpha_{n,p}\colon \Lambda_{p}^{\otimes(p/n)}
   \simeq \Lambda_{m}.
   \]

\end{enumeratei}

The arrows $(\Lambda_{n}, \alpha_{m,n}) \arr (\Lambda'_{n}, \alpha'_{m,n})$ are given by isomorphisms of invertible sheaves with a section $\Lambda_{n} \simeq \Lambda'_{n}$ that are compatible with the $\alpha_{m,n}$'s.

The fibered structure is obtained from the evident pseudo-functor structure.

We call the objects of $\cA_{\cR}(U)$ \emph{infinite roots}.

\end{definition}

{The point of the previous definition is that if $\cR=\infroot{(A,L)}$, then, as we explain later, to every local section of the sheaf $A$ we can associate an infinite root on $\cR$, and every infinite root arises in this way. Hence this construction will recover the \df structure that we start with, and if $\cR$ is only an infinite root stack in the sense of Definition \ref{def:airs}, it will produce a \df structure $(A_\cR,L_\cR)$ on $X$ ($A_\cR$ will be the sheaf of isomorphism classes of objects of the fibered category $\cA_\cR$) such that $\cR\cong \infroot{(A_\cR,L_\cR)}$.

\begin{remark}\label{rmk:base.root}
Notice that the object $\Lambda$ of $\div U$ and the isomorphism $\phi\colon \pi^{*}\Lambda \simeq \Lambda_{1}$ are in fact unique, up to a unique isomorphism. What's more, there is a canonical choice for $\Lambda$ and $\phi$.

The point is this. Set $\Lambda_{1} = (L_{1}, s_{1})$ and $\Lambda = (L, s)$. Then the isomorphism $\phi\colon  \pi^{*}L \simeq L_{1}$ corresponds by adjunction to a homomorphism $L \arr \pi_{*}L_{1}$; by Proposition \refall{prop:coherent-pushforward}{6} this is an isomorphism, and carries $s$ into $s_{1} \in L_{1}(U) = (\pi_{*}L)(U)$. This gives a canonical isomorphism of $\Lambda$ with $(\pi_{*}L_{1}, s_{1})$.

We call the object $(\pi_{*}L_{1}, s_{1})$ of $\div U$ the \emph{base} of the infinite root $(\Lambda_{n}, \Lambda_{m,n})$; we think of an infinite root as an infinite sequence of roots of its base. 

We obtain a \sm functor $\cA_{\cR}(U) \arr \div U$ by sending each infinite root into its base.
\end{remark}

\begin{remark}
It is immediate from the definition that the formation of $\cA_{\cR}$ commutes with arbitrary base change on $X$.
\end{remark}

The \sm structure is given by tensor product; this is well defined thanks to the following lemma.

\begin{lemma}\label{lem:tensor-product-roots}
$(\Lambda_{n}, \alpha_{m,n})$ and $(\Lambda'_{n}, \alpha'_{m,n})$ be infinite roots in an \irs $\cR$. Then the tensor product
   \[
   (\Lambda_{n} \otimes \Lambda'_{n}, \alpha_{m,n} \otimes \alpha'_{m,n})
   \]
is also an infinite root.
\end{lemma}

\begin{proof}
Here the essential point is to show that if we set $\Lambda_{n} = (L_{n}, s_{n})$ and $\Lambda'_{n} = (L'_{n}, s'_{n})$, then the restriction of $s_{n} \otimes s'_{n}$ to any geometric fiber is nonzero for sufficiently divisible $n$. This follows from the second statement in Lemma~\ref{lem:invertible-image} below.
\end{proof}

The proof of the following proposition, of Proposition \ref{prop:equality-sheaves} and of Theorem \ref{thm:equivalence} will be postponed (see Section \ref{subsec:proofs}).

\begin{proposition}\label{prop:A_R-1-category}
Let $\cR \arr X$ be an \irs. Then the symmetric monoidal category $\cA_{\cR}$ is fibered in equivalence relations (i.e. equivalent to a presheaf).

\end{proposition}

Hence by passing to isomorphism classes we get a sheaf of monoids $A_{\cR}\colon X\et \arr \catmon$, and the projection $\cA_{\cR} \arr A_{\cR}$ is an equivalence. By choosing a \sm quasi-inverse $A_{\cR} \arr \cA_{\cR}$ and composing with the \sm functor $\cA_{\cR} \arr \div_{X\et}$ that sends {an infinite root to its base}, we obtain a \sm $L_{\cR}\colon A_{\cR} \arr \div_{X\et}$, unique up to a unique isomorphism.

As with the fibered category $\cA_{\cR}$, the formation of $L_{\cR}$ and $A_{\cR}$ commutes with arbitrary base change on $X$.

\begin{proposition}\label{prop:fine-saturated}
The sheaf $A_{\cR}$ is \fs.
\end{proposition}

\begin{proof}
Since being \fs is a local condition in the étale topology, and étale-locally $\cR$ comes from a \fs \df structure, this follows from Proposition~\ref{prop:equality-sheaves} below.
\end{proof}

Hence, from an \irs $\cR \arr X$ we obtain a \fs \df structure $(A_{\cR}, L_{\cR})$.

Suppose that $(A, L)$ is a \df structure on $X$, and $\cR = \infroot{(A, L)}$. Let $U \arr X$ be an étale map, and $a \in A(U)$. Then we obtain an object $L(a)$ of $\div U$. Furthermore, for each positive integer $n$ we also obtain an object $\widetilde{L}(a/n) \in \div \cR_{U}$. The fact that the functor is \sm gives, for each $m \mid n$, isomorphisms $\alpha_{m,n}\colon \widetilde{L}(a/n)^{\otimes (n/m)} \simeq \widetilde{L}(a/m)$ in $\div \cR_{U}$. Furthermore, the isomorphism between the restriction of $L$ to $A$, and the composite of $L$ with the pullback $\div U \arr \div \cR_{U}$ yields an isomorphism $\phi\colon \Lambda_{1} = \widetilde{L}(a) \simeq L(a)$. This gives a \sm functor $A \arr \cA_{\cR}$; by definition, the composite of $A \arr \cA_{\cR}$ with $\cA_{\cR} \arr \div_{X\et}$ is precisely $L$.

\begin{proposition}\label{prop:equality-sheaves}
Suppose that $(A, L)$ is a \fs \df structure on $X$, and set $\cR \eqdef \infroot{(A,L)}$. Then the composite $A \arr \cA_{\cR} \arr A_{\cR}$ is an isomorphism.
\end{proposition}

\begin{corollary}\label{cor:df.infinite}
The \df structures $(A, L)$ and $(A_{\cR}, L_{\cR})$ are isomorphic. \qed
\end{corollary}

Given an \irs $\cR \arr X$, let us produce a functor $\cR \arr \infroot{(A_{\cR}, L_{\cR})}$. Let $f\colon T \arr \cR$ be a morphism; we need to construct a map $T \arr\infroot{(A_{\cR}, L_{\cR})}$, i.e. an extension $(f^{*}A_{\cR})_{\QQ} \arr \div_{T\et}$ of the \df structure $f^{*}L_{\cR}\colon f^{*}A_{\cR} \arr \div_{T\et}$. Call $f^{-1}A_{\cR}$ the pullback presheaf on $T\et$; its sections on an étale map $V \arr T$ are colimits $\indlim A_{\cR}(U)$, where the colimit is taken over all factorizations $V \arr U \arr X$, with $U \arr X$ étale, of the composite $V \arr T \arr \cR \arr X$. The sheafification of the presheaf $B_{T}$ on $T\et$ sending $V$ to $(f^{-1}A_{\cR})(V)_{\QQ}$ is the sheaf $(f^{*}A_{\cR})_{\QQ}$; by \cite[Proposition 3.3]{borne-vistoli1}, every \sm functor $B_{T}\arr \div_{T\et}$ extends uniquely to a \sm functor $(f^{*}A_{\cR})_{\QQ} \arr \div_{T\et}$. 

Consider the filtered category $I_{V}$ defined as follows. The objects are pairs $(m, V \arr U \arr X)$, where $m$ is a positive integer and $V \arr U \arr X$ is a factorization of the composite $V \arr T \arr \cR \arr X$. An arrow $\phi\colon (m, V \arr U \arr X) \arr (n, V \arr U' \arr X)$ exists only when $m \mid n$, in which case it consists of a morphism $\phi\colon U \arr U'$ such that the diagram
   \[
   \begin{tikzcd}[row sep = tiny]
   {}&U \ar{rd}\ar{dd}{\phi}&\\
   V\ar{ur}\ar{dr}&&X\\
   &U'\ar{ur}
   \end{tikzcd}
   \]
commutes. Composition is the obvious one.

There is a lax $2$-functor from $I_{V}\op$ into the $2$-category of symmetric monoidal categories, sending each $(m, V \arr U \arr X)$ into $\cA_{\cR}(U)$, and each morphism $\phi\colon (m, V \arr U \arr X) \arr (n, V \arr U' \arr X)$ into the composite of the pullback $\phi^{*}\colon \cA_{\cR}(U') \arr \cA_{\cR}(U)$ with the functor $\cA_{\cR}(U) \arr \cA_{\cR}(U)$ given by raising to the $(n/m)\th$ power. We have a canonical equivalence of \sm categories between $\indlim_{I_{V}}\cA_{\cR}(U)$ and the monoid $B_{T}(V)$.

There is also a \sm functor $\indlim_{I_{V}}\cA_{\cR}(U) \arr \div V$ that sends an object $(\Lambda_{n}, \alpha_{m,n}\bigr)$ over $(m, V \arr U \arr X)$ to $h^{*}\Lambda_{m}$, where $h \colon V \arr \cR_{U}$ is the morphism induced by $V \arr U$ and the composite $V \arr T \arr \cR$. By composing this with a quasi-inverse of the equivalence $\indlim_{I_{V}}\cA_{\cR}(U) \arr B_{T}(V)$ we obtain a \sm functor $B_{T}(V) \arr \div V$. This induces the desired \sm functor $B_{T} \arr \div_{T\et}$.

This way we defined a functor $\cR \arr \infroot{(A_{\cR}, L_{\cR})}$.

\begin{theorem}\label{thm:equivalence}
The functor $\cR \arr \infroot{(A_{\cR}, L_{\cR})}$ above is an equivalence.
\end{theorem}


\subsection{Infinite quotients in sharp \fs monoids}\label{subsec:infinite-quotients}

To prove the statements of the previous section we need to investigate the notion of an \emph{infinite quotient} in a sharp \fs monoid.

Let $P$ be a sharp \fs monoid, and assume that $P\gr$ has rank $r$. In other words we have $P\gr\cong \ZZ^{r}$, and consequently $P\gr_{\QQ}\cong \QQ^{r}$. Moreover $P\grq /P\gr = P\gr \otimes (\QQ/\ZZ)$ is isomorphic to $(\QQ/\ZZ)^{r}$. We consider $P\grq$ as a topological space via the usual metric topology on $\QQ^{r}$.

Set
\[
\widecheck P = \projlim_{n} (P\grq /P\gr)[n] = P\gr \otimes \widehat{\ZZ} \cong \widehat{\ZZ}^{r},
\]
where with the square brackets we denote the $n$-torsion, the map $(P\grq /P\gr)[n]\to (P\grq /P\gr)[m]$ for $m\mid n$ is given by multiplication by $n/m$, and $\widehat{\ZZ}$ denotes as usual the profinite completion of $\ZZ$. An element of $\widecheck P$ consists of a collection $\{\lambda_{n}\}$ of elements of $P\grq/P\gr$ such that $\lambda_{1} = 0$, and $(n/m)\lambda_{n} = \lambda_{m}$ whenever $m \mid n$.

Set
   \[
   \Delta_{P} = P_{\QQ}\setminus (P^{+} + P_{\QQ})\,,
   \]
where $P^{+}=P\setminus \{0\}$. Since $\Delta_{P}$ is the complement of an ideal in $P_{\QQ}$, it has the property that if $\gamma$, $\delta \in P_{\QQ}$ and $\gamma + \delta \in \Delta_{P}$, then $\gamma$ and $\delta$ are in $\Delta_{P}$. Hence, if $n$ is a positive integer, $\gamma \in P\grq$, and $n\gamma \in \Delta_{P}$, then $\gamma \in \Delta_{P}$.

The set $\Delta_{P}$ is clearly bounded in $P\grq$. Also, if $v_{1}$, \dots~$v_{m}$ are the indecomposable elements of $P$ we have
   \[
   \Delta_{P} = P_{\QQ} \setminus \bigcup_{i=1}^{m} (v_{i} + P_{\QQ})\,;
   \]
since $P_{\QQ}$ is closed in $P\grq$, we have that $\Delta_{P}$ is open in $P_{\QQ}$.

We set
   \[
   \Delta_{P}^{0} = \bigl\{\gamma \in \Delta_{P} \mid (\gamma + P\gr) \cap \Delta_{P}
   = \{\gamma\}\bigr\}\,.
   \]
Figure~1 below illustrates what $\Delta_{P}$ and $\Delta_{P}^{0}$ look like in a simple but significant example.
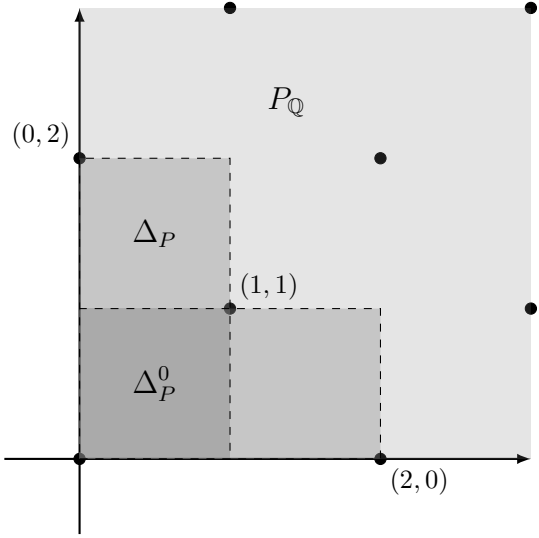
\begin{figure}[ht]
  \centering
\begin{tikzpicture}
    \coordinate (Origin)   at (0,0);
    \coordinate (XAxisMin) at (-1,0);
    \coordinate (XAxisMax) at (6,0);
    \coordinate (YAxisMin) at (0,-1);
    \coordinate (YAxisMax) at (0,6);
    \draw [thick,-latex] (XAxisMin) -- (XAxisMax);
    \draw [thick,-latex] (YAxisMin) -- (YAxisMax);
    \foreach \x in {0,...,1}{
      \foreach \y in {0,...,1}{
        \node[draw,circle,inner sep=1.5pt, fill] at (4*\x,4*\y) {};
      }
    }
    \foreach \x in {1,3}{
      \foreach \y in {1,3}{
        \node[draw,circle,inner sep=1.5pt, fill] at (2*\x,2*\y) {};
      }
    }
    \fill[fill=gray, fill opacity=0.2] (Origin)
        rectangle (6,6);
        \filldraw[fill=gray, dashed, fill opacity=0.3]  (0,0) -- (4,0) -- (4,2) -- (2,2) -- (2,4) -- (0,4) -- cycle;
           \filldraw[fill=gray, dashed, fill opacity=0.4]  (0,0) -- (2,0) -- (2,2) -- (0,2) -- cycle;
  \fill (4,0)  node[below right] {$\mathlarger{(2,0)}$};
   \fill (0,4)  node[above left] {$\mathlarger{(0,2)}$};
    \fill (2,2)  node[above right] {$\mathlarger{(1,1)}$};
      \fill (2.75,4.75)  node {$\mathlarger{\mathlarger{P_\QQ}}$};
      \fill (1,3)  node {$\mathlarger{\mathlarger{\Delta_P}}$};
       \fill (1,1)  node {$\mathlarger{\mathlarger{\Delta^0_P}}$};
\end{tikzpicture}
\caption{The cone $P_\QQ$ and the open loci $\Delta_P$ and $\Delta_P^0$ for the monoid $P=\langle (0,2), (1,1), (2,0) \rangle\subseteq \ZZ^2$.}
  \label{figure}
\end{figure}

By definition, the restriction of the projection $P\grq \arr P\grq/P\gr$ to $\Delta_{P}^{0}$ is injective.

\begin{lemma}\label{lem:neighborhood-0}
The set $\Delta_{P}^{0}$ is a neighborhood of $0$ in $P_{\QQ}$.
\end{lemma}

\begin{proof}
It is easy to see that $0 \in \Delta_{P}^{0}$. Also, we have
   \[
   \Delta_{P}^{0} = \bigcap_{\gamma \in P\gr \setminus \{0\}}
   \bigl(\Delta_{P} \setminus (\gamma + \Delta_{P})\bigr)\,;
   \]
but $\Delta_{P}$ is bounded, so there exists a finite number of $\gamma \in P\gr \setminus \{0\}$ such that $\Delta_{P} \cap (\gamma + \Delta_{P}) \neq \emptyset$. So it is enough to prove that $\Delta_{P} \setminus (\gamma + \Delta_{P})$ is a neighborhood of $0$ in $P_{\QQ}$ for all $\gamma \in P\gr \setminus \{0\}$.

If $\gamma \in -P_{\QQ}\cap (P\gr\setminus \{0\})=-P\setminus \{0\}$ we have
   \[
   \Delta_{P} \cap (\gamma + \Delta_{P}) = \gamma +
   \bigl(\Delta_{P} \cap (-\gamma + \Delta_{P})\bigr) = \emptyset\,
   \]
{since $-\gamma \in P\setminus \{0\}=P^+$, and by definition of $\Delta_P$.}
Otherwise, we have $0 \notin \gamma+ P_{\QQ}$, so $\Delta_{P} \setminus (\gamma+ P_{\QQ})$ is neighborhood of $0$ in $P_{\QQ}$, and $\Delta_{P} \setminus (\gamma+ P_{\QQ}) \subset \Delta_{P} \setminus (\gamma+ \Delta_{P})$. This finishes the proof.
\end{proof}

There is a group homomorphism $P\gr \arr \widecheck P$ sending each $p \in P\gr$ into the element $\infquot p \eqdef \{[p/n]\} \in \widecheck P$. This is easily seen to be injective. Consider the restriction $P \arr \widecheck P$.

We need to recognize elements in $\widecheck P$ that come from $P$. To do so, we introduce the following definition, in which we think of an element $\lambda \in P\grq/P\gr$ as a coset of $P\gr$ in $ P\grq$.

\begin{definition}\label{def:infinite-quotient}
An element $\{\lambda_{n}\}$ of $\widecheck P$ is \emph{an infinite quotient} if there exists a positive integer $m_{0}$ such that for any positive integer $j$ there exists a sequence $\gamma_{1}$, \dots,~$\gamma_{j}$ of elements of $\lambda_{jm_{0}} \cap \Delta_{P}$ such that $\gamma_{1} + \dots + \gamma_{j} \in \Delta_{P}$.

We call such an $m_{0}$ a \emph{characteristic integer} for $\{\lambda_{n}\}$.

We denote the set of infinite quotients in $P$ by $\infquot P$.
\end{definition}

The following is straightforward.

\begin{lemma}\label{lem:multiple-characteristic}
If $\{\lambda_{n}\}$ is an infinite quotient and $m_{0}$ is a characteristic integer, then any positive multiple of $m_{0}$ is also a characteristic integer for $\{\lambda_{n}\}$.\qed
\end{lemma}

Let us note that the image $\infquot p = \{[p/n]\}$  of an element $p \in P$ is an infinite quotient. In fact, by Lemma~\ref{lem:neighborhood-0} there exists $m_{0}$ such that $p/jm_{0} \in \Delta_{P}$ for all positive integers $j$; then it is enough to take $\gamma_{1} = \dots = \gamma_{j} = p/jm_{0}$.

The following proposition says in particular that the converse holds, i.e. infinite quotients in $\widecheck{P}$ correspond exactly to elements of $P$.

\begin{proposition}\call{prop:description-infinite-quotients} \hfil
\begin{enumeratea}

\itemref{1} Let $\{\lambda_{n}\}$ be an infinite quotient in $P$. For every sufficiently divisible $n$ we have $\lambda_{n} = [\gamma]$ for some $\gamma \in \Delta_{P}^{0}$.

\itemref{3} Let $\{\lambda_{n}\}$  and $\{\lambda'_{n}\}$ be infinite quotients in $P$. For every sufficiently divisible $n$ we have $\lambda_{n} = [\gamma]$ and $\lambda'_{n} = [\gamma']$ with $\gamma + \gamma' \in \Delta_{P}^{0}$.

\itemref{2} The image of $P$ in $\widecheck P$ is precisely $\infquot{P}$, hence the map $P\to \infquot P$ is an isomorphism.

\end{enumeratea}
\end{proposition}


\begin{proof}

Let us show that there is a norm $\abs{-}$ on $P\grq$ with the property that $\abs{\gamma+\delta} = \abs{\gamma} + \abs{\delta}$ for any $\gamma$ and $\delta$ in $P_{\QQ}$. For this, notice that there is a basis $v_{1}$, \dots~$v_{r}$ of $P\grq \cong \QQ^{r}$ with the property that every vector in $P_{\QQ}$ has non-negative coordinates (in fact, since $P$ is sharp the cone in the dual space $(P\gr_{\QQ})^{\vee}$ that is dual to $P_{\QQ}$ has nonempty interior, so it contains a basis of $P_{\QQ}^{\vee}$, and the dual basis in $P\gr_{\QQ}$ has this property). Then the norm defined by $\abs{x_{1}v_{1} + \dots + x_{r}v_{r}} \eqdef \abs{x_{1}} + \dots + \abs{x_{n}}$ has this property.

Now let $m_{0}$ be a characteristic integer for $\{\lambda_{n}\}$ and pick a positive real number $\epsilon$ such that every $\gamma \in P_{\QQ}$ with $\abs{\gamma} \leq \epsilon$ is in $\Delta_{P}^{0}$. Since $\Delta_{P}$ is bounded in $P\grq$ we can also choose $N > 0$ with the property that $N\epsilon$ is larger than the diameter of $\Delta_{P}$. If $n$ is divisible by $m_{0}$ and $n/m_{0} > N$, then we claim that we there exists $\gamma \in \lambda_{n}$ with $\abs{\gamma} \leq \epsilon$, so that $\gamma \in \Delta_{P}^{0}$, which will conclude the proof of part~\refpart{prop:description-infinite-quotients}{1}.

Write $n = jm_{0}$; it follows immediately from Definition~\ref{def:infinite-quotient} that $\lambda_{n} = [\gamma]$ for some $\gamma \in \Delta_{P}$. If $\abs{\gamma}>\epsilon$ for all $\gamma \in \lambda_{n} \cap \Delta_{P}$, then for any sequence $\gamma_{1}$, \dots~$\gamma_{j}$ in $\lambda_{n} \cap \Delta_{P}$ we have
   \[
   \abs{\gamma_{1} + \dots + \gamma_{j}} = \abs{\gamma_{1}} + \dots + \abs{\gamma_{j}}
   > j\epsilon > N\epsilon\,,
   \]
hence $\gamma_{1}$, \dots~$\gamma_{j} \notin \Delta_{P}$, and this contradicts Definition~\ref{def:infinite-quotient}.

\smallskip

For \refpart{prop:description-infinite-quotients}{2}, we already showed that $\infquot p \in \infquot P$.

Conversely, suppose that $\{\lambda_{n}\} \in \infquot{P}$. Let $m_{0}$ be a characteristic integer for $\{\lambda_{n}\}$; from Lemma~\ref{lem:multiple-characteristic} and part~\refpart{prop:description-infinite-quotients}{1}, we see that we can take $m_{0}$ such that $\lambda_{n} = [\gamma_{n}]$ with $\gamma_{n} \in \Delta_{P}^{0}$ for all $n$ divisible by $m_{0}$. Then by definition of infinite quotients and the fact that $\gamma_{jm_{0}}$ is the \emph{only} representative of $\lambda_{jm_{0}}$ in $\Delta_{P}$ for all positive $j$ (by definition of $\Delta_P^0$), we have $j\gamma_{jm_{0}} \in \lambda_{m_{0}} \cap\Delta_{P}$ for all $j > 0$, and hence $\gamma_{m_{0}} = j\gamma_{jm_{0}}$ for all $j > 0$. Setting $q \eqdef m_{0}\gamma_{m_{0}}$, we have $\gamma_{n} = q/n$ for all $n$ divisible by $m_{0}$, which implies that for all $n$ we have $\lambda_{n}=[q/n]$, so that $\{\lambda_{n}\} = \infquot q$.

\smallskip

Part~\refpart{prop:description-infinite-quotients}{3} follows immediately from \refpart{prop:description-infinite-quotients}{1} and \refpart{prop:description-infinite-quotients}{2}.
\end{proof}

\subsection{Picard groups of infinite root stacks over geometric points}\label{subsec:picard-groups}

Next we need some results on the Picard group of an infinite root stack over a geometric point. We will show that it can be identified with the quotient $P\gr_{\QQ}/P\gr$, where $P$ is the stalk of the sheaf $A$ at the point.

We recall our notations for logarithmic points.

\begin{notation}\label{notation-geometric-fiber}
If $k$ is a field and $P$ is a sharp \fs monoid, we will denote by $(A_{P}, L_{P})$ the \df structure on $\spec k$, where $A_{P}$ is the constant sheaf of monoids on $(\spec k){\et}$ corresponding to $P$, and $L_{P}\colon A \arr \div_{(\spec k)\et}$ corresponds to the homomorphism $\Lambda\colon P \arr k$ that sends $0$ into $1$ and everything else into $0$.

We also set $\infroot{P/k} \eqdef \infroot{(A_{P}, L_{P})}$. 
\end{notation}

As explained in Section \ref{sec:abstract-irs}, the \irs $\infroot{P/k}$ is isomorphic to the fibered product $\spec k \times_{\spec k[P]}  \bigl[\spec k[P_{\QQ}]/\mmu_{\infty}(P)\bigr]$, where $\spec k \arr \spec k[P]$ corresponds to the ring homomorphism $k[P] \arr k$ determined by $\Lambda$. We also have
   \[
   \infroot{P/k} = \bigl[\spec(k[P_{\QQ}]/(P^{+}))/\mmu_{\infty}(P)\bigr]
   \]
where the action of $\mmu_{\infty}(P)$ on $\spec(k[P_{\QQ}]/(P^{+}))$ is determined by the natural $P\grq/P\gr$-grading on $k[P_{\QQ}]/(P^{+})$. Moreover the reduced substack $(\infroot{P/k})_{\red}$ is the classifying stack $\cB_{k} \mmu_{\infty}(P) = \bigl[\spec k/ \mmu_{\infty}(P)\bigr]$.

Set $R \eqdef k[P_{\QQ}]/(P^{+})$, and note that for the natural $P\grq/P\gr$-grading, we have $R_{\lambda} = \oplus_{\gamma \in \lambda \cap \Delta_{P}} R_{\gamma}$, for $\lambda \in P\grq/P\gr$. 
Invertible sheaves on $\infroot{P/k}$ correspond to $P\grq/P\gr$-graded invertible modules on $R$; this gives a very concrete description of $\pic(\infroot{P/k})$. There is a natural homomorphism
   \[
   P\grq/P\gr = \hom(\mmu_{\infty}(P), \gm\bigr) \arr \pic(\infroot{P/k})
   \]
that sends $\gamma \in P\grq/P\gr$ into the graded $R$-module $R(\lambda)$, where $R(\lambda) = R$ as an $R$-module, but the $P\grq/P\gr$-grading is defined by $R(\lambda)_{\mu} = R_{\lambda + \mu}$.

Since $R$ is the inductive limit of the local artinian rings $k[\frac{1}{n}P]/(P^{+})$, every invertible module on $R$ is trivial; hence, every $P\grq/P\gr$-graded invertible module on $R$ is of the form $R(\lambda)$ for $\lambda \in P\grq/P\gr$. So the homomorphism above is surjective. Since $R(\lambda) \otimes_{R}k = k(\gamma)$ is a $P\grq/P\gr$-graded vector space, we see that $R(\lambda) \cong R(\mu)$ if and only if $\lambda = \mu$, and the homomorphism is also injective.

Let us record this in a lemma.

\begin{lemma}\label{lem:local-picard}
The natural homomorphism
   \[
   P\grq/P\gr = \hom\bigl(\mmu_{\infty}(P), \gm\bigr) \arr \pic(\infroot{P/k})
   \]
is an isomorphism.\qed
\end{lemma}

Furthermore, if $\lambda = [L] \in  \pic(\infroot{P/k}) = P\grq/P\gr$, we have $\H^{0}(\infroot{P/k}, L) = R(\lambda)_{0} = R_{\lambda}$. So $\dim_{k}\H^{0}(\infroot{P/k}, L) = \sharp(\lambda \cap \Delta_{P})$.

Let $(\Lambda_{n}, \alpha_{m,n}\bigr)$ and $(\Lambda'_{n}, \alpha'_{m,n}\bigr)$ be infinite roots on $\infroot{P/k}$, and set $\Lambda_{n} = (L_{n}, s_{n})$ and $\Lambda'_{n} = (L'_{n}, s'_{n})$. The following will be used later.

\begin{lemma}\label{lem:dim-1}
For sufficiently divisible $n$, we have $\dim_{k}\H^{0}(\infroot{P/k}, L_{n})= 1$, and the multiplication map
   \[
   \H^{0}(\infroot{P/k}, L_{n}) \otimes_{k} \H^{0}(\infroot{P/k}, L'_{n}) \arr
   \H^{0}(\infroot{P/k}, L_{n} \otimes L'_{n})
   \]
is an isomorphism.
\end{lemma}

\begin{proof}
This follows from \refall{prop:description-infinite-quotients}{1} and \refpart{prop:description-infinite-quotients}{3}.
\end{proof}

\subsection{Proofs}\label{subsec:proofs}

Here we collect the proofs of the statements that we left without proof in the previous sections.

%
%

\begin{lemma}\label{lem:invertible-image}
Let $\pi\colon \cR \arr X$ be an \irs, and let $(\Lambda_{n}, \alpha_{m,n})$ be in infinite root on $\cR$. Set $\Lambda_{n} = (L_{n}, s_{n})$. 

If $X$ is quasi-compact, then for sufficiently divisible $n$ the sheaf $\pi_{*}L_{n}$ is an invertible sheaf on $X$, and the section $s_{n} \in \H^{0}(X, \pi_{*}L_{n})$ does not vanish anywhere.

Furthermore, let $(\Lambda'_{n}, \alpha'_{m,n})$ be another infinite root on $\cR$, and set $\Lambda'_{n} = (L'_{n}, s'_{n})$. Then for sufficiently divisible $n$ the multiplication map
   \[
   \pi_{*}L_{n} \otimes \pi_{*}L'_{n} \arr \pi_{*}(L_{n}\otimes L'_{n})
   \]
is an isomorphism.
\end{lemma}

\begin{proof}
{Since $X$ is quasi-compact, there is a finite étale covering $X_{i} \arr X$ such that for all $i$, the stack $\cR_{X_{i}}$ is the pullback of the canonical \irs on $X_{P_{i}}$ for some fine saturated monoid $P_{i}$. Since the covering is finite, the formation of $\cA_{\cR}$ commutes with base change on $X$, and the pushforward $\pi_{*}$ also commutes with base change, we may assume that $X = X_{P_{i}}$. In particular, we may assume that $X$ is affine and noetherian.}

Each $L_{n}$ is invertible on $\cR$ and $\pi_{*}\cO_{\cR} = \cO_{X}$, so we see that the annihilator of $\pi_{*}L_{n}$ is trivial. Since each $\pi_{*}L_{n}$ is coherent, by Lemma~\refall{prop:coherent-pushforward}{2}, to prove the statement it is enough to check that $s_{n}$ generates all the fibers of $\pi_{*}L_{n}$. Again because $\pi_{*}$ commutes with base change, and by Nakayama's lemma, we can reduce to the case $X = \spec k$, where $k$ is a field.

We can also assume that $k$ is algebraically closed. Then $\cR = \infroot{P/k}$ for a certain sharp \fs monoid $P$. {Since $s_{n} \neq 0$ for sufficiently divisible $n$, by definition of infinite root,} the result follows from Lemma~\ref{lem:dim-1}.
\end{proof}

\begin{proof}[Proof of Proposition~\ref{prop:A_R-1-category}]
Since the category $\cA_{\cR}$ is fibered in groupoids, it is enough to show that an object of some $\cA_{\cR}(U)$ has no non-trivial automorphisms. We may assume that $X = U$, and $X$ is quasi-compact. Choose an object $(\Lambda_{n}, \alpha_{m,n}\bigr)$ of $\cA_{\cR}(X)$, and set $\Lambda_{n} = (L_{n}, s_{n})$. An automorphism of $(\Lambda_{n}, \alpha_{m,n}\bigr)$ is given by a sequence of elements $\xi_{n} \in \cO_{\cR}^{\times}(\cR) = \cO_{X}^{\times}(X)$ with $\xi_{n}s_{n} = s_{n}$ for all $n$, and such that $\xi_{n}^{(m/n)} = \xi_{m}$ whenever $m \mid n$. From Lemma~\ref{lem:invertible-image} we see that $\xi_{n} = 1$ when $n$ is sufficiently divisible, and this implies that $\xi_{n} = 1$ for all $n$.
\end{proof}

\begin{proof}[Proof of Proposition~\ref{prop:equality-sheaves}]
The statement can be checked on the geometric stalks; since formation of $A_{\cR}$ commutes with base change, we may assume that $X = \spec k$ is the spectrum of an algebraically closed field $k$, so that the logarithmic structure is given by a sharp \fs monoid $P$ and the monoidal functor $L\colon P \arr \div_{(\spec k)\et}$ sending $0$ to $(\cO_{\spec k}, 1)$ and everything else to $(\cO_{\spec k}, 0)$. Then the root stack $\cR = \infroot{(P, L)}$ equals $\cR = \bigl[\spec(k[P_{\QQ}]/(P^{+}))/\mmu_{\infty}(P)\bigr]$.

Let us identify $\pic \cR$ with $P\grq/P\gr$. We have a homomorphism of monoids $A_{\cR} \arr \widecheck P$ sending an infinite root $(\Lambda_{n}, \alpha_{m,n}\bigr)$ to $\{[L_{n}]\}$, where $\Lambda_{n}=(L_{n}, s_{n})$. 

\begin{lemma}
$\{[L_{n}]\} \in \widecheck P$ is an infinite quotient. 
\end{lemma}

\begin{proof}
Set $R = k[P_{\QQ}]/(P^{+})$ as above, {and $\lambda_n=[L_n]\in \pic \cR\cong P\grq/P\gr$. As  before, we denote by $R_\lambda$ the piece of degree $\lambda \in P\grq/P\gr$ of the $P\grq/P\gr$-graded ring $R$.} We will write the elements of $R$ as linear combinations $\sum_{i}a_{i}x^{\gamma_{i}}$, where $a_{i} \in k$ and $\gamma_{i} \in \Delta_{P}$. The multiplication in $R$ is given by the rule
   \[
   x^{\gamma}x^{\gamma'} = 
   \begin{cases}
   x^{\gamma + \gamma'}&\text{if }\gamma + \gamma' \in \Delta_{P}\\
   0 &\text{otherwise}
   \end{cases}
   \]
for any $\gamma$ and $\gamma'$ in $\Delta_{P}$.

Let $m_{0}$ such that $s_{m_{0}} \neq 0$; we claim that $m_{0}$ is a characteristic integer for $\{[L_{n}]\}$. Let $j$ be a positive integer; choose an isomorphism between $L_{jm_{0}}$ and the invertible sheaf corresponding to the invertible $R$-module $R(\lambda_{jm_{0}})$ (which, recall, is just the $R$-module $R$, but with $P\grq/P\gr$-grading defined by $R(\lambda_{jm_{0}})_\mu=R_{\lambda_{jm_0}+\mu}$), yielding an isomorphism $\H^{0}(\cR, L_{jm_{0}}) \simeq R_{\lambda_{jm_{0}}}$. The section $s_{jm_{0}}$ corresponds to an element $\sum_{\gamma \in \lambda_{jm_{0}}\cap\Delta_{P}} a_{\gamma}x^{\gamma}$; since $s_{jm_{0}}^{\otimes j} \neq 0$, we have
   \[
   \bigl(\sum_{\gamma \in \lambda_{jm_{0}}\cap\Delta_{P}} a_{\gamma}x^{\gamma}\bigr)^{j}
   \neq 0 \in R\,.
   \]
This implies that there must be a sequence of elements $\gamma_{1}$, \dots~$\gamma_{j}$ of $\lambda_{jm_{0}}\cap\Delta_{P}$ such that $x^{\gamma_{1}} \dots x^{\gamma_{j}} \neq 0$, which implies $\gamma_{1} + \dots + \gamma_{j} \in \Delta_{P}$. This show that $m_{0}$ is a characteristic integer, and completes the proof of the Lemma.
\end{proof}

Let $(\Lambda_{n}, \alpha_{m,n}\bigr)$ and $(\Lambda'_{n}, \alpha'_{m,n}\bigr)$ be two infinite roots on $\cR$. Assume that $[L_{n}] = [L'_{n}]$ for all $n$. {Notice that an automorphism of $L_{n}$ is given by an invertible element of $\H^{0}(\cR, \cO) = k$; hence the isomorphisms $L_{n} \simeq L'_{n}$, which exist by hypothesis, are well determined up to multiplication by scalars.} Then from Lemma~\ref{lem:dim-1} we see that for sufficiently divisible $n$ there is a unique isomorphism $L_{n} \cong L'_{n}$ carrying $s_{n}$ to $s'_{n}$. These give an isomorphism of the two infinite roots. This implies that the homomorphism $A_{\cR} \arr \infquot P$ is injective.

Now consider the composite $P = A \arr A_{\cR} \arr \infquot P$, which is easily seen to send $p \in P$ into $\infquot p \in \infquot P$. Since $A_{\cR} \arr \infquot P$ is injective and $P \arr \infquot P$ is an isomorphism, by Proposition~\refall{prop:description-infinite-quotients}{2}, the result follows.
\end{proof}


\subsection{Morphisms of \irss}\label{sec:morphisms-irss}

Let us fix a scheme $X$. Let $\phi = (\phi, \Phi)\colon (A, L) \arr (B,M)$ be a morphism of \fs \df structures on $X$. Recall that this means that $\phi\colon A \arr B$ is a homomorphism of sheaves of monoids on $X{\et}$, while $\Phi\colon L \cong M \circ \phi$ is a base-preserving isomorphism of \sm functors $A \arr \div_{X\et}$.

A morphism of \fs \df structures as above induces a morphism of fibered categories $\infroot{\phi}\colon \infroot{(B, M)} \arr \infroot{(A, L)}$ by composition. More precisely, let $t\colon T \arr X$ be a morphism of schemes, with $T$ affine. An object $(N, \alpha)$ of $\infroot{(B, M)}(T)$ consists of a \df structure $N\colon t^{*}B_{\QQ} \arr \div_{T\et}$, together with an isomorphism $\alpha\colon t^{*}M \simeq N\rest{t^{*}B}$ of \sm functors $t^{*}B \arr \div_{T\et}$. From this we obtain an object $(N\circ t^{*}\phi_{\QQ}, \id_{t^{*}\phi} * \alpha)$ of $\infroot{(A, L)}(T)$, where $\id_{t^{*}\phi} * \alpha\colon  t^{*}L\simeq N\rest{t^{*}B} \circ t^{*}\phi$ is the Godement product.

Together with the obvious action on arrows, this gives a functor from the opposite of the category of \df structures on $X$, to the category of categories fibered in groupoids on $X$. This functor, however, is not full.

\begin{example}\label{ex:not.full}
Take $X=\spec k$ and $(A, L)$ to be the trivial \df structure on $X$, while $(B, M)$ is the \df structure corresponding to the standard logarithmic point (Example \ref{ex:log-points}). In other words, $(B, M)$ is obtained from the chart $\NN \arr \div X$, in which $n$ is sent to $(\cO, 1)$ if $n = 0$ and to $(\cO, 0)$ if  $n \neq 0$. It is immediate to check that there is no homomorphism of \df structures $(B, M) \arr (A, L)$. On the other hand $\infroot{(A, L)} = \spec k$ while $\infroot{(B, M)} = \infroot{\NN/k}$; since $\infroot{\NN/k}\red = \cB_{k}\mmu_{\infty}$, there is a morphism $\spec k \arr \cB_{k}\mmu_{\infty} \subseteq \infroot{(B, M)}$ corresponding to the trivial torsor.
\end{example}

There is a necessary condition for a base-preserving functor $f\colon \infroot{(B, M)} \arr \infroot{(A, L)}$ of \irss over $X$ to come from a morphism of \df structures. 

Suppose that $f\colon \cS \arr \cR$ is a morphism of fibered categories over $\cataff X$, where $\cS$ and $\cR$ are \irss over $X$. Let $t\colon T\arr X$ be a morphism of schemes, with $T$ affine, and $\lambda = (\Lambda_{n}, \alpha_{m,n})$ be an infinite root on $\cR_{T}$. Set $\Lambda_{n} = (L_{n}, s_{n})$. Consider the pullback $f^{*}\lambda \eqdef (f^{*}\Lambda_{n}, f^{*}\alpha_{m,n})$. This is not necessarily an infinite root on $\cS_{T}$; the issue is that $f^{*}\Lambda_{n} = (f^{*}L_{n}, f^{*}s_{n})$, and it is not necessarily true that $f^{*}s_{n}$ does not vanish on each geometric fiber of $\cS \arr X$ for sufficiently divisible $n$. This justifies the following definition.

\begin{definition}
A morphism $f\colon \cS \arr \cR$ of \irs over $X$ is an isomorphism class of base-preserving functors $f\colon \cS \arr \cR$, such that for any geometric point $p:\spec\Omega \arr X$ and any infinite root $\lambda$ on the geometric fiber $\cR_{p}$, the pullback $f_{p}^{*}\lambda$ is again an infinite root on $\cS_{p}$.
\end{definition}

The composite of two morphisms of \irs is a morphism of \irss; thus, with this notion of morphism \irss over $X$ form a category.

\begin{remark}
As suggested by the referee, we can reinterpret Definition \ref{def:infinite.root} as follows: given an infinite root $(\Lambda_n, \alpha_{m,n})$ on the infinite root stack $\cR$ over $X$, each $\Lambda_n$ gives us a morphism of stacks $f_n\colon \cR\to [\mathbb{A}^1/\gm]$, and the isomorphisms $\alpha_{m,n}$ give natural isomorphisms between $f_m$ and the composite $\cR\xrightarrow{f_n} [\mathbb{A}^1/\gm]\to  [\mathbb{A}^1/\gm]$, where $m\mid n$ and the last morphism raises a line bundle with a section to the $n/m$-th power. Note that this map $[\mathbb{A}^1/\gm]\to  [\mathbb{A}^1/\gm]$ is exactly the projection from the $n/m$-th root stack of $ [\mathbb{A}^1/\gm]$, equipped with its tautological logarithmic structure. Thus, an infinite root induces a morphism of fibered categories over schemes $\cR\to \infroot{[\mathbb{A}^1/\gm]}=\varprojlim_n [\mathbb{A}^1/\gm]$, and a compatible morphism $X\to [\mathbb{A}^1/\gm]$, given by the base of the infinite root (see Remark \ref{rmk:base.root}).

Condition (d) of Definition \ref{def:infinite.root} then exactly says that the morphism $\cR\to X\times_{[\mathbb{A}^1/\gm]}\infroot{[\mathbb{A}^1/\gm]}$ is a morphism of infinite root stacks over $X$, in the sense of the definition above.
\end{remark}

Corollary~\ref{cor:df.infinite} implies that the base-preserving functor $\infroot{\phi}\colon \infroot{(B, M)} \arr \infroot{(A, L)}$ induced by a morphism of \fs \df structures $\phi\colon (A, L) \arr (B, M)$ is a morphism of \irss: for a geometric point $p\colon \spec \Omega\to X$, an infinite root on ${\infroot{(A, L)}}_p$ comes from an element of the stalk $A_p$, and the pullback to  $\infroot{(B, M)}_p$ coincides with the image of that element via the map $A_p\to B_p$, and hence is also an infinite root. So we obtain a functor from the opposite of the category of \df structures to the category of \irss.

We can also define a functor going in the opposite direction, from \irss to \df structures. At the level of objects, we send an \irs $\cR$ to $(A_{\cR}, L_{\cR})$. If $f\colon \cS \arr \cR$ a morphism of \irss, pullback of infinite roots defined above defines a base-preserving functor $f^{*}\colon \cA_{\cR} \arr \cA_{\cS}$, and hence a morphism of \df structures $f^{*}\colon (A_{\cR}, L_{\cR}) \arr (A_{\cS}, L_{\cS})$.

Consider the morphism $\infroot{f^{*}}\colon \infroot{(A_{\cS}, L_{\cS})} \arr \infroot{(A_{\cR}, L_{\cR})}$ induced by $f^{*}$. It is easy to show that the diagram
   \[
   \begin{tikzcd}
   {}\cS \rar{f}\dar& \cR\dar\\
   {}\infroot{(A_{\cS}, L_{\cS})} \rar{\infroot{f^{*}}} & \infroot{(A_{\cR}, L_{\cR})}
   \end{tikzcd}
   \]
is $2$-commutative.
%

\begin{theorem}\label{thm:equivalence1} The functors above define an equivalence between the opposite of the category of \fs \df structures on $X$ and the category of \irss on $X$.
\end{theorem}

\begin{proof}
This follows from Proposition~\ref{prop:equality-sheaves} and Theorem~\ref{thm:equivalence}.
\end{proof}

Since an equivalence of fibered categories between \irss is clearly a morphism of \irss, we obtain the following.

\begin{corollary}\label{cor:isomorphisms}
Let $(A, L)$ and $(B, M)$ be \df structures on a scheme $X$. An equivalence of \irss $\infroot{(B, M)} \simeq \infroot{(A, L)}$ over $X$ induces an isomorphism of \df structures $(A, L) \simeq(B, M)$.
\end{corollary}

The theorem can be restated in global terms, without fixing a base scheme. Consider the category $\fslogsch$ of fine saturated logarithmic schemes. Let us define a category $\rootstack$ in which the objects are pairs $(X, \cR)$, where $X$ is a scheme and $\cR \arr X$ is an \irs on $X$. The arrows $(Y, \cS) \arr (X, \cR)$ are isomorphism classes of commutative diagrams
   \[
   \begin{tikzcd}
   {}\cS \rar \dar & \cR \dar\\
   {}Y \rar & X
   \end{tikzcd}
   \]
with the property that the induced base-preserving functor $\cS \arr Y \times_{X} \cR$ is a morphism of \irss over $Y$. Note that the map $Y\to X$ is actually determined by the functor $\cS\to \cR$, by Proposition \ref{prop:properties-irs}.

On the other hand if $(X, A, L)$ and $(Y, B, M)$ are \fs logarithmic schemes, a morphism of logarithmic schemes $(Y, B, M) \arr (X, A, L)$ corresponds to a morphism of schemes $f\colon Y \arr X$ and a morphism of \df structures $f^{*}(A, L) \arr (B, M)$. Since $\infroot{f^{*}(A, L)} = Y\times_{X}\infroot{(A, L)}$,  we have the following variant of Theorem~\ref{thm:equivalence1}.

\begin{theorem}\label{thm:equivalence2}
There is an equivalence between the category of \fs logarithmic schemes and the category of \irss, sending a \fs scheme $(X, A, L)$  into $\bigl(X, \infroot{(A, L)}\bigr)$.\qed
\end{theorem}

\begin{remark}
Note in particular that if $X$ and $Y$ are \fs logarithmic schemes and a base-preserving functor $F\colon \infroot{X}\to \infroot{Y}$ is representable and faithfully flat, then the associated pullback will automatically carry infinite roots to infinite roots, and thus $F$ is a morphism of infinite root stacks. Because of the previous theorem it will then come from a unique morphism of logarithmic schemes $X\to Y$.
\end{remark}

For later use we add the following.

\begin{lemma}\label{lem:between->logarithmic}
Let $\phi\colon (A, L) \arr (B, M)$ and $\psi\colon (A, L) \arr (C, N)$ be morphisms of \df structures on a scheme $X$, such that $\psi$ is Kummer. Suppose that $f\colon  \infroot{(B,M)}\arr \infroot{(C, N)}$ is a morphism of fibered categories over $X$ making the diagram
   \[
   \begin{tikzcd}[column sep = small]
   \infroot{(B, M)} \ar{rr}{f} \ar{rd}[swap]{\infroot{\phi}}&&\infroot{(C,N)} \ar{ld}{\infroot \psi}\\
   {}& \infroot{(A, L)}
   \end{tikzcd}
   \]
commute. Then $f$ is a morphism of \irss.
\end{lemma}

\begin{proof}
We need to check that $f$ sends infinite roots in geometric fibers to infinite roots; by base change, we may assume that $X = \spec k$, where $k$ is an algebraically closed field. For consistency with the previous notation, set $P \eqdef A$, $Q \eqdef B$ and $R \eqdef C$; we need to check that the homomorphism $f^{*}\colon R\grq/R\gr \arr Q\grq/Q\gr$ induced by $f$ sends $\infquot R$ into $\infquot Q$ (here we are using the identification $Q\grq/Q\gr \simeq \pic \infroot{Q/k}$). Taking projective limits and using the identifications ${\widecheck P} \simeq P\gr \otimes \widehat{\ZZ}$ we obtain a commutative diagram
   \[
   \begin{tikzcd}[column sep = small]
   {}&P\gr \otimes \widehat{\ZZ} \ar{rd} \ar{ld}&\\
   R\gr \otimes \widehat{\ZZ} \ar{rr}{f^{*}} && Q\gr \otimes \widehat{\ZZ}
   \end{tikzcd}
   \]
in which the two diagonal arrows take $P$ into $R$ and $Q$ respectively. We need to show that $f^{*}$ takes $R$ into $Q$. Since the homomorphism $P \arr R$ is Kummer, given $r \in R$ we can find a positive integer $n$ such that $nr$ comes from $P$; this implies that $nf^{*}(r) = f^{*}(nr)$ is in $Q$. {Since $\widehat{\ZZ}/\ZZ$ is torsion free, by the Lemma below, so is $(Q\gr \otimes \widehat{\ZZ})/Q\gr = Q\gr\otimes (\widehat{\ZZ}/\ZZ)$, because $Q\gr$ is free. Hence we see that $f^{*}(r) \in Q\gr$; since $Q$ is saturated this implies $f^{*}(r) \in Q$, and this concludes the proof.}
\end{proof}

\begin{lemma}
The abelian group $\widehat{\ZZ}/\ZZ$ is torsion-free.
\end{lemma}

\begin{proof}
Consider the embeddings $\widehat{\ZZ} = \prod_{p}\ZZ_{p} \subseteq \prod_{p}\QQ_{p}$ and $\QQ \subseteq \prod_{p}\QQ_{p}$. Since $(\prod_p \ZZ_{p}) \cap \QQ = \ZZ$, we have that the natural homomorphism $\widehat{\ZZ}/\ZZ \arr (\prod_p \QQ_{p})/\QQ$ is injective, which proves the result.
\end{proof}

The preceding Lemma~\ref{lem:between->logarithmic} has the following immediate corollary.

\begin{corollary}\label{cor:kummer->logarithmic}
Let $f\colon Y\to X$ and $g\colon Z\to X$ be two morphisms of \fs logarithmic schemes, such that $Y\to X$ is Kummer. Then every morphism $\infroot{Z}\to \infroot{Y}$ fitting in a commutative diagram
  \[
   \begin{tikzcd}[column sep = small]
   \infroot{Z} \ar{rr} \ar{rd}[swap]{\infroot{g}}&&\infroot{Y} \ar{ld}{\infroot f}\\
   {}& \infroot{X}
   \end{tikzcd}
   \]
comes from a morphism $Z\to Y$ of logarithmic schemes over $X$.\qed
\end{corollary}


\section{The small fppf site of an \irs\\and Kato's Kummer-flat site}
\label{realsec:fp-sheaves-fppf}

In this section we introduce the small fppf site of an \irs, a site obtained by using a natural flat topology, and we compare it with Kato's Kummer-flat site of the corresponding logarithmic structure.

Our main result here is that there is a natural comparison functor, that induces an equivalence of the corresponding topoi; in other words sheaves on the Kummer-flat site of a logarithmic scheme are the same as sheaves on the fppf site of its \irs.

\subsection{Finitely presented sheaves and the small fppf site of an \irs}%
\label{sec:fp-sheaves-fppf}
Let $\cR$ be an \irs over a scheme $X$. Set $\cR_{n} \eqdef \radice n {(X, A_{X}, L_{X})}$; by Proposition~\ref{prop:projective-limit} we have $\cR = \projlim_{n}\cR_{n}$. We will denote by $\pi_{n}\colon \cR \arr \cR_{n}$ the canonical map.

The pullback functors $\FP\cR_{n} \arr \FP\cR$ induce a functor $\indlim_{n}\FP\cR_{n} \arr \FP\cR$.

\begin{proposition}\label{prop:limit-finitely-presented}
Assume that $X$ is quasi-compact and quasi-separated. Then the functor $\indlim_{n}\FP\cR_{n} \arr \FP\cR$ is an equivalence.
\end{proposition}

\begin{proof}
Assume that the logarithmic structure on $X$ comes by pullback from a map $X \arr X_{P}$ for some sharp \fs monoid $P$. Denote by $U_{n}$ and $U_{\infty}$ the fibered products $X \times_{X_{P}}X_{P}^{[n]}$ and $X \times_{X_{P}}X_{P}^{[\infty]}$ respectively, by $R_{n}$ and $R_{\infty}$ the products $U_{n} \times_{\spec \ZZ} \mmu_{n}(P)$ and $U_{\infty} \times_{\spec \ZZ} \mmu_{\infty}(P)$. We have that $\projlim_{n}U_{n} \simeq U_{\infty}$, $\projlim_{n}R_{n} \simeq R_{\infty}$, and $\projlim_{n}(R_{n}\times_{U_{n}}R_{n}) \simeq (R_{\infty}\times_{U_{\infty}}R_{\infty})$; furthermore, the isomorphisms above are compatible with the structure maps of the groupoids $R_{n}\double U_{n}$ and $R_{\infty} \double U_{\infty}$, that give presentations for $\cR_{n}$ and $\cR$ respectively. Consequently $R_{\infty} \double U_{\infty}$ is the projective limit of the $R_{n}\double U_{n}$ in the category of groupoid schemes. 

By \cite[Théorème~8.5.2]{ega43}, if $Y_{\lambda}$ is a projective system of affine schemes over a fixed quasi-compact and quasi-separated scheme $X$, and $Y \eqdef \projlim_{\lambda}Y_{\lambda}$, the pullbacks $\FP Y_{\lambda} \arr \FP Y$ induce an equivalence $\indlim \FP Y_{\lambda} \simeq \FP Y$. From this we obtain an equivalence of the category $\FP(R_{\infty} \double U_{\infty})$ of \fp sheaves with descent data on the groupoid $R_{\infty} \double U_{\infty}$ and the inductive limit $\indlim_{n}\FP(R_{n} \double U_{n})$. But $\FP(R_{\infty} \double U_{\infty})$ is equivalent to $\FP\cR$, while $\FP(R_{n} \double U_{n})$ is equivalent to $\FP\cR_{n}$; the resulting equivalence $\indlim_{n}\FP\cR_{n} \simeq \FP\cR$ is easily seen to be isomorphic to the functor in the statement.

In general, let $\{X_{i}\to X\}_{1 \leq i \leq r}$ be {an \'etale covering} of $X$, such that for each $i$ the restriction of the logarithmic structure on $X$ to $X_{i}$ comes by pullback from a map $X_{i} \arr X_{P_{i}}$ for some sharp \fs monoid $P_{i}$. Set
   \begin{align*}
   X' &\eqdef \coprod_{i}X_{i}\,,\\
   X'' &\eqdef X' \times_{X} X' = \coprod_{i,j}(X_{i} \times_X X_{j})\,,\\
   \cR'_{n} &\eqdef X'\times_{X}\radice{n}X\,,\\
   \cR''_{n} &\eqdef X''\times_{X}\radice{n}X\,,\\
   \cR' &\eqdef X'\times_{X}\cR\,,\\
   \cR'' &\eqdef X''\times_{X}\cR\,.
   \end{align*}

Notice that the functors $\indlim_{n}\FP \cR'_{n} \arr \FP\cR'$ and $\indlim_{n}\FP \cR''_{n} \arr \FP\cR''$ are equivalences, because of the previous case, and because filtered inductive limits commute with finite products. Furthermore, the natural functors $\FP\cR_{n} \arr \FP(\cR'_{n} \double \cR''_{n})$ and $\FP\cR \arr \FP(\cR' \double \cR'')$ are equivalences, by {\'etale} descent; hence it is enough to show that the functor
   \[
   \indlim_{n}\FP(\cR'_{n} \double \cR''_{n}) \arr \FP(\cR' \double \cR'')
   \]
is an equivalence. In fact, since filtered inductive limits are exact, we easily see that $\indlim_{n}\FP(\cR'_{n} \double \cR''_{n})$ is equivalent to the category of objects of $\indlim_{n}\FP\cR'_{n}=\FP \cR'$ with descent data in $\indlim_{n}\FP\cR''_{n}=\FP\cR''$, as desired.
\end{proof}

Recall from Proposition~\refall{prop:equivalence-fp}{4} that if $\cX$ is a category fibered in groupoids over $\aff$, the category $\FP\cX$ of \fp sheaves on $\cX$ is equivalent to the category $\FP\cX\rep$ whose objects are representable maps $\cA \arr \cX$, endowed with the fpqc topology. For root stacks, we have a much smaller site with an equivalent category of \fp sheaves.

\begin{definition}\label{def:small-fppf}
Let $\cX$ be a category fibered in groupoids over $\aff$. The \emph{small fppf site} $\cX\fppf$ is the full subcategory of $\cX\rep$ whose objects are flat locally \fp representable base-preserving functors $\cA \arr \cX$.

A covering $\{\cA_{i} \arr \cA\}$ is a collection of arrows in $\cX\fppf$ such that the induced base-preserving functor $\bigsqcup_{i}\cA_{i} \arr \cA$ is surjective and fppf.
\end{definition}

This site is particularly important for our purposes, because as we will see (Theorem~\ref{equiv.topoi}) the topos of the small fppf site of an \irs is equivalent to Kato's Kummer-flat topos of the corresponding logarithmic scheme.

Since any representable fppf morphism $\cA\to \cX$ is also fpqc, we have an inclusion functor $i\colon \cX\fppf\to \cX\rep$, which is a morphism of sites, in the sense of \cite[Tag~00X0]{stacks-project}, and induces a morphism of topoi $(i_{*},i^{-1})\colon\sh\cX\rep\to \sh\cX\fppf$. The functor $i_{*}\colon \sh\cX\rep \arr \sh\cX\fppf$ is given by restriction. We define the structure sheaf $\cO_{\cX\fppf}$ as $i_{*}\cO_{\cX\rep}$; this makes $(i_{*},i^{-1})$ into a morphism of ringed topoi.

We denote by $i^{*}\colon \mod{\cO_{\cX\fppf}} \arr \mod{\cO_{\cX\rep}}$ the corresponding pullback of sheaves of $\cO$-modules; we will use the same symbol for the functors $i^{*}\colon \catqcoh{\cX\fppf} \arr \catqcoh{\cX\rep}$  and $i^{*}\colon \FP{\cX\fppf} \arr \FP{\cX\rep}$ induced by restriction.
 
If $\cY$ is another category fibered in groupoids and $f\colon \cY \to \cX$ is a base-preserving functor, then we have pushforward and pullback functors $f_{*}\colon \sh\cY\fppf\to  \sh\cX\fppf$ and $f^{-1}\colon \sh\cX\fppf \to \sh\cY\fppf$, together with an adjunction $f^{-1} \dashv f_{*}$. {These functors are defined in the usual manner; for example the pushforward $f_*$ is defined by the formula $f_*F(\cA\to \cX)=F(\cA\times_\cX\cY\to \cY)$ on objects, and by the analogous formula on arrows.} {Note that in this case $(f_*, f^{-1})$ is not necessarily a morphism of topoi, as $f^{-1}$ might not be left exact.}

\begin{remark}
Additionally, pullback on quasi-coherent sheaves is compatible with the morphism $(i_{*}, i^{-1})$, i.e. the following diagram is 2-commutative
   \[
   \begin{tikzcd}
   \catqcoh\cX\fppf\ar{r}{i^{*}}\ar{d}{f^{*}} &\catqcoh\cX\rep\ar{d}{f^{*}}\\
   \catqcoh\cY\fppf\ar{r}{i^{*}} & \catqcoh\cY\rep
   \end{tikzcd}
   \]
where we used the same symbol $i^{*}$ for $\cX$ and $\cY$.
\end{remark}

\begin{lemma}\label{lem:alg->fppf=rep}
Suppose that $\cX$ is an algebraic stack over $\aff$ with schematic diagonal. Then the functors
   \[
   i^{*}\colon \catqcoh \cX\fppf \arr \catqcoh \cX\rep
   \quad\text{and}\quad
   i^{*}\colon \FP \cX\fppf \arr \FP \cX\rep
   \]
are equivalences.
\end{lemma}

{
The assumption that the diagonal is schematic is unnecessary, but it will be satisfied in the cases where we apply this lemma.
}

\begin{proof}
The point is that in this case there exists a representable fppf map $U \arr \cX$, where $U$ is a scheme. Set $R \eqdef U \times_{\cX}U$, which is again a scheme by assumption. Denote by $\cX\sfppf$ and $\cX\sch$ the subcategories of $\cX\fppf$ and $\cX\rep$ respectively, consisting of maps $T \arr \cX$ in which $T$ is a scheme ($\cX\sch$ already appeared in section \ref{sec:qc}). It follows from Lemma~\ref{lem:subcategory-equivalence-qc} that the embeddings $\cX\sfppf \subseteq \cX\fppf$ and $\cX\sch \subseteq \cX\rep$ induce equivalences of the corresponding topoi, hence equivalences of the categories of \qc and \fp sheaves. By Proposition~\ref{prop:equivalence-qc}, the obvious functor $\catqcoh{\cX\sch} \arr \catqcoh{(R \double U)}$ is an equivalence; and easy descent arguments show that the composite $\catqcoh{\cX\sfppf} \xarr{i^{*}}\catqcoh{\cX\sch} \arr \catqcoh{(R \double U)}$ is also an equivalence. The first statement follows from this.

The second statement follows from the fact that being of finite presentation is a local property in the fpqc topology.
\end{proof}

In general the pullback $i^{*}\colon \catqcoh \cX\fppf \arr \catqcoh \cX\rep$ will not be an equivalence for an arbitrary fibered category with an fpqc atlas (see Remark~\ref{rmk:fppf-fpqc-not-equivalence}).  Nonetheless, this is true if we restrict to \fp sheaves on an infinite root stack $\cR$.

\begin{proposition}\label{prop:equivalence-fppf-fpqc}
Let $\cR$ be an \irs over a quasi-separated scheme. The functor
   \[
   i^{*}\colon \FP\cR\fppf\arr \FP\cR\rep
   \]
is an equivalence.
\end{proposition}

\begin{proof}
Assume that $X$ is also quasi-compact.

If $n$ is a positive integer we have a $2$-commutative diagram of categories
   \[
   \begin{tikzcd}
   {}\FP(\cR_{n})\fppf \rar{\pi_{n}^{*}} \dar{i^{*}_{n}} & \FP\cR\fppf \dar{i^{*}}\\
   {}\FP(\cR_{n})\rep \rar{\pi_{n}^{*}} & \FP\cR\rep
   \end{tikzcd}
   \]
where, as before, $\cR_{n}$ is the $n$-th root stack of the logarithmic structure determined by $\cR$, $\pi_{n}\colon \cR\arr \cR_{n}$ is the projection, and $i_{n}\colon (\cR_{n})\fppf \arr (\cR_{n})\rep$ is the embedding. Thus we get a $2$-commutative diagram
   \[
   \begin{tikzcd}[column sep = large]
   {}\indlim_{n}\FP(\cR_{n})\fppf \rar \dar{\indlim i^{*}_{n}} & \FP\cR\fppf \dar{i^{*}}\\
   {}\indlim_{n}\FP(\cR_{n})\rep\rar{\indlim\pi_{n}^{*}}  & \FP\cR\rep\,.
   \end{tikzcd}
   \]
The left hand column is an equivalence by Lemma~\ref{lem:alg->fppf=rep}, while the bottom row is an equivalence by Proposition~\ref{prop:limit-finitely-presented}. From this it follows that $i^{*}\colon \FP\cR\fppf \arr \FP\cR\rep$ is full and essentially surjective. It remains to show that it is faithful. So, take two \fp sheaves $F$ and $G$ on $\cR\fppf$, and a morphism $\phi\colon G \arr F$ that is $0$ when pulled back to $\cR\rep$; we need to show that $\phi$ is $0$. This seems surprisingly non-trivial to us.

We will make use of the following result, that is a consequence of Lemma~\ref{lem:refine-to-irs}, proven below.

\begin{lemma}\label{lem:stated}
Let $\cR$ be an infinite root stack over a scheme $X$, and let $\cA\to \cR$ be a representable fppf morphism. Then there exists a representable fppf morphism $\cS\to \cA$, where $\cS\to Y$ is an infinite root stack over a scheme $Y$.
\end{lemma}

\begin{lemma}\label{lem:pushforward-exact}
Let $\cR$ be an infinite root stack over a scheme $X$. Then the restriction $i_{*}\colon \catqcoh \cR\rep \arr \mod{\cO_{\cR\fppf}}$ is exact.
\end{lemma}

\begin{proof}
Since $i_{*}$ is a right adjoint, it is enough to prove that it is right exact. Let $\phi\colon G \arr F$ be a surjective homomorphism of \qc sheaves on $\cR$; we need to show that $i_{*}\phi\colon i_{*}G \arr i_{*}F$ is surjective. Let $\cA \arr \cR$ be a representable \fp flat map, and $s \in F(\cA)$; we need to show that there exists a representable fppf map $\cB \arr \cA$ such that the pullback of $s$ to $F(\cB)$ is in the image of $G(\cB)$. 

By Lemma \ref{lem:stated} stated above, there exists a representable fppf morphism $\cS \arr \cA$ where $\cS \arr Y$ is an \irs over a scheme $Y$. By base-changing to $\cS$, we can assume that $\cA = \cR$. Furthermore, the problem is clearly local in the fppf topology of $X$; hence we can assume that $X = \spec R$ is affine, and that the logarithmic structure of $X$ comes from a morphism $\ZZ[P] \arr R$, where $P$ is a sharp fine saturated monoid. Set $A \eqdef R \otimes_{\ZZ[P]} \ZZ[P_{\QQ}]$ and $H \eqdef \mmu_{\infty}(R)$; then $H$ acts on $A$, and $\cR = [\spec A/H]$ (Corollary~\ref{cor:local-model-chart}). We also have $R = A^{H}$. The \qc sheaves $G$ and $F$ correspond to $H$-equivariant $A$-modules $N$ and $M$, and $\phi\colon G \arr F$ corresponds to a $H$-equivariant morphism of $A$-modules $\Phi\colon N \arr M$. Then $G(\cR) = N^{H}$ and $F(\cR) = M^{H}$. But the group scheme $H$ is diagonalizable, hence the functor $M \arr M^{H}$ is exact. This shows that $s \in F(\cR)$ comes from $G(\cR)$, and completes the proof.
\end{proof}

\begin{lemma}\label{lem:unity-isom}
If $F$ is a finitely presented sheaf on $\cR\fppf$, the unit homomorphism $F \arr i_{*}i^{*}F$ is an isomorphism.
\end{lemma}

\begin{proof}
Let $\cA \arr \cR$ be a representable fppf cover, such that there is an exact sequence of $\cO_{\cA\fppf}$-modules $\cO_{\cA\fppf}^{\oplus n} \arr \cO_{\cA\fppf}^{\oplus m} \arr F \arr 0$. Using Lemma~\ref{lem:refine-to-irs} again, by refining $\cA$ we may assume that $\cA$ is an \irs. Since $i^{*}$ is right exact and because of Lemma~\ref{lem:pushforward-exact} we have a commutative diagram
   \[
   \begin{tikzcd}
   \cO_{\cA\fppf}^{\oplus n} \rar\dar& \cO_{\cA\fppf}^{\oplus m} \dar\rar& F \dar\rar& 0\\
   i_{*}\cO_{\cA\rep}^{\oplus n} \rar& i_{*}\cO_{\cA\rep}^{\oplus m} \rar& i_{*}i^{*}F \rar& 0
   \end{tikzcd}
   \]
with exact rows. But $i_{*}\cO_{\cX\rep} = \cO_{\cX\fppf}$, and since $i_{*}$ commutes with finite direct sums the two left-hand vertical maps are isomorphisms; this implies the thesis.
\end{proof}

The fact that $i_{*}$ is faithful follows immediately from Lemma~\ref{lem:unity-isom}. This completes the proof of Proposition~\ref{prop:equivalence-fppf-fpqc} when $X$ is quasi-compact.

For the general case, when $X$ is not necessarily quasi-compact, notice that by assigning to each open subscheme $U \subseteq X$ the categories $\FP(U\times_{X}\cR)\fppf$ and $\FP(U\times_{X}\cR)\rep$ we get a stack in the Zariski topology, and $i^{*}$ extends to a base-preserving functor of Zariski stacks. If $\{U_{\alpha}\}$  is a covering by quasi-compact open subschemes, then
   \[
   i^{*}\colon \FP(U_{\alpha}\times_{X}\cR)\fppf \arr \FP(U_{\alpha}\times_{X}\cR)\rep
   \]
and
   \[
   i^{*}\colon \FP(U_{\alpha\beta}\times_{X}\cR)\fppf \arr
   \FP(U_{\alpha\beta}\times_{X}\cR)\rep\,,
   \]
where we have set $U_{\alpha\beta}\eqdef U_{\alpha}\times_{X}U_{\beta}$, are equivalences for all $\alpha$ and $\beta$. Consequently $i^{*}\colon \FP\cR\fppf\arr \FP\cR\rep$ is also an equivalence.
\end{proof}

\begin{remark}\label{rmk:fppf-fpqc-not-equivalence}
It follows from the proof of Proposition~\ref{prop:equivalence-fppf-fpqc} that the image of the pullback $i^{*}\colon \catqcoh{\cR\fppf} \arr \catqcoh{\cR\rep}$ is equivalent to the colimit $\indlim_{n}\catqcoh{\cR_{n}}$. It is easy to show that $\indlim_{n}\catqcoh{\cR_{n}}$ is not closed under infinite direct sums, when $\cR$ is not trivial (if for each $n$ we pick a \qc sheaf $F_{n}$ on $\cR_{n}$ that is not a pullback of a sheaf $\cR_{m}$ for any proper divisor $m$ of $n$, the direct sum $\oplus_{n}F_{n}$ will not come from the direct limit). Hence $i^{*}\colon \catqcoh{\cR\fppf} \arr \catqcoh{\cR\rep}$ is not an equivalence.
\end{remark}


\subsection{The relation with Kato's Kummer-flat site}\label{sec:fppf-kummer-flat}

In this section we will show that the Kummer-flat topos of a logarithmic scheme (\cite{kato2, niziol-k-theory, illusie-nakayama-tsuji}) can be recovered as the fppf topos of the corresponding infinite root stack.

We briefly recall the construction of the Kummer-flat topos of a logarithmic scheme.

Recall that a morphism of logarithmic schemes $f\colon Y\to X$ is log-flat if the following holds: fppf locally on $X$ and $Y$ we can find Kato charts $P\to M_{X}$ and $Q\to M_{Y}$ and a morphism $P\to Q$ such that the diagram
   \begin{equation}\label{eq:diagram2}
   \begin{tikzcd}
   Y \rar\dar &X_{Q} \dar\\
   X \rar & X_{P}
   \end{tikzcd}
   \end{equation}
commutes, and the induced map $Y\to X\times_{X_{P}}X_{Q}$ is flat (here, as usual, $X_{P}=\sz P$ for a monoid $P$). Recall also that a homomorphism of monoids $P\to Q$ is Kummer if it is injective, and every element of $Q$ has a positive multiple in the image. A morphism $f\colon Y\to X$ is Kummer if the corresponding $f^{*}A_{X}\to A_{Y}$ is Kummer, meaning that the homomorphism of monoids $(f^{*}A_{X})_{y}\to (A_{Y})_{y}$ is Kummer for any geometric point $y\to Y$.

\begin{definition}[Kazuya Kato]
A morphism of fine saturated logarithmic schemes $f\colon Y\to X$ is \emph{Kummer-flat} if it is log-flat and Kummer, and the underlying map of schemes is locally of finite presentation.
\end{definition}

\begin{remark}
The finite presentation of the underlying map of schemes is not essential, but will be useful for our purposes, so we include it in the definition.
\end{remark}

Since charts can be made up from stalks, if $f\colon Y\to X$ is Kummer-flat, then locally we can find charts as above such that in addition $P\to Q$ is Kummer, and it is proved in \cite[Proposition 1.3]{illusie-nakayama-tsuji} that we can also make $Y\to X\times_{X_{P}}X_{Q}$ locally of finite presentation.

For a logarithmic scheme $X$, there is a site, called the \emph{Kummer-flat site} and denoted by $X\kfl$, whose objects are morphisms of logarithmic schemes $U\to X$ that are Kummer-flat, with morphisms of logarithmic schemes over $X$ as arrows, and with jointly surjective families $\{U_{i}\to U\}_{i \in I}$ of Kummer-flat morphisms as coverings. The corresponding topos $\sh X\kfl$ is called the \emph{Kummer-flat topos} of $X$.  
 
\begin{remark}
The site $X\kfl$ has a final object, the identity morphism $X = X$, and fibered products. Given a diagram
   \[
   \begin{tikzcd}
   {}&V\dar\\
   Z\rar & Y
   \end{tikzcd}
   \]
in $X\kfl$, the fibered product is given by the fibered product $V\times_{Y}Z$ in the category of fine saturated logarithmic schemes over $k$, together with the induced Kummer-flat map $V\times_{Y}Z\to X$. Notice, however, that the underlying scheme of the fibered product $V\times_{Y}Z$ may not be the fibered product $\underline V \times_{\underline Y} \underline Z$ (see Section 2.4 of \cite{ogus}).
\end{remark}

\begin{remark}
If we have two objects $Y\to X$ and $Z\to X$ of $X\kfl$, then any morphism $Z\to Y$ in $X\kfl$ is also Kummer. This follows from the fact that if two morphisms of \fs torsion-free monoids $P\to Q$ and $P\to R$ are Kummer and we have a commutative diagram
   \[
   \begin{tikzcd}
   P\rar\dar& R \\
   Q\ar{ru}&
   \end{tikzcd}
   \]
then $Q\to R$ is also Kummer.

Indeed, any $r \in R$ has some multiple $nr$ coming from $p$, which means that it also comes from $Q$. Moreover the map is injective: if $q$ and $q'$ go to the same element $r$, take a positive integer $n$ such that $nr, nq$ and $nq'$ all come from $P$. Then if $p$ goes, say, to $nq$ and $p'$ goes to $nq'$, since $P\to R$ is injective and $p$, $p'$ both go to $nr$, we must have $p=p'$, which means that $nq=nq'$, and so $q=q'$ by torsion-freeness.
\end{remark}


\begin{proposition}\label{kfl.morphisms}
Let $f\colon Y\to X$ be a Kummer-flat (resp. Kummer-étale) morphism of logarithmic schemes. Then the induced morphism $\infroot{f}\colon \infroot{Y}\arr \infroot{X}$ between the infinite root stacks is representable, locally \fp and flat (resp. representable and étale).
\end{proposition}
\begin{proof}
Let us prove the result for Kummer-flat maps. Since the question is local for the fppf topology of $X$ and $Y$, we can assume that we have a diagram
   \[
   \begin{tikzcd}
   Y \rar\ar{rd} & X\times_{X_{P}} X_{Q} \rar\dar & X_{Q}\dar \\
   {}& X \rar &X_{P}
   \end{tikzcd}
   \]
in which the map $X_{Q} \arr X_{P}$ is induced by a Kummer homomorphism $P \arr Q$, $X \arr X_{P}$ and $Y \arr X_{Q}$ are strict, and $Y\arr X\times_{X_{P}}X_{Q}$ is flat and \fp, and also strict. By Proposition~\ref{prop:base-change} we see that we have two cartesian diagrams
   \[
   \begin{tikzcd}
   {}\infroot{X\times_{X_P}X_{Q}}\rar\dar & \infroot{X_{Q}}\dar\\
   {}\infroot{X} \rar & \infroot{X_{P}}
   \end{tikzcd}
   \]
and
   \[
   \begin{tikzcd}
   {}\infroot{Y} \rar \dar &{}\infroot{X\times_{X_P}X_{Q}}\dar\\
   Y \rar & X\times_{X_{P}}X_{Q}\,;
   \end{tikzcd}
   \]
hence it is enough to prove that the map $\infroot{X_{Q}} \arr \infroot{X_{P}}$ is representable, flat and \fp. But the Kummer homomorphism $P \arr Q$ induces and isomorphism $P_{\QQ}\arr Q_{\QQ}$, hence an isomorphism $X_{P}^{[\infty]} \simeq X_{Q}^{[\infty]}$. On the other hand the homomorphism $\mmu_{\infty}(P) \arr \mmu_{\infty}(Q)$ is Cartier dual to the group homomorphism $P\gr_{\QQ}/P\gr \arr Q\gr_{\QQ}/Q\gr$, which is surjective with kernel $Q\gr/P\gr$. Call $\Gamma \eqdef \hom(Q\gr/P\gr, \gm)$ the Cartier dual to $Q\gr/P\gr$; we see that $\mmu_{\infty}(P) \arr \mmu_{\infty}(Q)$ is injective, with cokernel $\Gamma$. Since $\Gamma$ is \fp, being Cartier dual to a finite group, the conclusion follows from Lemma~\ref{lem:affine-map}.
\end{proof}

Because of proposition \ref{kfl.morphisms} there is a natural functor $F\colon X\kfl\to {\infroot X}\fppf$ from the Kummer-flat site of $X$ to the small fppf site of $\infroot{X}$, acting on objects by taking $f\colon Y\to X$ to $\infroot{f}\colon \infroot Y\arr \infroot X$, and on arrows by taking $g\colon Z\to Y$ over $X$ to $\infroot g\colon \infroot Z\arr \infroot Y$ over $\infroot X$.

\begin{lemma}\label{infty.products}
The functor $F$ preserves fibered products.
\end{lemma}

\begin{proof}
The statement means that if
   \[
   \begin{tikzcd}
   W \rar\dar & V\dar\\
   Z\rar & Y
   \end{tikzcd}
   \]
is a cartesian diagram in $X\kfl$, then the diagram
   \[
   \begin{tikzcd}
   {}\infroot W\rar\dar & \infroot V\dar\\
   \infroot Z\rar & \infroot Y
   \end{tikzcd}
   \]
is $2$-cartesian, i.e. the induced morphism $\infroot W\arr \infroot Z\times_{\infroot Y}\infroot V$ is an equivalence.

Recall first of all that the morphisms $Z\to Y$ and $V\to Y$ are Kummer, and denote by $A$, $B$, $C$ and $D$ the sheaves of monoids giving the logarithmic structures of $Y$, $V$, $Z$ and $W$ respectively. Recall (\cite[Section 2.4]{ogus}) that $W$ is obtained in the following way: we first form the fibered product of the underlying schemes $\underline{V}\times_{\underline{Y}}\underline{Z}$ and, locally where we have charts $P\to A, Q\to B, R\to C$, equip it with the logarithmic structure coming from the pushout $Q\oplus^{P}R$ of the diagram
   \[
   \begin{tikzcd}
   P\rar\dar & Q\dar\\
   R\rar& Q\oplus^{P} R
   \end{tikzcd}
   \]
and then base change along $\sz{(Q\oplus^{P}R)^{\fsa}}\to \sz{Q\oplus^{P}R}$ (here $(-)^{\fsa}$ denotes the ``fine saturation'' of a monoid, see \cite[Section 1.2]{ogus}). Now note that since the functor $P\mapsto P_{\QQ}$ preserves pushouts (being a left adjoint; its right adjoint is the inclusion of the category of monoids $M$ such that $M \arr M_{\QQ}$ is an isomorphism into all \fs monoids), the diagram
   \[
   \begin{tikzcd}
   P_{\QQ}\rar\dar & Q_{\QQ}\dar\\
   R_{\QQ}\rar& (Q\oplus^{P} R)_{\QQ}
   \end{tikzcd}
   \]
is also a pushout; but in this case the maps $P_{\QQ}\to Q_{\QQ}$ and $P_{\QQ}\to R_{\QQ}$ are isomorphisms, since $P\to Q$ and $P\to R$ are Kummer. Consequently the remaining two maps in the diagram are also isomorphisms, and we have $(Q\oplus^{P}R)_{\QQ} = P_{\QQ}$.

Now we construct a quasi-inverse to the natural functor $\infroot W\arr \infroot Z\times_{\infroot Y}\infroot V$. Take an object of $(\infroot Z \times_{\infroot Y} \infroot V)(T)$, i.e. a triple $(\xi,\eta, f)$ where $\xi\colon (B_{T})_{\QQ}\to \div_{T\et}$ and $\eta\colon (C_{T})_{\QQ}\to \div_{T\et}$ are liftings of the \df structures coming from $V$ and $Z$ respectively, and $f$ is an isomorphism between their restrictions to $(A_{T})_{\QQ}$. Call $E$ the pushout of the diagram
   \[
   \begin{tikzcd}
   A_{T}\rar\dar& B_{T}\\
   C_{T}& 
   \end{tikzcd}
   \]
of sheaves of monoid over $T$. The preceding remarks imply that $(A_{T})_{\QQ}, (B_{T})_{\QQ},(C_{T})_{\QQ}$ are all isomorphic, and they are also isomorphic to $E_\QQ$, so we have an induced \df structure $E_{\QQ} \to \div_{T\et}$. Moreover since $E_{\QQ}=(A_{T})_{\QQ}$ is integral and saturated, the map $E\to E_{\QQ}$ factors through $E\to E^{\fsa}$, the fine saturation of the sheaf $E$. By restriction along $E^{\fsa}\to E_{\QQ}$, this gives a logarithmic structure on $T$ that makes the diagram
   \[
   \begin{tikzcd}
   T\rar\dar & V\dar\\
   Z\rar & Y
   \end{tikzcd}
   \]a commutative diagram of \fs logarithmic schemes. Consequently there is an induced strict morphism $T\to W$, and together with the lifting $(D_{T})_{\QQ}\cong E_{\QQ}\to \div_{T\et}$ of the \df structure coming from $W$ this gives our object of $\infroot W(T)$. We leave the remaining verifications to the reader.
\end{proof}

From Lemma~\ref{infty.products} we see that $F$ gives a morphism of sites $\infroot{X}\fppf \arr X\kfl$, in the sense of \cite[Tag~00X0]{stacks-project}; this in turn induces a morphism of topoi 
   \[
   (F_{*}, F^{-1})\colon \sh (\infroot{X}\fppf) \arr \sh (X\kfl)\,.
   \]

\begin{theorem}\label{equiv.topoi}
The morphism of topoi $(F_{*}, F^{-1})\colon \sh (\infroot{X}\fppf) \arr \sh (X\kfl)$ is an equivalence.
\end{theorem}

Putting Theorem~\ref{equiv.topoi} together with Propositions \ref{prop:equivalence-fp} and \ref{prop:equivalence-fppf-fpqc} we obtain one of the main results of this paper.

\begin{corollary}\label{cor:equiv-fp}
If $X$ is a \fs quasi-separated logarithmic scheme, there is an equivalence of additive tensor categories between the categories of \fp sheaves on the Kummer-flat site $X\kfl$ and \fp sheaves on the \irs $\infroot{X}$.\qed
\end{corollary}

\begin{proof}[Proof of Theorem~\ref{equiv.topoi}]
We will apply the following lemma from the Stacks Project.

\begin{lemma}[\hbox{\cite[Tag~039Z]{stacks-project}}]\call{lem:criterion-equivalence-topoi}
Let $\cC$ and $\cD$ be sites and $F\colon \cC \arr \cD$ a functor. Assume that the following conditions hold.
\begin{enumeratea}

\itemref{1} $F$ is continuous and cocontinuous.

\itemref{2} Given two arrows $a$, $b \colon U'\to U$ in $\cC$ such that $Fa = Fb$, there exists a covering $\{f_{i}\colon U_{i}'\to U'\}$ in $\cC$ such that $a\circ f_{i} = b\circ f_{i}$ for every $i$.

\itemref{3} Given two objects $U$ and $U'$ of $\cC$ and a morphism $c\colon FU' \arr FU$ in $\cD$, then there exists a covering $\{f_{i}\colon U_{i}'\to U'\}$ in $\cC$ and morphisms $c_{i}\colon U_{i}'\to U$ such that $Fc_{i} = c\circ Ff_{i}$ for every $i$.

\itemref{4} Given $V \in \cD$, then there exists a covering of $V$ in $\cD$ of the form $\{FU_{i}\to V\}$.
\end{enumeratea}

Then the induced morphism of topoi $(F_{*}, F^{-1})\colon \sh\cD \arr  \sh\cC$ is an equivalence.
\end{lemma}

The fact that $F$ is continuous follows from Lemma~\ref{infty.products} {and Proposition \ref{kfl.morphisms}}.  Showing that it is cocontinuous amounts to proving that for any Kummer-flat map $Z\to X$, any covering $\{\cA_{j}\to \infroot{Z}\}$ of the corresponding object in $\infroot{X}\fppf$ can be refined by the family of maps $\{\infroot{Z_{i}}\to \infroot{Z}\}$, for some Kummer-flat covering $\{Z_{i}\to Z\}$. Clearly this will follow from \refpart{lem:criterion-equivalence-topoi}{4} applied to $Z$ in place of $X$.

Part~\refpart{lem:criterion-equivalence-topoi}{2} (local faithfulness) follows directly from Theorem \ref{thm:equivalence1}, which implies that if two morphisms of fine saturated logarithmic schemes $Z\to Y$ induce equivalent morphisms $\infroot{Z}\to \infroot{Y}$, then they coincide, and part~\refpart{lem:criterion-equivalence-topoi}{3} (local fullness) from Corollary~\ref{cor:kummer->logarithmic}, that implies that every morphism $\infroot{Z}\arr \infroot{Y}$ in $\infroot{X}\fppf$ comes from a morphism $Z\to Y$ of logarithmic schemes. {In fact, in both instances the adjective ``local'' can be removed, but this is not needed to apply the lemma}. All that is left is to prove \refpart{lem:criterion-equivalence-topoi}{4}.

We do it in the form of the following lemma, which therefore will conclude the proof.

\begin{lemma}\label{lem:refine-to-irs}
Let $\cA\to \infroot{X}$ be an object of $\infroot{X}\fppf$. Then there exists a Kummer-flat morphism $Y\to X$ with a factorization $\infroot{Y}\arr \cA \to \infroot{X}$, such that $\infroot{Y}\to \cA$ is fppf.
\end{lemma}

From now on we focus on the proof of this lemma.

After shrinking $X$ in the étale topology, we can assume that $X$ is affine, and that we have a Kato chart $P\to \cO_{X}(X)$ for the logarithmic structure of $X$. Call $U_{\infty}$ and $U_{n}$ the pullbacks of $X^{[\infty]}_{P}$ and $X^{[n]}_{P}$ to $X$ (see \ref{sec:local-models} for the notation); also, set $G_{\infty} \eqdef \mmu_{\infty}(P)$ and $G_{n} \eqdef \mmu_{n}(P)$. Then $\infroot{X} = [U_{\infty}/G_{\infty}]$ and $\radice n X = [U_{n}/G_{n}]$. Let $V \eqdef \cA\times_{\infroot X} U_{\infty}$, so that $\cA = [V/G_{\infty}]$. 

Next we need a technical lemma.

\begin{lemma}
Let $V$ a scheme with an action of a profinite diagonalizable group scheme $G\arr \spec \ZZ$, and $p \in V$. Then there exists a diagonalizable subgroup scheme $H \subseteq G$ of finite index and an open affine $H$-invariant neighborhood of $p$ in $V$.
\end{lemma}

Here by \emph{profinite} we mean that $G$ is a projective limit of finite diagonalizable group schemes, or, equivalently, that the group of characters $\widehat{G}$ is torsion.

\begin{proof}
If $k(p)$ is the residue field of $p$, the action of $G$ on $V$ induces a morphism $\phi_{p}\colon G_{k(p)} \arr V$ sending the origin to $p$. Let $V'$ be a an open affine neighborhood of $p$, consider the inverse image $\phi_{p}^{-1}(V') \subseteq G_{k(p)}$. The subgroup schemes of the form $H_{k(p)}$, where $H$ is a diagonalizable subgroup scheme of finite index in $G$, form a basis of open neighborhoods of the origin in $G_{k(p)}$. By restricting to a diagonalizable subgroup scheme of finite index $H \subseteq G$ such that $H_{k(p)} \subseteq \phi_{p}^{-1}(V')$, we can assume that $G_{k(p)}=\phi_{p}^{-1}(V')$, i.e. the orbit of $p$ is contained in $V'$.

Denote by $\alpha\colon G \times V \arr V$ the action and by $\pi\colon G \times V \arr V$ the projection. Notice that since $\pi$ is an affine integral map, it is closed, and consequently, so is $\alpha$. Let $Z$ be the complement of $V'$ in $V$; then $\alpha\bigl(\pi^{-1}(Z)\bigr)$ is a closed invariant subset of $V$ whose complement contains $p$ and is contained in $V'$. By substituting $V$ with the complement of $\alpha\bigl(\pi^{-1}(Z)\bigr)$ we may assume that $V$ is an open subscheme of an affine scheme.

Set $A \eqdef \cO(V)$ and $\overline{V}\eqdef \spec A$. The action of $G$ on $V$ induces an action of $G$ on $\overline{V}$, and the natural morphism $j\colon V \arr \overline{V}$, is $G$-equivariant. If $V$ were quasi-affine (that is, according to \cite[Définition~5.1.1]{ega2}, it were also quasi-compact) then $j$ would be an open embedding. In general we don't know whether $j$ is always an open embedding; however, if $i\colon V \arr Y$ is an {open} embedding into an affine scheme $Y$, there is an induced morphism $\overline{V} \arr Y$. Let $W$ be the inverse image of $V$ in $\overline{V}$: we obtain a cartesian diagram
   \[
   \begin{tikzcd}
   W\arrow{r}\dar & \overline{V}\dar\\
   V \ar[hookrightarrow]{ur}{j} \rar{i}\uar[dashrightarrow, bend left=40]{s} & Y
   \end{tikzcd}
   \]
Thus $j$ factors as a section $s\colon V \arr W$, which is closed embedding, because $W \arr V$ is affine, followed by the open embedding $W \subseteq \overline{V}$. Let $Z$ be the complement of $W$ in $\overline{V}$, and pick an element $f \in A = \cO(\overline{V})$ that vanishes along $Z$, but not at $s(p)$. Write $f = \sum_{\lambda \in \widehat{G}}f_{\lambda}$, where $\widehat{G}$ is the group of characters of $G$. Since $\widehat{G}$ is torsion the intersection of the kernels of the $\lambda \in \widehat{G}$ with $f_\lambda \neq 0$ is a finite index subgroup scheme $H \subseteq G$; by restricting the action to $H$, we may assume that $f$ is $G$-invariant. Then $G$ acts on $\overline{V}_{f} = \spec A_{f}$, and $\overline{V}_{f} \subseteq W$. The subscheme $j^{-1}(\overline{V}_{f})$ is an affine $G$-invariant neighborhood of $p$, and this concludes the proof.
\end{proof}

Now, let $\{V_{i}\}$ be an open covering of $V$ with open affine subschemes, such that for each $i$ there exists a subgroup scheme $H_{i} \subseteq G_\infty$ of finite index such that $V_{i}$ is $H_{i}$-invariant. Clearly, the induced morphisms $[V_{i}/H_i] \arr [V/G_\infty]$ are affine, \fp and flat. Assume that for each $i$ we have found a \fs logarithmic scheme $Y_{i}$ with a Kummer-flat morphism $Y_{i}\to X$, and a factorization $\infroot{Y_{i}} \arr [V_{i}/H_{i}]$ such that $\infroot{Y_{i}} \arr [V_{i}/H_{i}]$ is fppf. If we set $Y \eqdef \bigsqcup_{i}Y_{i}$, we get a morphism of logarithmic schemes $Y \arr X$, such that the induced base-preserving functor $\infroot{Y} \arr \infroot {X}$ factors as
   \[
   \infroot{Y} = \bigsqcup{\infroot{Y_{i}}} \arr \bigsqcup_{i}[V_{i}/H_{i}] \arr \cA \arr \infroot{X},
   \]
and the base-preserving functor $\infroot{Y} \arr \cA$ is fppf.   

This means that we can replace $\cA$ with $[V_{i}/H_{i}]$; in this case the morphism $\cA \arr \infroot{X}$, which we are assuming to be locally \fp, is affine, hence \fp. This means that $V$ is affine, and \fp over $U_{\infty}$.

Denote by $\Gamma_{n}$ the kernel of the homomorphism $G_{\infty} \arr G_{n}$, and set $V_{n} \eqdef V/\Gamma_{n}$. We have a commutative diagram
   \begin{equation}\label{eq:diagram1}
   \begin{tikzcd}
   V \rar \dar &U_{\infty} \dar\\
   V_{n}\rar &U_{n}\,;
   \end{tikzcd}
   \end{equation}
we claim that for  for sufficiently divisible $n$ the homomorphism $V_{n} \arr U_{n}$ is flat and \fp, and the diagram is cartesian. This is a consequence of the following lemma.

\begin{lemma}\call{lem:invariants-finitely-presented}
Let $G$ be a profinite diagonalizable group over $\spec \ZZ$; write $G = \projlim_{i \in I}G_{i}$, where $I$ is a filtered partially ordered set and the $G_{i}$ are finite diagonalizable groups. Denote by $\Gamma_{i}$ the kernel of the homomorphism $G \arr G_{i}$. 

Suppose that $G$ acts on two rings $R$ and $A$, and let $R \arr A$ a \fp $G$-equivariant ring homomorphism. For each $i \in I$ set $R_{i} = R^{\Gamma_{i}}$ and $A_{i} = A^{\Gamma_{i}}$. Then there exists $i_{0} \in I$ such that for all $i \geq i_{0}$ we have the following.

\begin{enumeratea}

\itemref{1} The induced homomorphism $R_{i} \arr A_{i}$ is \fp.

\itemref{2} The natural homomorphism $A_{i}\otimes_{R_{i}}R \arr A$ is an isomorphism.

\itemref{3} If $R \arr A$ is flat, then $R_{i} \arr A_{i}$ is also flat.

\end{enumeratea}
\end{lemma}

\begin{proof}
Suppose that $A$ has a presentation with $m$ generators $(a_{1}, \dots, a_{m})$. Since $G$ is profinite we have that $A$ is the union of the $A_{i}$; hence for sufficiently large $i$ the generators are $\Gamma_{i}$-invariant. Let $\bfx = (x_{1}, \dots, x_{m})$ be a sequence of indeterminates, and consider the surjective homomorphism $R[\bfx] \arr A$ that sends each $x_{j}$ to $a_{j}$. If we let $\Gamma_{i}$ act on $R[\bfx]$ by acting on the coefficients and leaving the $x_{j}$ invariant, this homomorphism is $\Gamma_{i}$-equivariant.

If the ideal of relations $I \subseteq R[\bfx]$ is generated by $n$ elements, we have an exact sequence
   \[
   R[\bfx]^{\oplus n} \arr R[\bfx] \arr A \arr 0\,;
   \]
by enlarging $i$ we can arrange that the generators be also $\Gamma_{i}$-invariant, and the the sequence above is $\Gamma_{i}$-equivariant. Taking invariants under $\Gamma_{i}$ is an exact functor, because $\Gamma_{i}$ is diagonalizable; clearly $R[\bfx]^{\Gamma_{i}} = R_{i}[\bfx]$, so for all sufficiently large $i$ we have an exact sequence
   \[
   R_{i}[\bfx]^{\oplus n} \arr R_{i}[\bfx] \arr A_{i} \arr 0\,.
   \]
This proves \refpart{lem:invariants-finitely-presented}{1}.

Part \refpart{lem:invariants-finitely-presented}{2} is immediate when $A = R[\bfx]$ and $\Gamma_{i}$ acts leaving the indeterminates fixed. In the general case it follows from the exact sequence above, and right exactness of the tensor product.

For part~\refpart{lem:invariants-finitely-presented}{3} see \cite[Théorème~11.2.6]{ega43}.
\end{proof}

We are now ready to finish the proof of Lemma \ref{lem:refine-to-irs} (and this will also conclude the proof of Theorem~\ref{equiv.topoi}). {Recall that we have reduced the statement to considering the case of an object of  $\infroot{X}\fppf$ of the form $\cA=[V/G_\infty]\to \infroot{X} = [U_{\infty}/G_{\infty}]$, where $V$ is affine and the $G_\infty$-equivariant morphism $V\to U_\infty$ is finitely presented. By the previous lemma} we can find an $n$ sufficiently large such that $V_{n}=V/\Gamma_n$ is flat and \fp over $U_{n}$, and diagram~(\ref{eq:diagram1}) is cartesian. The scheme $U_{n}$ has a tautological \fs logarithmic structure coming from the morphism $U_{n} \arr \radice{n}X$; with this logarithmic structure, the morphism $U_{n} \arr X$ is obviously Kummer-flat. We have $\infroot{U_{n}} = [U_{\infty}/\Gamma_{n}]$ by Corollary~\ref{cor:local-model-chart}; if we give to $V_{n}$ the logarithmic structure pulled back from $U_{n}$, then $V_{n}\to X$ is Kummer-flat, and  by Proposition~\ref{prop:base-change} we have
   \[
   \infroot{V_{n}} \cong V_{n}\times_{U_{n}} \infroot{U_{n}} \cong [V/\Gamma_{n}].\,
   \]
So $[V/\Gamma_{n}]$ is an \irs. Since the morphism $[V/\Gamma_{n}] \arr [V/G_{\infty}]$ is representable, faithfully flat and \fp, this concludes the proof.
\end{proof}

As a variation on this one can consider the small Kummer-étale site $X\ket$ of a \fs logarithmic scheme $X$, whose objects are Kummer-étale morphism $Y \arr X$ (see \cite{illusie-kato-nakayama,niziol-k-theory}). The definition of Kummer-étale is obtained by replacing ``flat'' with  ``étale'' in the definition of the Kummer-flat morphism, and assuming that in diagram (\ref{eq:diagram2}) the order of the cokernel of $P\gr \arr Q\gr$ is not divisible by the characteristic of the residue field of any point of $X$. 

We can link the site $X\ket$ with the small étale site $\infroot X\et$, whose objects are representable étale morphisms $\cA \arr \infroot{X}$. More precisely one can prove the following.

\begin{theorem}\label{thm:equivalence-ket}
If $Y \arr X$ is a Kummer-étale morphism of \fs logarithmic schemes, the induced base-preserving functor $\infroot Y \arr \infroot X$ is \'{e}tale and representable. The resulting functor $X\ket \arr \infroot X\et$ induces an equivalence of topoi $\sh (\infroot{X}\et) \simeq \sh (X\ket)$.\qed
\end{theorem}

Therefore we also obtain an equivalence of categories $\sh(\infroot{X}\et) \simeq \sh (X\ket)$. In characteristic~$0$ (that is, if $X$ is a scheme over $\QQ$) the stack $\infroot{X}$ is a limit of De\-ligne--Mumford stacks, and one can show that there is also an equivalence of categories $\FP(\infroot{X}\et)$ and $\FP(\infroot X)$; however, this fails in positive characteristic. This is a reflexion of the fact that the Kummer-étale site works very well in characteristic~$0$, but is not adequate for many applications in positive characteristic.

Using Lemma~\ref{lem:refine-to-irs} we can also prove a converse to Proposition~\ref{kfl.morphisms}.

\begin{theorem}\label{thm:flat<->Kummer-flat}
Let $f\colon Y \arr X$ be a morphism of \fs logarithmic schemes, such that the underlying morphism of schemes is locally of finite presentation. Then $\infroot{f}\colon \infroot{Y} \arr \infroot{X}$ is flat, locally finitely presented and representable if and only if $f$ is Kummer-flat.
\end{theorem}

\begin{proof}
The ``if'' part of the statement follows directly from Proposition \ref{kfl.morphisms}. Let us prove the converse.

Let $\overline{y}\colon \spec\Omega\arr  Y$ be a geometric point of $Y$, and $\overline{x} \eqdef f\circ \overline{y}$. Set $P \eqdef M_{X, \overline{x}}$ and $Q \eqdef M_{Y,\overline{y}}$. By refining $\underline X$ and $\underline Y$ in the étale topology we may assume that they are affine, and there exists a commutative diagram
   \[
   \begin{tikzcd}
   Y \rar\dar{f} &X_{Q}\dar\\
   X \rar & X_{P}
   \end{tikzcd}
   \]
in which the rows are strict, and the right-hand column is induced by $f^{*}\colon  P \arr Q$. Set $U_{\infty} \eqdef X \times_{X_{P}}X_{P}^{[\infty]}$ and $V_{\infty} \eqdef Y \times_{X_{Q}}X_{Q}^{[\infty]}$, so that $\infroot{X} = [U_{\infty}/\mmu_{\infty}(P)]$ and $\infroot{Y} = [V_{\infty}/\mmu_{\infty}(Q)]$. We have a commutative (non-cartesian) diagram
   \[
   \begin{tikzcd}
   V_{\infty} \rar\dar& U_{\infty}\dar\\
   \infroot{Y} \rar &\infroot{X}\hsmash\,.
   \end{tikzcd}
   \]

We will prove the theorem by first showing that $f^{*}\colon P \arr Q$ is Kummer, and then that $f$ is log-flat.

\step{1: $f^{*}\colon P \arr Q$ is injective} For this purpose we may base change through $\overline{x}$, and assume that $X = \spec \Omega$ with the logarithmic structure given by $M_{X} = \cO_{X}^\times \oplus P$, so that $\infroot{X} = \infroot{P/\Omega} = \bigl[\spec( \Omega[P_{\QQ}]/(P^{+}))/\mmu_{\infty}(P)\bigr]$. The flatness of $\infroot{Y} \arr \infroot{X}$ implies the flatness of the composite $V_{\infty} \arr \infroot{Y} \arr \infroot{X}$, which equals the composite $V_{\infty} \arr U_{\infty} \arr \infroot{X}$. In turn, since flatness is preserved by base-change, this implies the flatness of the projection $V_{\infty}\times_{\infroot{X}}U_{\infty} \arr U_{\infty}$.

Note that $V_{\infty}\times_{\infroot{X}}U_{\infty} = V_{\infty}\times\mmu_{\infty}(P)$; the projection
   \begin{equation}\label{eq:flat-map-1}
   V_{\infty}\times\mmu_{\infty}(P) \arr U_{\infty}\, ,
   \end{equation}
which is flat by the argument above, is obtained by composing the morphism $V_{\infty}\times\mmu_{\infty}(P) \arr U_{\infty}\times\mmu_{\infty}(P)$ with the action $U_{\infty}\times\mmu_{\infty}(P) \arr U_{\infty}$.

Now, consider a prime $p$, which equals the characteristic of $\Omega$ if this is positive, and is arbitrary if it is $0$. For any torsion abelian group $A$ we will denote by $A\{p\}$ the $p$-primary torsion part, and by $A\{p'\}$ the prime to $p$ part; clearly $A = A\{p\} \oplus A\{p'\}$. If $G$ is a profinite diagonalizable group scheme over $\Omega$, with character group $\widehat{G}$, we call $G\{p\}$ and $G\{p'\}$ the diagonalizable group schemes over $\Omega$ with character groups $\widehat{G}\{p\}$ and $\widehat{G}\{p'\}$ respectively, so that $G = G\{p\} \times G\{p'\}$. By abuse of notation, we will use $\mmu_{\infty}(P)$ to denote the fibered product $\spec \Omega \times_{\spec \ZZ} \mmu_{\infty}(P)$. We claim that the morphism
   \begin{equation}\label{eq:flat-map-2}
   V_{\infty}\times\mmu_{\infty}(P)\{p\} \arr U_{\infty}
   \end{equation}
obtained by composing the morphism (\ref{eq:flat-map-1}) with the embedding $V_{\infty}\times\mmu_{\infty}(P)\{p\} \subseteq V_{\infty}\times\mmu_{\infty}(P)$, is flat. Since (\ref{eq:flat-map-1}) is flat, it is enough to show that the embedding $\mmu_{\infty}(P)\{p\} \subseteq \mmu_{\infty}(P)$ is flat. But this is clear, since this is obtained by taking a limit of the embeddings $\mmu_{\infty}(P)\{p\} \subseteq \mmu_{\infty}(P)\{p\} \times \mmu_{n}(P)$ for $n$ not divisible by $p$, and each $\mmu_{n}(P)$ is étale over $\Omega$.

Since (\ref{eq:flat-map-2}) is flat and $U_{\infty} = \spec \Omega[P_{\QQ}]/(P^{+})$ is topologically a point, we see that (\ref{eq:flat-map-2}) is faithfully flat. Now write $Y = \spec R$: then (\ref{eq:flat-map-2}) corresponds to a homomorphism of $\Omega$-algebras
   \[
   \Phi\colon \Omega[P_{\QQ}]/(P^{+}) \arr
   (R \otimes_{\Omega[Q]}{\Omega[Q_{\QQ}]} ) \otimes_{\Omega} \Omega[P\gr\otimes(\QQ/\ZZ)\{p\}]\hsmash{\,,}
   \]
which must then be injective. The homomorphism $\Phi$ can be described as follows. If $a \in P_{\QQ}$ we denote by $t^{a}$ the corresponding element of $\Omega[P_{\QQ}]/(P^{+})$, and if $b \in Q_{\QQ}$ we denote by $u^{b}$ the corresponding element of $\Omega[Q_{\QQ}]$. We also denote by $\pi \colon P_{\QQ} \arr P_{\QQ}\gr \arr  (P\gr_{\QQ}/P\gr)\{p\}\simeq P\gr\otimes(\QQ/\ZZ)\{p\}$ the projection, and by $v^{c}$ the element of $\Omega[P\gr\otimes(\QQ/\ZZ)\{p\}]$ corresponding to some $c \in P\gr\otimes(\QQ/\ZZ)\{p\}$. Then we have
   \[
   \Phi(t^{a}) = 1 \otimes u^{f^{*}(a)}\otimes v^{\pi(a)}\hsmash{\,.}
   \]

Now, assume that $f^{*}\colon P \arr Q$ is not injective, and take $a$, $a' \in P$ with $a \neq a'$ and $f^{*}(a) = f^{*}(a)$. Fix a prime $l$ different from $p$. Then for sufficiently large $m$ we have $t^{{a}/{l^{m}}} \neq t^{{a'}/{l^{m}}}$, while $f^{*}(a/l^{m}) =  f^{*}(a'/l^{m})$ and $\pi({a}/{l^{m}}) = 0 = \pi({a'}/{l^{m}})$. Therefore $\Phi(t^{{a}/{l^{m}}}) = \Phi(t^{{a'}/{l^{m}}})$, and this gives a contradiction.

\step{2: $P \arr Q$ is Kummer} Lemma~\ref{lem:refine-to-irs} implies the existence of a Kummer-flat morphism $g\colon Z \arr X$ such that $\infroot{g}\colon \infroot{Z} \arr \infroot{X}$ factors as
   \[
   \begin{tikzcd}[column sep = 5ex]
   \infroot{Z} \rar{\psi} &\infroot{Y} \rar{\infroot{f}} &\infroot{X}\,,
   \end{tikzcd}
   \]
where $\psi$ is a morphism of stacks that is representable, flat and surjective. By Theorem~\ref{thm:equivalence2} (and the remark following it) $\psi$ is a morphism of \irss, and there exists a morphism of logarithmic schemes $h\colon Z \arr Y$ with $\infroot{h} = \psi$, and $g = f \circ h$. Let $\overline{z}\colon \spec \Omega \arr Z$ be a geometric point; set $\overline{y} = h(\overline{z})$ and $\overline{x} = f(\overline{y})$. We need to show that $f^{*}\colon M_{X, \overline{x}} \arr M_{Y, \overline{y}}$ is Kummer. However, the composite $M_{X, \overline{x}} \xarr{f^{*}} M_{Y, \overline{y}} \xarr{h^{*}} M_{Z, \overline{z}}$ is Kummer, and $h^{*}\colon M_{Y, \overline{y}} \arr M_{Z, \overline{z}}$ is injective (from step $1$ applied to $\infroot{Z}\to \infroot{Y}$). Now the result follows from an easy check.

\step {3: $f\colon Y \arr X$ is log-flat} 
By definition, the morphism $Y \arr X$ factors through the root stack $\radice{Q}{X}$  of $X$ with respect to the Kummer homomorphism $P\to Q$ (see the first paragraphs of Section \ref{sec:irss-log-schemes}). We also have $\infroot{Y} = \projlim_{n}\radice{n}{Y}$ and $\infroot{X} = \varprojlim_{n} \radice{\frac{1}{n}Q}{X}$. If $n$ is a positive integer, we have a commutative diagram
   \[
   \begin{tikzcd}
   {}\infroot{Y}\dar{\infroot{f}}\rar &
   {}\radice{n}{Y}\dar{f_{n}}\rar&
   Y\rar\dar\ar[bend left = 40]{dd}{f}
   & X_{Q}\ar{dd}\\
   {}\infroot{X} \rar &
   {}\radice{\frac{1}{n}Q}{X}\rar&
   {}\radice{Q}{X}\dar&\\
   &&X \rar & 
   X_{P}
   \end{tikzcd}
   \]
in which the two left squares are cartesian. From this it follows that $\radice{n}Y \arr \radice{\frac{1}{n}Q}{X}$ is representable. We claim that it is also flat when $n$ is sufficiently divisible.

Set $U_{n} \eqdef X\times_{X_{P}}X_{\frac{1}{n}Q}$ and $W_{n} \eqdef U_{n}\times_{\radice{\frac{1}{n}Q}{X}} \radice{n}Y$; since $U_{n} \arr \radice{\frac{1}{n}Q}{X}$ is flat and surjective, it is enough to show that $W_{n} \arr U_{n}$ is flat when $n$ is sufficiently divisible. We have $\projlim_{n} U_{n} = U_{\infty}$, $\projlim_{n}\bigl( \radice{\frac{1}{n}Q}{X}\bigr) = \infroot{X}$, and $\projlim_{n} \radice{n}Y = \infroot Y$, hence $\projlim_{n}  W_{n} = U_{\infty}\times_{\infroot{X}} \infroot{Y}$. Furthermore, if $m \mid n$ we have a diagram
   \[
   \begin{tikzcd}
   W_{n} \ar{rrr}\ar{ddd}\ar{rd} &
   &{}&U_{n}\ar{ddd}\ar{ld} \\
   {}&{}\radice{n}Y
   \rar\dar&\radice{\frac{1}{n}Q}{X}\dar&{}\\
   {}&{}\radice{m}Y \rar&\radice{\frac{1}{m}Q}{X}&{}\\
   W_{m}\ar{rrr}\ar{ru}&
   &{}&U_{m}\ar{lu}
   \end{tikzcd}
   \]
in which the upper, middle and lower squares are cartesian; hence $W_{n} = W_{m}\times_{U_{m}}U_{n}$.
Since the projection $U_{\infty}\times_{\infroot{X}} \infroot{Y}\to U_{\infty}$ is flat, because it is obtained by base change from $\infroot{f}$, we can apply \cite[Théorème~11.2.6]{ega43} and conclude that $W_{n} \arr U_{n}$ is flat when $n$ is sufficiently divisible. Choose such an $n$.

Now set $V_{n} \eqdef Y\times_{X_{Q}}X_{\frac{1}{n}Q}$ and $V_{n}' \eqdef V_{n}\times_{\radice{\frac{1}{n}Q}{X}}U_{n}$. We have a commutative diagram
   \[
   \begin{tikzcd}
   V_{n}' \ar{rr}\dar && U_{n}\dar\\
   V_{n} \rar & \radice{n}Y \rar & \radice{\frac{1}{n}Q}X\,;
   \end{tikzcd}
   \]
let us endow $U_{n}$, $V_{n}$ and $V_{n}'$ with the tautological logarithmic structures coming from the maps into $\radice{\frac{1}{n}Q}X$ (notice that in the case of $V_{n}$ and $V_{n}'$ these coincide with the ones coming from the maps into $\radice{n}Y$). The maps $V_{n}' \arr U_{n}$ and $V_{n}' \arr V_{n}$ are flat and strict, hence Kummer-flat. On the other hand $V_{n} \arr Y$ and $U_{n} \arr X$ are Kummer-flat by definition, so the composites $V'_{n} \arr Y$ and $V'_{n} \arr X$ are Kummer-flat. From \cite[Theorem~0.2]{illusie-nakayama-tsuji} (which we recall below) we conclude that $f\colon Y \arr X$ is log-flat. This concludes the proof of Theorem~\ref{thm:flat<->Kummer-flat}.
\end{proof}

\begin{lemma}\cite[Theorem~0.2]{illusie-nakayama-tsuji} 
Let $f\colon X\to Y$ and $g\colon Y\to Z$ be morphisms of \fs logarithmic schemes, and assume that $f$ is surjective and Kummer. If $f$ and $g\circ f$ are log-flat (resp. log-\'etale), then $g$ is also log-flat (resp. log-\'etale). \qed
\end{lemma}

Analogous arguments prove the following.

\begin{theorem}\label{thm:etale<->Kummer-etale}
Let $f\colon Y \arr X$ be a morphism of \fs logarithmic schemes, such that the underlying morphism of schemes is locally of finite presentation. Then $\infroot{f}\colon \infroot{Y} \arr \infroot{X}$ is étale if and only if $f$ is Kummer-étale.\qed
\end{theorem}

\section{Parabolic sheaves on \fs logarithmic schemes}\label{sec:parabolic}

In this section we extend the parabolic interpretation of \qc sheaves on finite root stacks given in \cite{borne-vistoli1} to \irss. Let us review some definitions from \cite{borne-vistoli1}, specialized to the case of integral monoids, the only one that is of interest here.

Let $P$ be an integral monoid. We define the \emph{weight lattice} $P\wt$ as the set $P\gr$, with the partial order relation defined $x \leq y$ when $y - x \in P \subseteq P\gr$. We will interpret $P\wt$ as a category in which there is at most one arrow between any two objects, in the usual way.

If $X$ is a scheme and $A$ is a sheaf of integral monoids on $X\et$, we define a fibered category $A\wt \arr X\et$. If $U \arr X$ a map with $U$ affine, then $A\wt(U)$ is the set $A\gr(U)$ with the ordering defined by $x \leq y$ if $y - x \in A(U)$. We will think of $A\wt$ as a fibered category on $X\et$ whose fibers are partially ordered sets.

Let us fix a \fs logarithmic scheme $X$ with a \df structure $L\colon A\to \div_{X\et}$.

\begin{definition}\label{def:parabolic-sheaf1}
Assume at first that there is a global chart $P\arr A(X)$. A \emph{parabolic sheaf} $E$ on the logarithmic scheme $X$ is a functor $E\colon P_{\QQ}\wt\to \catqcoh X $ that we denote by $a\mapsto E_{a}$, for $a$ an object or an arrow of $P_{\QQ}\wt$, with an additional datum for any $p \in P\gr$ and $a \in P_{\QQ}\gr$ of an isomorphism of $\cO_{X}$-modules
   \[
   \rho^{E}_{p,a}\colon E_{p+a}\cong L_{p}\otimes E_{a}
   \]
called the pseudo-periods isomorphism.

These isomorphism are required to satisfy some compatibility conditions. Let $p$, $p'\in P\gr$, $r \in P$, $q \in P_{\QQ}$ and $a \in P_{\QQ}\gr$. Then the following diagrams
   \[
   \begin{tikzcd}
   E_{a} \rar{E_{r}}\dar &[4ex] E_{r+a}\dar{\rho^{E}_{r,a}}\\
   \cO_{X}\otimes E_{a}\rar{\sigma_{r}\otimes \id} & L_{r}\otimes E_{a}
   \end{tikzcd}
   \]
   \[
   \begin{tikzcd}
   E_{p+a}\rar{\rho^{E}_{p,a}}\dar[swap]{E_{q}} &[4ex]  L_{p}\otimes E_{a}\dar{\id\otimes E_{q}}\\
   E_{p+q+a}\rar{\rho^{E}_{p,q+a}} & L_{p}\otimes E_{q+a}
   \end{tikzcd}
   \]
   \[
   \begin{tikzcd}
   E_{p+p'+a}\rar{\rho^{E}_{p+p',a}}\dar[swap]{\rho^{E}_{p,p'+a}} &[4ex] L_{p+p'}\otimes E_{a}
   \dar{\mu_{p,p'} \otimes \id} \\
   L_{p}\otimes E_{p'+a}\rar{\id\otimes \rho^{E}_{p',a}} & L_{p}\otimes L_{p'}\otimes E_{a},
  \end{tikzcd}
   \]
where $\mu_{p,p'}\colon L_{p+p'}\cong L_{p}\otimes L_{p'}$ is the natural isomorphism given by the symmetric monoidal functor $L$,  are commutative. Furthermore we assume that the composite
   \[
   E_{a}=E_{0+a}\xarr{\rho^{E}_{0,a}}  L_{0}\otimes E_{a}\cong \cO_{X}\otimes E_{a}
   \]
coincides with the natural isomorphism $E_{a}\cong \cO_{X}\otimes E_{a}$.
\end{definition}

A homomorphism $E' \arr E$ of parabolic sheaves on $X$ is given by a base-preserving natural transformation of functors $P_{\QQ}\wt\to \catqcoh X$, which is moreover compatible with the pseudo-periods isomorphisms.

As in the case of parabolic sheaves with fixed weights, the definition extends to the general case (without a global chart), where one requires the commutativity of the diagrams and compatibility of $\rho^{E}$ with pullback. One shows that in the presence of a global chart, the corresponding categories are equivalent (the analogue of \cite[Proposition~5.10]{borne-vistoli1}).

More precisely, denote by $\qcoh_{X\et}$ the fibered category on $X\et$ associated with the pseudo-functor from $X\et$ into abelian categories, sending each map $U \arr X$ with $U$ affine into the category of quasi-coherent sheaves on $U$. 

\begin{definition}
A \emph{parabolic sheaf} $E$ on $X$ consists of the following data.

\begin{enumeratea}

\item A cartesian base-preserving cartesian functor $E\colon A_{\QQ}\wt \arr \qcoh_{X\et}$,  denoted by $a \arr E_{a}$.

\item For any $U \arr X$ in $X\et$, any $a \in A\wt(U)$ and $b \in A_{\QQ}\wt(U)$, an isomorphism of $\cO_{U}$-modules
   \[
   \rho^{E}_{a,b}\colon E_{a+b} \simeq L_{a} \otimes E_{b}.
   \]

\end{enumeratea}

These data are required to satisfy the conditions analogous to those of Definition~\ref{def:parabolic-sheaf1}, and the following.

If $f\colon U \arr V$ is an arrow in $X\et$, $a \in A\wt(V)$ and $b \in A_{\QQ}\wt(V)$, then the isomorphism
   \[
   \rho^{E}_{f^{*}a,f^{*}b}\colon 
   E_{f^{*}(a+b)} = E_{f^{*}a + f^{*}b} \simeq L_{f^{*}a} \otimes E_{f^{*}b}
   \]
is the pullback of $j^{E}_{a, b}\colon E_{a+b} \simeq L_{a} \otimes E_{b}$ (note that $f^{*}E_{a}\cong E_{f^{*}a}$ for every $a \in  A_{\QQ}\wt(V)$ since $E$ is cartesian).
\end{definition}

This gives an abelian category $\catpar X$ of parabolic sheaves on $X$.

The following is an analogue of \cite[Theorem~6.1]{borne-vistoli1}.

\begin{theorem}\label{BV.rational}
There is an equivalence of abelian categories between $\catpar X$ and $\catqcoh {\infroot X}$.
\end{theorem}

\begin{proof}
Let us describe the functor $\Phi\colon \catqcoh {\infroot X}\arr \catpar X $. Denote by $\pi\colon \infroot{X} \arr X$ the projection, and by $\Lambda\colon A_{\QQ}\to \Div_{{\infroot X}\et}$ the universal \df structure on $\infroot{X}$ (that we denoted by $\widetilde{L}$ in Section \ref{sec:reconstruction}). Moreover consider the induced functor $A_{\QQ}\wt\to \pic_{{\infroot{X}}\et}$; by abuse of notation we will still denote it by $\Lambda$.

Let $F\in \catqcoh {\infroot X}$ be a quasi-coherent sheaf, $U\to X$ be an \'{e}tale morphism with $U$ affine, and $v \in A_{\QQ}\wt(U)$ be a section on $U$. We define a sheaf $(\Phi F)_{v}$ as
\[
 (\Phi F)_{v}=\pi_{*}(F\otimes_{\infroot{X}}\Lambda_{v}).
\]
This is a quasi-coherent sheaf on $X$ because of Proposition \ref{prop:coherent-pushforward}.

Now for $v \in A_{\QQ}\wt(U)$ and $a \in A_{\QQ}(U)$ we get a morphism $(\Phi F)_{v}\to (\Phi F)_{v+a}$ by applying $\pi_{*}$ to the morphism
\[
F\otimes  \Lambda_{v}\to (F\otimes \Lambda_{v})\otimes \Lambda_{a} \cong F\otimes \Lambda_{v+a}
\]
induced by the section $s_{a}$ of $\Lambda_{a}$.

Finally, for $a \in A\wt(U)$ and $b\in A_{\QQ}\wt(U)$, the isomorphism 
\[
\rho_{a,b}^{\Phi F}\colon (\Phi F)_{a+b}\cong L_{a}\otimes (\Phi F)_{b}
\]
is given by the following composition, where we use the projection formula for $\pi$:
\begin{align*}
(\Phi F)_{a+b}  & = \pi_{*}(F\otimes \Lambda_{a+b})\\
 & \cong  \pi_{*}(F\otimes \Lambda_{b}\otimes \Lambda_{a})\\
 & \cong \pi_{*}(F\otimes \Lambda_{b}\otimes \pi^{*}L_{a})\\
 & \cong L_{a}\otimes \pi_{*}(F\otimes \Lambda_{b})\\
 & \cong L_{a}\otimes (\Phi F)_{b}.
\end{align*}
One easily verifies that this gives a parabolic sheaf on $X$, and that this construction can be extended to an additive functor $\catqcoh \infroot X\to \catpar X$.

Let us also describe the quasi-inverse $\Psi\colon \catpar X\arr \catqcoh {\infroot X}$, that we can construct \'etale locally.  Hence we may assume that $X$ has a chart $X\to [\sz{P}/\widehat{P}]$ (in the sense of Definition \ref{def:chart}). Call $\eta\colon E\to X$ the corresponding $\widehat{P}$-torsor, and $A=\eta_*\cO_E$, a $P\gr$-graded $\cO_X$-algebra. We leave it to the reader to check that the infinite root stack has a quotient presentation as
\[
\infroot{X}\cong [(E\times_{\sz{P}}\sz{P_\QQ})/\widehat{P_\QQ}]=[\underline{\spec}_X(A\otimes_{\ZZ[P]}\ZZ[P_\QQ])/\widehat{P_\QQ}]
\]
where the action is via the induced $P\gr_\QQ$-grading on $A\otimes_{\ZZ[P]}\ZZ[P_\QQ]$. This is analogue to Corollary \ref{cor:local-model-chart} and \cite[Proposition 4.13]{borne-vistoli1}. Hence, a quasi-coherent sheaf on $\infroot{X}$ can be described as a $P\gr_\QQ$-graded quasi-coherent sheaf of $A\otimes_{\ZZ[P]}\ZZ[P_\QQ]$-algebras on $X$.

Given a parabolic sheaf $E\colon P\wt_\QQ\arr \catqcoh X$, consider the $P\gr_\QQ$-graded quasi-coherent sheaf $\bigoplus_{a\in P\gr_\QQ}E_a$ on $X$. This has a structure of $A$-module, coming from the isomorphisms $E_a\otimes L_p\cong E_{a+p}$ (note that $A\cong \bigoplus_{p\in P\gr}L_p$), and a structure of $\sz{P_\QQ}$-algebra, coming from the maps $E_q\colon E_{a} 
\to E_{a+q}$ for $a\in P\gr_\QQ$ and $q\in P_\QQ$. By the properties of a parabolic sheaf, the two actions of $\ZZ[P]$ via the maps to $\ZZ[P]\to A$ and $\ZZ[P]\to \ZZ[P_\QQ]$ are compatible, and induce a structure of $P\gr_\QQ$-graded quasi-coherent sheaf of $A\otimes_{\ZZ[P]}\ZZ[P_\QQ]$-modules. Hence we obtain the desired object $\Psi(E)\in \catqcoh {\infroot X}$.

The rest of the proof of Theorem $6.1$ in \cite{borne-vistoli1} goes through without changes.
\end{proof}

\subsection{Finitely presented parabolic sheaves} We are interested in characterizing parabolic sheaves corresponding to \fp sheaves on $\infroot{X}$. Being \fp is a local condition, so we may assume that $X$ is quasi-compact, and that there is a global chart $P \arr A(X)$, where $P$ is a \fs sharp monoid. Suppose that $F$ is a \qc sheaf on $\infroot{X}$, and call $E\colon P_{\QQ}\wt \arr \catqcoh X$ the corresponding parabolic sheaf. For each $a \in P\wt_{\QQ}$ we have $E_{a} = \pi_{*}(F\otimes_{\cO_{\infroot X}}\Lambda_{a})$; hence, by \refall{prop:coherent-pushforward}{2} each $E_{a}$ is \fp on $X$. However, the converse is not true.

\begin{example}
Let $X$ be the standard logarithmic point (\ref{ex:log-points}), and consider the parabolic sheaf $E \in \catpar X$ given by $E_{q}=k$ for $q \in \ZZ$ and $E_{q}=0$ otherwise, with the only possible maps and the trivial quasi-periods isomorphisms. Although $E_{q}$ is finitely presented for every $q \in \QQ$, we claim that $E$ is not finitely presented as a quasi-coherent sheaf on $\infroot X$.

In fact it is immediate from \ref{prop:limit-finitely-presented} and the explicit description of the pullback along $\pi\colon \infroot X\to \radice{n} X$ given below that a finitely presented parabolic sheaf $F \in \catpar X$ has the property that for any $q \in \QQ$, for small enough $q'>0$ we have $F_{q}\cong F_{q+q'}$, and this is not true for the sheaf $E$.

This property is generalized by \ref{thm:char-fp-parabolic} below.
\end{example}

It is, however, true for the finite root stack $\radice{n}{X}$: if $E\colon \tfrac{1}{n}P\wt \arr \catqcoh X$ is a parabolic sheaf, such that $E_{a}$ is \fp on $X$ for all $a \in \tfrac{1}{n}P\wt$, then the corresponding \qc sheaf on $\radice{n}{X}$ is \fp. This follows easily from the construction of the functor from parabolic sheaves to sheaves on $\radice{n}{X}$ given in \cite{borne-vistoli1}: by inspecting the construction, one sees that the sheaf $\Psi E \in \catqcoh \radice{n} X$ corresponding to the parabolic sheaf $E$ is constructed by taking a finite direct sum of the sheaves $E_{a}$ on a smooth cover of the stack $\radice{n} X$, and then by using descent. Since the sheaves $E_{a}$ are finitely presented and the sum is finite, $\Psi E$ is also finitely presented. This fact can be exploited as follows.

Let $n$ be a positive integer; denote by $\pi_{n}\colon \infroot{X} \arr \radice{n}{X}$ the projection. The sheaf $\pi_{n*}F$ corresponds to the restriction of $E$ to $\frac{1}{n}P\wt$. Set $F[n] \eqdef \pi_{n}^{*}\pi_{n*}F$, and call $E[n]\colon P_{\QQ}\wt \arr \catqcoh X$ the corresponding parabolic sheaf. We have that $E_{a}$ is \fp for all $a \in P_{\QQ}\wt$ if and only if each $F[n]$ is \fp, because of the fact above. Furthermore whenever $m \mid n$ the morphism $\pi_{m}$ has a factorization $\infroot{X} \xarr{\pi_{m}}\radice {n}{X} \xarr{\pi^{m}_n} \radice{m}{X}$; the isomorphism $\pi^{m}_{n*}\pi_{m*}F \simeq \pi_{n*}F$ gives a map $\pi_{m*}F \arr \pi^{m}_{n*}\pi_{n*}F$; applying $\pi_{m}^{*}$ we obtain a morphism $F[m]\arr F[n]$; this defines an inductive system of sheaves $\{F[n]\}$ on $\infroot{X}$. The natural homomorphisms $F[n] \arr F$ induce a homomorphism $\indlim_{n} F[n] \arr F$.

\begin{proposition}
The homomorphism $\indlim_{n} F[n] \arr F$ is an isomorphism.
\end{proposition}

\begin{proof}
The statement is local in the étale topology, hence we can assume that $X = \spec R$ is affine, and that the logarithmic structure is induced by a morphism $X \arr X_{P}$. Set $A \eqdef R\otimes_{\ZZ[P]}\ZZ[P_{\QQ}]$, $A_{n} \eqdef R\otimes_{\ZZ[P]}\ZZ[\frac{1}{n}P]$ and $G \eqdef \mmu_{\infty}(P)$; moreover call $G_{n}\eqdef \mmu_{\infty}(\frac{1}{n}P)\subseteq G$, i.e. the kernel of the projection $\mmu_{\infty}(P)\to \mmu_{n}(P)$.

Then $F$ corresponds to a $G$-equivariant $A$-module $M$, the pushforward $\pi_{n*}F$ corresponds to the ($\mmu_{n}(P)$-equivariant) $A_{n}$-module $M^{G_{n}}$, and $F[n]$ to the tensor product $M^{G_{n}} \otimes_{\ZZ[\frac{1}{n}P]} \ZZ[P_{\QQ}]$; so it is enough to check that the natural morphism $\indlim_{n}(M^{G_{n}} \otimes_{\ZZ[\frac{1}{n}P]} \ZZ[P_{\QQ}]) \arr M$ is an isomorphism. We have a factorization
   \[
   \bigcup_{n}M^{G_{n}} = \indlim_{n}M^{G_{n}} \arr \indlim_{n}(M^{G_{n}} \otimes_{\ZZ[\frac{1}{n}P]} 
   \ZZ[P_{\QQ}]) \arr M
   \]
of the embedding $\bigcup_{n}M^{G_{n}} \subseteq M$. Since $\ZZ[P_{\QQ}] = \bigcup_{n}\ZZ[\tfrac{1}{n}P]$, we have that the map $\indlim_{n}M^{G_{n}} \arr \indlim_{n}(M^{G_{n}} \otimes_{\ZZ[\frac{1}{n}P]}\ZZ[P_{\QQ}])$ is surjective; hence it is enough to show that $\bigcup_{n}M^{G_{n}} = M$. This is standard, using the coaction $\mu\colon M \arr M\otimes \ZZ[P\gr_\QQ/P\gr]$. We have that $\bigcup_n\ZZ[\frac{1}{n}P\gr/P\gr] = \ZZ[P\gr_\QQ/P\gr]$. If $m \in M$, write $\mu(m) = \sum_{i}m_{i}\otimes a_{i}$, and choose $n$ such that $a_{i} \in \ZZ[\frac{1}{n}P{\gr}/P\gr]$ for all $i$. Then $m \in M^{G_{n}}$.
\end{proof}

So if $F$ is \fp each of the $F[n]$ is \fp, and $F[n] = F$ for some $n$ (or, equivalently, for all sufficiently divisible $n$). It remains to translate this into a criterion for the parabolic sheaf $E$. For this we need a formula for $E[n]$.

Let $\catpar_{n}X$ be the category of parabolic sheaves on $X$ with weights in $\tfrac{1}{n}P$, as defined in \cite[Definition~5.6]{borne-vistoli1}; then we have an equivalence $\catpar_{n}X \simeq \catqcoh \radice{n}{X}$. The functor $\pi_{n*}\colon \catqcoh\infroot{X} \arr \radice{n}{X}$ corresponds to the restriction functor $\Res_{n}\colon \catpar X \arr \catpar_{n}X$ that sends each $E\colon P_{\QQ}\wt \arr \catqcoh X$ into its restriction to $\tfrac{1}{n}P\wt \arr \catqcoh X$. Let us define an induction functor $\Ind_{n}\colon \catpar_{n}X \arr \catpar X$.

Let $E'\colon \tfrac{1}{n}P\wt \arr \catqcoh X$ be a parabolic sheaf with coefficients in $\tfrac{1}{n}P$. For each $a \in P_{\QQ}$ consider the subset
   \[
   \tfrac{1}{n}P^{\leq a} \eqdef \{u \in \tfrac{1}{n}P \mid u \leq a\}
   \]
and set
   \[
   (\Ind_{n}E')_{a} = \colim_{u\in \frac{1}{n}P^{\leq a}}E'_{u}\,.
   \]
If $a \leq b$ the inclusion $\tfrac{1}{n}P^{\leq a} \subseteq \tfrac{1}{n}P^{\leq b}$ gives a homomorphism of \qc sheaves $(\Ind_{n}E')_{a} \arr (\Ind_{n}E')_{b}$; this defines a functor $\Ind_{n}E'\colon P_{\QQ}\wt \arr \catqcoh X$. If $p \in P$ then $\tfrac{1}{n}P^{\leq p+a} = p + \tfrac{1}{n}P^{\leq a}$; hence
   \[
   (\Ind_{n}E')_{p+a} = \colim_{u\in \frac{1}{n}P^{\leq a}}E'_{p+u}\,.
   \]
By taking the colimit, the isomorphisms $E'_{p+u} \simeq L_{p}\otimes E'_{u}$ induce a pseudo-period isomorphism $(\Ind_{n}E')_{p+a} \simeq L_{p}\otimes(\Ind_{n}E')_{a}$. We leave it to the reader to verify that this gives a structure of parabolic sheaf to $\Ind_{n}E'$, and defines a functor $\Ind_{n}\colon \catpar_{n}X \arr \catpar X$.

\begin{proposition}
The induction functor $\Ind_{n}\colon \catpar_{n}X \arr \catpar X$ is left adjoint to the restriction functor $\Res_{n}\colon \catpar X \arr \catpar_{n}X$.
\end{proposition}

\begin{proof}
If $a \in \tfrac{1}{n}P$ we have $(\Ind_{n}E')_{a} = E'_{a}$, so that $E' \simeq \Res_{n}\Ind_{n} E'$. This gives an isomorphism of functors $\id_{\catpar_{n}X} \simeq \Res_{n} \circ \Ind_{n}$.

On the other hand if $E$ is a parabolic sheaf in $\catpar X$, there is an obvious homomorphism $\colim_{u \in \tfrac{1}{n}P^{\leq a}}E_{u} \arr E_{a}$, which induced a natural transformation $\Ind_{n} \circ \Res_{n} \arr \id_{\catpar X}$. We leave it to the reader to check that these are respectively the counit and the unit of an adjunction.
\end{proof}

Hence $\Ind_{n}\colon \catpar_{n}X \arr \catpar X$ corresponds to the functor $\pi_{n}^{*}\colon \catqcoh \radice{n}{X} \arr \catqcoh\infroot{X}$.

Putting all this together we get the following characterization of \fp parabolic sheaves on $X$.

\begin{theorem}\label{thm:char-fp-parabolic}
Let $E\colon P_{\QQ}\wt \arr \catqcoh X$ be a parabolic sheaf on $X$. Then $E$ corresponds to a \fp sheaf on $\catqcoh \infroot{X}$ if and only if the following two conditions are satisfied.

\begin{enumeratea}

\item $E_{a}$ is \fp as a \qc sheaf on $X$ for all $a \in P_{\QQ}$.

\item There exists a positive integer $n$ such that for all $a \in P_{\QQ}$ the natural homomorphism
   \[
   \colim_{u \in \frac{1}{n}P^{\leq a}}E_{u} \arr E_{a}
   \]
is an isomorphism.\qed
\end{enumeratea}
\end{theorem}

\begin{remark}
It is an interesting open question to characterize the \fp parabolic sheaves $E$ on $X$ such that the corresponding \qc sheaf on $\infroot{X}$ is locally free. If $X$ is a regular scheme with the logarithmic structure induced by a divisor with normal crossings, then it is shown in \cite[Proposition~2]{borne} that this happens if and only if $E_{a}$ is locally free on $X$ for all $a$. In general this condition is neither necessary nor sufficient, and the exact characterization of these parabolic sheaves is probably subtle.
\end{remark}

\bibliographystyle{amsalpha}
\bibliography{irs}

\newcommand{\etalchar}[1]{$^{#1}$}
\providecommand{\bysame}{\leavevmode\hbox to3em{\hrulefill}\thinspace}
\providecommand{\MR}{\relax\ifhmode\unskip\space\fi MR }
\providecommand{\MRhref}[2]{%
  \href{http://www.ams.org/mathscinet-getitem?mr=#1}{#2}
}
\providecommand{\href}[2]{#2}
\begin{thebibliography}{ACG{\etalchar{+}}13}

\bibitem[AC14]{abramovichgw}
Dan Abramovich and Qile Chen, \emph{Stable logarithmic maps to
  {D}eligne-{F}altings pairs {II}}, Asian J. Math. \textbf{18} (2014), no.~3,
  465--488.

\bibitem[ACG{\etalchar{+}}13]{olssonetal}
D.~Abramovich, Q.~Chen, D.~Gillam, Y.~Huang, M.~Olsson, M.~Satriano, and
  S.~Sun, \emph{Logarithmic geometry and moduli}, Handbook of Moduli (G.~Farkas
  and H.~Morrison, eds.), International Press, 2013.

\bibitem[AGV08]{dan-tom-angelo2008}
Dan Abramovich, Tom Graber, and Angelo Vistoli, \emph{Gromov-{W}itten theory of
  {D}eligne-{M}umford stacks}, Amer. J. Math. \textbf{130} (2008), no.~5,
  1337--1398.

\bibitem[AOV08]{dan-olsson-vistoli1}
Dan Abramovich, Martin Olsson, and Angelo Vistoli, \emph{Tame stacks in
  positive characteristic}, Ann. Inst. Fourier (Grenoble) \textbf{58} (2008),
  no.~4, 1057--1091.

\bibitem[Bis97]{biswas}
Indranil Biswas, \emph{Parabolic bundles as orbifold bundles}, Duke Math. J.
  \textbf{88} (1997), no.~2, 305--325.

\bibitem[Bor09]{borne}
Niels Borne, \emph{Sur les repr\'{e}sentations du groupe fondamental d'une
  vari\'{e}t\'{e} priv\'{e}e d'un diviseur \`{a} croisements normaux simples},
  Indiana Univ. Math. J. \textbf{58} (2009), no.~1, 137--180.

\bibitem[BV12]{borne-vistoli1}
Niels Borne and Angelo Vistoli, \emph{Parabolic sheaves on logarithmic
  schemes}, Adv. Math. \textbf{231} (2012), no.~3-4, 1327--1363.

\bibitem[BV15]{borne-vistoli2}
\bysame, \emph{The {N}ori fundamental gerbe of a fibered category}, J.
  Algebraic Geom. \textbf{24} (2015), no.~2, 311--353.

\bibitem[Cad07]{cadman}
Charles Cadman, \emph{Using stacks to impose tangency conditions on curves},
  Amer. J. Math. \textbf{129} (2007), no.~2, 405--427.

\bibitem[Che14]{chen}
Qile Chen, \emph{Stable logarithmic maps to {D}eligne-{F}altings pairs {I}},
  Ann. of Math. (2) \textbf{180} (2014), no.~2, 455--521.

\bibitem[CSST17]{knvsroot}
David Carchedi, Sarah Scherotzke, Nicol\`o Sibilla, and Mattia Talpo,
  \emph{Kato-{N}akayama spaces, infinite root stacks and the profinite homotopy
  type of log schemes}, Geom. Topol. \textbf{21} (2017), no.~5, 3093--3158.

\bibitem[Gro61]{ega2}
A.~Grothendieck, \emph{\'{E}l\'ements de g\'eom\'etrie alg\'ebrique. {II}.
  \'{E}tude globale \'el\'ementaire de quelques classes de morphismes}, Inst.
  Hautes \'Etudes Sci. Publ. Math. (1961), no.~8, 222.

\bibitem[Gro66]{ega43}
\bysame, \emph{\'{E}l\'ements de g\'eom\'etrie alg\'ebrique. {IV}. \'{E}tude
  locale des sch\'emas et des morphismes de sch\'emas. {III}}, Inst. Hautes
  \'Etudes Sci. Publ. Math. (1966), no.~28, 255.

\bibitem[GS06]{gross-siebert}
Mark Gross and Bernd Siebert, \emph{Mirror symmetry via logarithmic
  degeneration data. {I}}, J. Differential Geom. \textbf{72} (2006), no.~2,
  169--338.

\bibitem[GS10]{gross-siebert-II}
\bysame, \emph{Mirror symmetry via logarithmic degeneration data, {II}}, J.
  Algebraic Geom. \textbf{19} (2010), no.~4, 679--780.

\bibitem[GS11]{gross-siebert-III}
\bysame, \emph{From real affine geometry to complex geometry}, Ann. of Math.
  (2) \textbf{174} (2011), no.~3, 1301--1428.

\bibitem[GS13]{gross-siebert-GW}
\bysame, \emph{Logarithmic {G}romov-{W}itten invariants}, J. Amer. Math. Soc.
  \textbf{26} (2013), no.~2, 451--510.

\bibitem[Hag03]{hagihara}
Kei Hagihara, \emph{Structure theorem of {K}ummer \'etale {$K$}-group},
  $K$-Theory \textbf{29} (2003), no.~2, 75--99.

\bibitem[IKN05]{illusie-kato-nakayama}
L.~Illusie, K.~Kato, and C.~Nakayama, \emph{Quasi-unipotent logarithmic
  {R}iemann-{H}ilbert correspondences}, J. Math. Sci. Univ. Tokyo \textbf{12}
  (2005), 1--66.

\bibitem[INT13]{illusie-nakayama-tsuji}
L.~Illusie, C.~Nakayama, and T.~Tsuji, \emph{On log flat descent}, Proc. Japan
  Acad. Ser. A Math. Sci. \textbf{89} (2013), no.~1, 1--5.

\bibitem[IS07]{iyer-simpson}
Jaya~N.N. Iyer and Carlos~T. Simpson, \emph{A relation between the parabolic
  {C}hern characters of the de {R}ham bundles}, Math. Ann. \textbf{338} (2007),
  347--383.

\bibitem[Kat89]{kato}
Kazuya Kato, \emph{Logarithmic structures of {F}ontaine-{I}llusie}, Algebraic
  analysis, geometry, and number theory ({B}altimore, {MD}, 1988), Johns
  Hopkins Univ. Press, Baltimore, MD, 1989, pp.~191--224.

\bibitem[Kat91]{kato2}
\bysame, \emph{Logarithmic structures of {F}ontaine-{I}llusie. {II}}, Preprint
  (incomplete), 1991.

\bibitem[Kat00]{katof}
Fumiharu Kato, \emph{Log smooth deformation and moduli of log smooth curves},
  Internat. J. Math. \textbf{11} (2000), no.~2, 215--232.

\bibitem[KN99]{kato-nakayama}
Kazuya Kato and Chikara Nakayama, \emph{Log {B}etti cohomology, log \'etale
  cohomology, and log de {R}ham cohomology of log schemes over {${\bf C}$}},
  Kodai Math. J. \textbf{22} (1999), no.~2, 161--186.

\bibitem[KR00]{rosenberg-kontsevich}
Maxim Kontsevich and Alexander~L. Rosenberg, \emph{Noncommutative smooth
  spaces}, The {G}elfand {M}athematical {S}eminars, 1996--1999, Gelfand Math.
  Sem., Birkh\"auser Boston, Boston, MA, 2000, pp.~85--108.

\bibitem[ML98]{maclane}
Saunders Mac~Lane, \emph{Categories for the working mathematician}, second ed.,
  Graduate Texts in Mathematics, vol.~5, Springer-Verlag, New York, 1998.

\bibitem[MO05]{matsuki-olsson}
Kenji Matsuki and Martin Olsson, \emph{{K}awamata-{V}iehweg vanishing as
  {K}odaira vanishing for stacks}, Math. Res. Lett. \textbf{12} (2005),
  no.~2-3, 207--217.

\bibitem[MS80]{metha-seshadri}
V.B. Metha and C.S. Seshadri, \emph{Moduli of vector bundles on curves with
  parabolic structures}, Math. Ann. \textbf{248} (1980), 205--239.

\bibitem[MY92]{maruyama-yokogawa}
M.~Maruyama and K.~Yokogawa, \emph{Moduli of parabolic stable sheaves}, Math.
  Ann. \textbf{293} (1992), 77--99.

\bibitem[Niz08]{niziol-k-theory}
Wies{\l}awa Nizio\l, \emph{{K}-theory of log-schemes {I}}, Documenta Math.
  \textbf{13} (2008), 505--551.

\bibitem[Ogu]{ogus}
Arthur Ogus, \emph{Lectures on logarithmic algebraic geometry}, TeXed notes,
  \url{http://math.berkeley.edu/~ogus/preprints/log_book/logbook.pdf}.

\bibitem[Ols04]{olsson3}
Martin~C. Olsson, \emph{Semistable degenerations and period spaces for
  polarized {$K3$} surfaces}, Duke Math. J. \textbf{125} (2004), no.~1,
  121--203.

\bibitem[Ols07]{olsson-log-twisted}
\bysame, \emph{({L}og) twisted curves}, Compos. Math. \textbf{143} (2007),
  no.~2, 476--494.

\bibitem[Ols08a]{olsson1}
Martin Olsson, \emph{Logarithmic interpretation of the main component in toric
  {H}ilbert schemes}, Curves and abelian varieties, Contemp. Math., vol. 465,
  Amer. Math. Soc., Providence, RI, 2008, pp.~231--252.

\bibitem[Ols08b]{olsson2}
Martin~C. Olsson, \emph{Compactifying moduli spaces for abelian varieties},
  Lecture Notes in Mathematics, vol. 1958, Springer-Verlag, Berlin, 2008.

\bibitem[SST16]{logmckay}
S.~Scherotzke, N.~Sibilla, and M.~Talpo, \emph{On a logarithmic version of the
  derived {M}c{K}ay correspondence}, Preprint, arXiv:1612.08961, 2016.

\bibitem[SSV13]{sagave-schurg-vezzosi}
Steffen Sagave, Timo Sch{\"u}rg, and Gabriele Vezzosi, \emph{Logarithmic
  derived geometry {I}}, arXiv:1307.4246 [math.AG], 2013.

\bibitem[{Sta}14]{stacks-project}
The {Stacks Project Authors}, \emph{Stacks project},
  \url{http://stacks.math.columbia.edu}, 2014.

\bibitem[Tal17a]{parabolic.moduli}
Mattia Talpo, \emph{Moduli of parabolic sheaves on a polarized logarithmic
  scheme}, Trans. Amer. Math. Soc. \textbf{369} (2017), no.~5, 3483--3545.

\bibitem[Tal17b]{real.parabolic}
\bysame, \emph{Parabolic sheaves with real weights as sheaves on the
  {K}ato-{N}akayama space}, Preprint, arXiv:1703.04777, 2017.

\bibitem[Toe99]{toen-rr}
B.~Toen, \emph{Th\'eor\`emes de {R}iemann-{R}och pour les champs de
  {D}eligne-{M}umford}, $K$-Theory \textbf{18} (1999), no.~1, 33--76.

\bibitem[TV17]{TVnew}
Mattia Talpo and Angelo Vistoli, \emph{The {K}ato-{N}akayama space as a
  transcendental root stack}, arXiv:1611.04041, pulished online in
  \href{http://academic.oup.com//imrn/article/doi/10.1093/imrn/rnx079/3770480/The-KatoNakayama-Space-as-a-Transcendental-Root?guestAccessKey=96fed519-1ad1-4a30-ab12-343c93a119ce}{Int.
  Math. Res. Not.}, 2017.

\bibitem[Vis05]{vistoli-descent}
Angelo Vistoli, \emph{Grothendieck topologies, fibered categories and descent
  theory}, Fundamental algebraic geometry, Math. Surveys Monogr., vol. 123,
  Amer. Math. Soc., Providence, RI, 2005, pp.~1--104.

\bibitem[VV02]{vezzosi-vistoli02}
Gabriele Vezzosi and Angelo Vistoli, \emph{Higher algebraic {K}-theory of group
  actions with finite stabilizers}, Duke Math. J. \textbf{113} (2002), no.~1,
  1--55.

\end{thebibliography}

\end{document}